\documentclass[11pt,twoside]{article}

\usepackage{subcaption}
\usepackage{rotating}
\usepackage{old-arrows}
\usepackage{enumitem}
\usepackage{amsmath}
\usepackage[utf8]{inputenc} 
\usepackage[T1]{fontenc}    
\usepackage{booktabs}       
\usepackage{amsfonts}       
\usepackage{nicefrac}       
\usepackage{microtype}      

\usepackage{epsf}
\usepackage{epsfig}
\usepackage{fancyhdr}
\usepackage{graphics}
\usepackage{graphicx}
\usepackage{psfrag}
\usepackage{fullpage}
\usepackage{pdfpages}

\newcommand*{\backrefalt}[4]{%
    \ifcase #1 \footnotesize{(Not cited.)}%
    \or        \footnotesize{(Cited on page~#2.)}%
    \else      \footnotesize{(Cited on pages~#2.)}%
    \fi}
\usepackage{amsthm}
\usepackage{amsmath}
\usepackage{amssymb,bbm}
\usepackage{caption}

\usepackage{xcolor}
\usepackage{fancybox}

\usepackage{algorithm}
\usepackage{algpseudocode}
\usepackage{multicol}
\usepackage{caption}

\usepackage{microtype}
\usepackage{graphicx}
\usepackage{booktabs} 
\usepackage{url}
\usepackage[colorlinks,linkcolor=magenta,citecolor=blue, pagebackref=true,backref=true]{hyperref}
\usepackage{changes}

\makeatletter
\newcommand{\xRightarrow}[2][]{\ext@arrow 0359\Rightarrowfill@{#1}{#2}}
\makeatother

\makeatletter
\newcommand{\xLeftarrow}[2][]{\ext@arrow 0359\Leftarrowfill@{#1}{#2}}
\makeatother

\newtheorem{theorem}{Theorem}
\numberwithin{theorem}{section}
\newtheorem{proposition}{Proposition}
\numberwithin{proposition}{section}
\newtheorem{lemma}{Lemma}
\numberwithin{lemma}{section}

\numberwithin{corollary}{section}

\numberwithin{figure}{section}

\theoremstyle{definition}
\newtheorem{remark}{Remark}
\numberwithin{remark}{section}

\numberwithin{example}{section}
\newtheorem{assumption}{Assumption}
\numberwithin{assumption}{section}
\newtheorem{definition}{Definition}
\numberwithin{definition}{section}

\numberwithin{condition}{section}

\numberwithin{problem}{section}

\numberwithin{exercise}{section}

\newtheorem{innercustomgeneric}{\customgenericname}
\providecommand{\customgenericname}{}
\newcommand{\newcustomtheorem}[2]{%
  \newenvironment{#1}[1]
  {%
   \renewcommand\customgenericname{#2}%
   \renewcommand\theinnercustomgeneric{##1}%
   \innercustomgeneric
  }
  {\endinnercustomgeneric}
}

\DeclareMathOperator*{\argmax}{arg\,max}

\newcustomtheorem{customthm}{Theorem}
\newcustomtheorem{customlemma}{Lemma}
\newcustomtheorem{customprop}{Proposition}

\newcommand{\Ncal}{\mathrm{Normal}}

\long\def\comment#1{}

\newcommand{\Ecal}{\ensuremath{\mathcal{E}}}
\newcommand{\Dir}{\mathrm{Dir}}

\newcommand{\Ocal}{\ensuremath{\mathcal{O}}}
\newcommand{\Ebb}{\ensuremath{\mathbb{E}}}
\newcommand{\Gcal}{\ensuremath{\mathcal{G}}}
\newcommand{\Bcal}{\ensuremath{\mathcal{B}}}
\newcommand{\Xcal}{\ensuremath{\mathcal{X}}}
\newcommand{\Fcal}{\ensuremath{\mathcal{F}}}
\newcommand{\Ycal}{\ensuremath{\mathcal{Y}}}
\newcommand{\Pcal}{\ensuremath{\mathcal{P}}}
\newcommand{\Zbb}{\ensuremath{\mathbb{Z}}}
\newcommand{\Rbb}{\ensuremath{\mathbb{R}}}
\newcommand{\Nbb}{\ensuremath{\mathbb{N}}}
\newcommand{\Pbb}{\ensuremath{\mathbb{P}}}
\newcommand{\diam}{\mathrm{diam}}

\newcommand{\Poi}{\mathrm{Poi}}
\newcommand{\Bin}{\mathrm{Bin}}
\newcommand{\NB}{\mathrm{NB}}
\newcommand{\PX}{\ensuremath{\mathbb{P}_X}}

\newcommand{\Kup}{\ensuremath{\bar{K}}}
\newcommand{\Kvar}{\ensuremath{K}}

\newcommand{\parenth}[1]{\left( #1 \right)}

\newcommand{\abss}[1]{\left| #1 \right |}

\newcommand{\norm}[1]{\left\|#1\right\|}
\allowdisplaybreaks

\begin{document}
\begin{center}


{\bf{\LARGE{Strong identifiability and parameter learning \\
\vspace{.1in}
in regression with heterogeneous response}}}
  
\vspace*{.2in}
 {\large{
 \begin{tabular}{cccccc}
 Dat Do$^{\diamond}$ &  Linh Do$^{\ddagger}$ & XuanLong Nguyen$^{\diamond}$\\
 \end{tabular}
}}

\vspace*{.1in}

\begin{tabular}{c}
University of Michigan, Ann Arbor$^{\diamond}$; Tulane University$^{\ddagger}$
\end{tabular}

\vspace{.1in}

\today
\end{center}

\begin{abstract}
Mixtures of regression are useful for regression learning with respect to an uncertain and heterogeneous response variable of interest. In addition to being a rich predictive model for the response given some covariates, the model parameters provide meaningful information about the heterogeneity in the data population, which is represented by the conditional distributions for the response given the covariates associated with a number of distinct but latent subpopulations. In this paper, we investigate conditions of strong identifiability, MLE rates of convergence for the conditional density and model parameters, and the Bayesian posterior contraction behavior arising in finite mixture of regression models, under exact-fitted and over-fitted settings and when the number of components is unknown. This theory is applicable to common choices of link functions and families of conditional distributions employed by practitioners. We provide simulation studies and data illustrations, which shed some light on the parameter learning behavior found in several popular regression mixture models reported in the literature.
\end{abstract}


\section{Introduction}

Regression is often associated with the task of curve fitting --- given data samples for pairs of random variables $(X, Y)$, find a function $y=F(x)$ that captures the relationship between $X$ and $Y$ as well as possible. As the underlying population for the $(X, Y)$ pairs becomes increasingly complex, much effort has been devoted to learning more complex models for the regression function $F$. 
In many data domains, however, due to the heterogeneity of the behavior of the response variable $Y$ with respect to covariate $X$, no single function $F$ can fit the data pairs well, no matter how complex $F$ is. Many authors noticed this challenge and adopted a mixture modeling framework into the regression problem, going back to \cite{GOLDFELD19733,spath1979algorithm}, and continuing with more recent applications, e.g., \cite{bermudez2020modelling,ding2006using,mufudza2016poisson,park2010bias, PARK2009683}. 

To capture the uncertain and highly heterogeneous behavior of response variable $Y$ given covariate $X$, one needs more than one single regression model. Suppose that there are $k$ different regression behaviors, one can represent the \emph{conditional} distribution of $Y$ given $X$ by a mixture of $k$ conditional density functions associated with $k$ underlying (latent) subpopulations. One can draw from modeling tools of conditional densities such as generalized linear models or more complex components to increase model fitness for the regression task \cite{jacobs1991adaptive, Hjort-etal-10}. 
%
%
%
Making inferences in regression mixtures can be achieved in a frequentist framework (e.g., maximum conditional likelihood estimation (MLE) \cite{desarbo1988maximum}), or a Bayesian framework \cite{hurn2003estimating}. 
In addition to enhanced predictability for the response variable given the covariate, a key benefit of regression mixture models is that the model parameters may be used to explicate the relationship between these variables more accurately and meaningfully.  

Despite the aforementioned long history of applications, a satisfactory level of understanding of several key issues concerning model parameters' identifiability and a large sample theory of regression mixture models remains far from being complete. This is perhaps due to the somewhat unusual position where regression mixture model based methods sit --- like any regression problem one is interested in prediction performance, but unlike the traditional viewpoint of a single curve-fitting task one must come to terms with the multi-modality of the response variable due to the underlying data population's latent heterogeneity. Thus, one must also be interested in the quality of parameter estimates representing such heterogeneity.
There is a slowly growing theoretical literature, but most existing works are limited to the questions of consistency of estimation for the mixture of generalized linear models with some specific classes of conditional densities and link functions, or simulation-based methods \cite{hennig2000identifiablity,jiang1999identifiability,jaki2019effects, wang2014note, grun2008identifiability, viele2002modeling}. In particular, \cite{hennig2000identifiablity} investigates the identifiability of the mixture of Gaussian regression models with linear link functions. \cite{jiang1999identifiability} generalizes the results for the exponential families. \cite{wang2014note} further extends the identifiability results to more general link functions, but no analysis of parameter estimation. \cite{viele2002modeling} shows the consistency for density learning of this model under the Bayesian setting. On parameter estimation behavior, more recently \cite{khalili2007variable} proposed a penalized MLE method for model selection for the class of identifiable mixture of regression models with linear link functions and established rates of parameter estimation. \cite{ho2019convergence} investigated the parameter estimation behavior for the Gaussian mixture of regression models.

In this paper, we study parameter identifiability, parameter estimation behavior, and prediction performance arising from the finite mixture of regression models. We work with general conditional density kernels and link functions, investigate both an MLE approach and a Bayesian approach for estimation. Consider a regression mixture model in the following form:
\begin{equation}\label{eq:true_model}
    f_{G_0}(y|x) = \sum_{j=1}^{k_0} p_j^0 f(y|h_1(x, \theta_{1j}^0), h_2(x, \theta_{2j}^0)), 
\end{equation}
where $x \in \mathcal{X} \subset \mathbb{R}^p$ is a vector including the explanatory variables, $y \in \mathcal{Y}$ is the response variable. The conditional density function $f_{G_0}(y|x)$ take the mixture form, where the discrete probability measure $G_0 = \sum_{j=1}^{k_0} p_j^0 \delta_{(\theta_{1j}^0,\theta_{2j}^0)}$ encapsulates all unknown parameters in the model, with $(p_j^0)_{j=1}^{k_0}$ being the mixing proportion, and $(\theta_{1j}^0)_{j=1}^{k_0}$ and $(\theta_{2j}^0)_{j=1}^{k_0}$ being parameters in a compact subspace $\Theta_1$ of $\mathbb{R}^{d_1}$ and $\Theta_2$ of $\mathbb{R}^{d_2}$, respectively. We call $G_0$ the latent mixing measure associated with the regression mixture model.
The link functions $h_1:\Xcal \times \Theta_{1} \to H_1$ and $h_2: \Xcal \times \Theta_2 \to H_2$ are known, where $H_1, H_2$ are compact subsets of $\mathbb{R}$. The family of densities $\{f(y|\mu,\phi): \mu\in H_1,  \phi\in H_2\}$ is given, where all of them are dominated by a common distribution $\nu$ on $\mathcal{Y}$ which can be either a counting or continuous measure. 
In many applications, the family $f$ is a dispersion exponential family distribution with parameter $\mu = h_1(x, \theta_1)$ is modeled as the mean, and $\phi = h_2(x, \theta_2)$ is modeled as the variance of $f(y|h_1(x, \theta_1), h_2(x, \theta_2))$ so that the mixture of regression models can capture the average trends and dispersion of subpopulations in the data. We are interested assessing the quality of the conditional density estimates, as well as that of parameters $(p_j^0)_{j=1}^{k_0}$, $(\theta_{1j}^0)_{j=1}^{k_0}$, and $(\theta_{2j}^0)_{j=1}^{k_0}$ from i.i.d. samples $(x_i, y_i)_{i=1}^{n}$, where distribution of $y_i$ given $x_i$ is given in the model~\eqref{eq:true_model} and $x_i$ follows some (unknown) marginal distribution $\Pbb_X$ on $\mathcal{X}$. 

Our parameter estimation theory inherits from and generalizes several recent developments in the finite mixture models literature. \cite{chen1995optimal} initiated the theoretical investigation of parameter estimation in a univariate finite mixture model by introducing a notion of \emph{strong} identifiability. \cite{nguyen2013convergence} developed a theory for both finite and infinite mixture models in a multivariate setting using optimal transport distances. \cite{Ho-Nguyen-EJS-16} studied convergence rates in various families of vector-matrix distributions. 
A central concept in these papers is the notion of strong identifiability of (unconditional) mixtures of density functions.  This is a condition on a parametrized family of density function $f$ of $y$, as there is no covariate $x$ here. At a high level, it requires that the family of function $f$, along with their partial derivatives with respect to the parameters up to a certain order, are linearly independent. Once this condition is satisfied, one can establish a lower bound on the distance between mixture distributions in terms of the optimal transport distances between the corresponding latent mixing parameters. Such a bound is called an \emph{inverse bound}, which plays a crucial role in deriving the rates of parameter estimates. 

With regression mixture modeling, we move from the unconditional mixtures described above to conditional mixture models. Thus, there are several fundamental distinctions. 
First, one works with the family of conditional density functions in the form $f(y|h(x,\theta))$, which involves both the conditional density kernel $f$ and the link function $h$. A strong identifiability condition for conditional distributions that we develop will inevitably involve both variables $x$ and $y$. The focus of inference is on the conditional distribution of $Y$ given covariate $X$, while the marginal distribution of $X$ is assumed unknown and of little interest. Accordingly, the identifiability condition must ideally require as little information from the marginal distribution of the covariate as possible. Moreover, given that the identifiability condition holds, the inverse bound that we establish will be a lower bound on the \emph{expected} distance of the conditional densities, where the expectation is taken with respect to the marginal distribution of the covariate. This is also crucial because we will obtain rates of conditional density estimation in terms of the mentioned expected distance and use the inverse bound to derive the rates of convergence for the corresponding mixing parameters of interest.

Another interesting feature that distinguishes conditional mixtures from unconditional mixtures is that the former tends to satisfy strong identifiability conditions more easily than the latter. This is because of the role the covariate $x$ plays in providing more constraints that prevent the violation of the linear independence condition. For instance, it is trivial that an unconditional mixture of Bernoulli distributions is not identifiable, but it will be shown (not so easily) that the mixture of conditional distributions using the Bernoulli kernel is not only identifiable but also strongly identifiable. There are situations where there is a lack of strong identifiability, such as in the case of negative binomial regression mixtures, a model extensively employed in practice (e.g., see \cite{park2010bias, PARK2009683}), but we shall show that such situations occur precisely only in a Lebesgue measure zero subset of the parameter space.

To summarize, there are several contributions made in this paper. First, we develop a rigorous notion of strong identifiability for general regression mixture models. We provide a characterization of such a notion in terms of simple conditions on the conditional density kernel $f$ and link function $h$ and show that they are satisfied by a broad range of density kernels and link functions often employed in practice. Second, we study several examples of regression mixtures when strong identifiability is violated and investigate the consequences. Third, we establish learning rates for regression mixtures given strong identifiability, under both Bayesian estimation and MLE frameworks. We consider three different learning scenarios: when the number of mixture components $k_0$ is known (i.e., exact-fitted setting), when only an upper bound is known (i.e., overfitted setting), and when even such an upper bound is unknown. Finally,
 we conduct a series of simulation studies to support the theory and discuss the connections with empirical findings in the regression mixture literature  \cite{jaki2019effects, park2010bias, sarstedt2008model}. 
 
The rest of this paper is organized as follows. Section~\ref{sec:preliminaries} provides preliminaries on the mixture of regression models.  
In Section~\ref{sec:identifiable-inverse-bound}, we present notions of strong identifiability and associated characterization for regression mixtures, followed by a set of inverse bounds. Building upon this strong identifiability theory, in Section~\ref{sec:stats-efficiency}, we establish the rates of conditional density estimation and parameter estimation. In Section~\ref{sec:experiments}, we carry out simulation studies and data illustrations to support our theory and discuss the empirical findings in the literature. Finally, Section~\ref{sec:conclusion} discusses future directions.
All proofs are deferred to the Supplementary material.

\paragraph{Notation} Given a mixing measure $G = \sum_{j=1}^{k}p_j \delta_{(\theta_{1j},\theta_{2j})}$, the mixture of regression model with respect to $G$ is denoted by
{\fontsize{8}{10}\selectfont $f_{G}(y|x) = \sum_{j=1}^{k} p_j f(y|h_1(x, \theta_{1j}), h_2(x, \theta_{2j}))$.}
The joint distribution of $(x, y)$ is 
$d\Pbb_{G}(x, y) = d\PX(x) \times f_G(y|x)d\nu(y)$,
where $\PX$ is an unknown distribution of covariate $X$. $\Ebb_X$ denotes  the expectation w.r.t. $\Pbb_X$. We write $f_j(y|x) = f(y|h_1(x, \theta_{1j}), h_2(x, \theta_{2j}))$ for short, for $j=1, \dots, k$, if there is no confusion. 
Denote $\Theta = \Theta_1 \times \Theta_2$, $H = H_1\times H_2$. Let $\Ecal_{k}(\Theta)$ be the space of mixing measures with exactly $k$ atoms in $\Theta$, and $\Ocal_{k}(\Theta) = \cup_{\kappa=1}^{k} \Ecal_{\kappa}(\Theta)$ the space of mixing measures with no more than $k$ atoms in $\Theta$. If there is no confusion, we write $\Ecal_k(\Theta)$ and $\Ocal_K(
\Theta)$ as $\Ecal_k$ and $\Ocal_K$ for short. 
For two sequence $(a_n)_{n=1}^{\infty}$ and $(b_n)_{n=1}^{\infty}$, we write $a_n \preccurlyeq b_n$ if there is a constant $C$ such that $a_n \leq C b_n$ for all $n$. We also write $a_n \succcurlyeq b_n$ if $b_n \preccurlyeq a_n$, and $a_n \asymp b_n$ if we have both $a_n\preccurlyeq b_n$ and $a_n \succcurlyeq b_n$. The multiplicative constants in those inequalities will be specified in the main results for clarity.    
We use $d_H, d_{TV}$, and $K$ for the Hellinger distance, total variation distance, and Kullback-Leibler (KL) divergence between densities, respectively.

\section{Preliminaries}\label{sec:preliminaries}
\paragraph{Regression mixture models} A mixture of regression model may be applied with many different family distributions and link functions to fit a large range of data distributions. For example, when the response variable $y$ is continuous, we can choose the family of (conditional) density to be normal $\{\Ncal(y|\mu, \phi): \mu \in \mathbb{R}, \phi\in \mathbb{R}_{+}\}$, and parametrize $\mu_j$ and $\phi_j$ via two link functions $\mu_j = h_1(x, \theta_{1j}), \phi_j = h_2(x, \theta_{2j})$, for $j = 1, \dots, k$. These functions can be represented by polynomials or trigonometric polynomials with variable $x$ and coefficients $\theta_{1j},\theta_{2j}$. 
Alternatively, when $y$ is a counting variable, one can use the Binomial distribution $\{\Bin(y|N, q): q \in [0, 1]\}$ if $y$ is bounded 
and the Poisson distribution $\{\Poi(y|\mu): \mu\in \mathbb{R}_{+}\}$ otherwise. 
If one wishes to take into account the dispersion of $y$, Negative Binomial distribution $\{\NB(y|\mu, \phi): \mu, \phi\in \mathbb{R}_{+}\}$, where $\NB(y|\mu, \phi) = \dfrac{\Gamma(\phi + y)}{\Gamma(\phi) y!} \left(\dfrac{\mu}{\phi + \mu}\right)^{y} $ $\left(\dfrac{\phi}{\phi + \mu}\right)^{\phi}$ may be used. If the values of $\mu$ or $\phi$ need to be non-negative or belong to a compact set, one may apply functions such as exponential functions or the sigmoid (inverse logit) function compositing with a polynomial or trigonometric polynomial parametrized by $\theta_{1}, \theta_{2}$. The general theory to be presented will be applicable to all these models, and others.

\paragraph{Wasserstein distances} As discussed in the Introduction, all parameters in the mixture model for the conditional distribution $f_G(y|x)$ of the response $y$ given covariate $x$ are encapsulated by the latent mixing measure $G = \sum_{j=1}^{k}p_j \delta_{(\theta_{1j},\theta_{2j})}$. In order to characterize identifiability and learning rates of parameter learning, one needs a suitable metric for the mixing measure $G$. 
Wasserstein distances have become a useful tool to quantify the convergence of latent mixing measures in mixture models \cite{nguyen2013convergence}. 
Given two discrete measures $G = \sum_{j=1}^{k} p_j \delta_{\theta_{j}}$ and $G' = \sum_{j=1}^{k'} p_j' \delta_{\theta'_{j}}$ on a normed space $\Theta$ endowed with a norm $\|\cdot \|$, the $W_r$ Wasserstein metric, in which $r\geq 1$, is defined as:
\begin{equation*}
    W_{r}(G, G') = \left[\inf_{q} \sum_{i, j=1}^{k, k'} q_{ij} \norm{\theta_{i} - \theta'_{j}}^{r}\right]^{1/r},
\end{equation*}
where the infimum is taken over all joint distribution on $[1, \dots, k]\times [1, \dots, k']$ such that $\sum_{i=1}^{k} q_{ij} = p_j', \sum_{j=1}^{k'} q_{ij} = p_i$. Note that for $G_0= \sum_{j=1}^{k_0} p^0_j \delta_{\theta^{0}_{j}}\in \Ecal_{k_0}$, if $G= \sum_{j=1}^{k} p_j \delta_{\theta_{j}}$ varies on $\Ocal_{k}$ such that $W_r(G, G_0)\to 0$ and $\Theta$ is compact, then 
\begin{equation}
    W_r^r(G, G_0) \asymp \sum_{i=1}^{k_0} \left|\sum_{\theta_j\in V_i} p_{j} - p_{i}^{0} \right| + \sum_{i=1}^{k_0}\sum_{\theta_j\in V_i} p_{j} \norm{\theta_{j} - \theta_{i}^{0}}^{r},
\end{equation}
where $V_i =\{\theta: \norm{\theta - \theta_{i}^0}\leq \norm{\theta - \theta_{i'}^0} \forall i'\neq i\}$ is the Voronoi cell of $\theta_i^0$ in $\Theta$ (see, e.g., \cite{ho2019singularity}). Hence,
for every atom of $G_0$, there is a subset of atoms of $G$ converging to it at the same rate as $W_r(G, G_0) \to 0$. Therefore, the convergence in a Wasserstein metric $W_r$ implies the convergence of parameters in mixture models. In this paper, unless noted otherwise the space $\Theta= \Theta_1\times \Theta_2$ is chosen to be a compact subset of $\mathbb{R}^{d_1+d_2}$ and $\norm{\cdot}$ is the usual $\ell^2$ distance.

\paragraph{Mixtures of conditional densities.}
In a regression mixture model, a focus of inference will be on the conditional density $f_G(y|x)$, while there will be as little assumption as possible on the marginal distribution of covariate $X$. It is clear from the representation of $f_G(y|x)$ that the identifiability and parameter learning behavior of the regression problem will repose upon suitable conditions specified by $f$, $h$ and the unknown parameter $G$.
The analysis of conditional density estimation requires us to control how large the conditional density family $\{f(y|h_1(x, \theta_1), h_2(x, \theta_2)): \theta_1\in \Theta_1, \theta_2\in \Theta_2\}$ is. This can be accomplished by assuming Lipschitz conditions on $f$, $h_1$ and $h_2$. 
In particular, we say that $f$ is uniformly Lipschitz if there exists $c_{f} > 0$ such that for all $\mu, \mu'\in H_1, \phi, \phi' \in H_2$:
    \begin{equation}\label{eq:uniform-Lipschitz-f}
        \sup_{y\in \Ycal} |f(y|\mu, \phi) - f(y|\mu'
        , \phi')| \leq c_{f} (|\mu-\mu'| + |\phi - \phi'|).
    \end{equation}
The link functions $h_1$ and $h_2$ are called uniformly Lipschitz if there are $c_1, c_2 > 0$ such that
    for all $\theta_1, \theta_1' \in \Theta_1, \theta_2, \theta_2'\in \Theta_2$:
    {\fontsize{9.5}{10}\selectfont
    \begin{equation}\label{eq:uniform-Lipschitz-h}
        \sup_{x\in \Xcal} |h_1(x, \theta_1) - h_1(x, \theta_1')| \leq c_1 \norm{\theta_1 - \theta_1'}, \,\sup_{x\in \Xcal} |h_2(x, \theta_2) - h_2(x, \theta_2')| \leq c_2 \norm{\theta_2 - \theta_2'}.
    \end{equation}}

In a regression problem, one is interested in prediction error guarantee in addition to assessing the quality of parameter estimates. For a standard (single component) regression model, we often model $f(y|x) = f(y|h_1(x, \theta_1), h_2(x, \theta_2))$, where $h_1(x, \theta)$ is the mean parameter, i.e., $\Ebb[Y|X=x] = h_1(x, \theta)$. After estimating $\hat{\theta}_1$ from the data, the prediction error is customarily taken to be the mean square error $\Ebb_X (h_1(X, \theta_1^0) - h_1(X, \hat{\theta}_1))^2$, where $\theta_1^0$ is the true parameter. For a regression mixture, let the true latent mixing measure be $\sum_{j=1}^{k_0} p_j^0 \delta_{(\theta_{1j}^0, \theta_{2j}^0)}$ for which an estimate is denoted by $\sum_{j=1}^{k} \hat{p}_j \delta_{(\hat{\theta}_{1j}, \hat{\theta}_{2j})}$. In this setting, due to the heterogeneous nature of the response, the predicted value for $y$ at any $x$ may be taken by the quantity
$\sum_{j=1}^{k} \hat{p}_j \delta_{h_1(x, \hat{\theta}_{1j})}$,
or its mean $\sum_{j=1}^{k} \hat{p}_j h_1(x, \hat{\theta}_{1j})$.
As a result, the prediction error for the mean estimate can be written as 
$\Ebb_X W_2^2\left(\sum_{j=1}^{k_0} {p}^{0}_j \delta_{h_1(X, {\theta}^{0}_j)}, \sum_{j=1}^{k} \hat{p}_j \delta_{h_1(X, \hat{\theta}_{1j})}\right)$. 
If one is interested in describing the prediction error in terms of both the mean trend and dispersion, one can use $\Ebb_X W_2^2\left(\sum_{j=1}^{k_0} {p}^{0}_j \delta_{(h_1(X, {\theta}^{0}_{1j}), h_2(X, {\theta}^{0}_{2j}))},  \sum_{j=1}^{k} \hat{p}_j \delta_{(h_1(X, \hat{\theta}_{1j}), h_2(X, \hat{\theta}_{2j}))}\right).$

\paragraph{Key inequalities}
The following basic inequality controls the \emph{expected} total variation distance between conditional densities by a Wasserstein distance between the corresponding parameters:
\begin{lemma}\label{lem:basic-bounds}
Assume conditions ~\eqref{eq:uniform-Lipschitz-f} and \eqref{eq:uniform-Lipschitz-h} hold. Then
for every $G\in \Ocal_{K}(\Theta)$ and $K\geq 1$, we have
\begin{equation}\label{eq:intro-basic-bounds}
    \Ebb_{X} [d_{TV}(f_{G}(\cdot|X), f_{G_0}(\cdot|X))] \preccurlyeq W_1(G, G_0),
\end{equation}
where the multiplicative constant in this inequality only depends on $c_{f}, c_1$, and $c_2$.
\end{lemma}
The inequality established in the above lemma quantifies the impact of parameter estimation on the quality of conditional density estimation: if $G$ is well estimated, then so is the conditional distribution represented by the conditional densities $f_G(Y|X)$. In order to quantify the identifiability and convergence of the unknown parameter $G$, we will need to establish inequalities of the following type: 
\begin{equation}\label{eq:intro-inverse-bounds}
    \Ebb_{X} [d_{TV}(f_{G}(\cdot|X), f_{G_0}(\cdot|X))] \succcurlyeq W_r^r(G, G_0),
\end{equation}
for all $G$ in some space of latent mixing measures, and $r$ depends on that space. Following \cite{nguyen2013convergence,wei2020convergence,do2022beyond}, we refer to this as \emph{inverse bounds}, because in our setting, they allow us to lower bound the distance between conditional probability models ($f_G$ and $f_{G_0}$) by the distance between the parameters of inferential interest ($G$ and $G_0$). 
Unlike prior works, our inverse bounds control the \emph{expected} total variational distance under the marginal distribution of the covariate $X$. A simple observation is that these inverse bounds are quantitative versions of the classical identifiability condition \cite{teicher1963identifiability} for the regression problem, because if $f_{G} = f_{G_0}$ for a.e. $x, y$, then the bound~\eqref{eq:intro-inverse-bounds} entails that $G = G_0$. 
Moreover, the inverse bounds play an important role in establishing the convergence rate for parameter estimation. They allow us to translate convergence rates for density estimation (left-hand side of Eq.~\eqref{eq:intro-inverse-bounds}) into that of parameter estimation (right-hand side of Eq.~\eqref{eq:intro-inverse-bounds}). The technique to prove inverse bounds is to rely on a notion of strong identifiability to be developed for regression mixture models in the following section.

\section{Strong identifiability and inverse bounds}\label{sec:identifiable-inverse-bound}


\subsection{Conditions of strong identifiability}
Identifiability and strong identifiability conditions play important roles in the theoretical analysis of mixture models \cite{teicher1963identifiability, chen1995optimal, Ho-Nguyen-EJS-16}.
They provide a finer characterization of the non-singularity of the Fisher information for mixtures of distributions \cite{ho2019singularity}. In plain words, these conditions require that the kernel density function of interest and its derivatives up to a certain order with respect to all relevant parameters be linearly independent. For the mixture of regression model~\eqref{eq:true_model}, the kernel density function is that of the conditional probability of variable $y$ given covariate $x$. The following definition is our formulation of strong identifiability for the conditional density functions:
\begin{definition} \label{def:strong_identifiability_hierarchy}
The family of conditional densities {\fontsize{9.5}{10}\selectfont $\{f(y|h_1(x, \theta_1),h_2(x, \theta_2)):$ 
$ \theta_1\in \Theta_1, \theta_2\in \Theta_2 \}$ }(or in short, $f(\cdot|h_1, h_2)$) is 
\underline{identifiable in order $r$}, where $r=1$ (resp., $r=2$) with
\underline{complexity level $k$}, if $f(y|h_1(x, \theta_1),h_2(x, \theta_2))$ is differentiable up to order $r$ with respect to $(\theta_1,\theta_2)$, 
and (A1.) (resp., (A2.)) holds. 
\begin{itemize}
\item[(A1.)] (First order identifiable) For any given $k$ distinct elements
$(\theta_{11},\theta_{21}), \ldots, $ $(\theta_{1k},\theta_{2k}) \in \Theta_1\times \Theta_2$, if there exist $\alpha_j\in \Rbb, \beta_j\in \Rbb^{d_1}, \gamma_j\in \Rbb^{d_2}$ as $j = 1,\dots, k$ such that for almost all $x, y$ (w.r.t. $\PX \times \nu$)
\begin{eqnarray}
\sum_{j=1}^{k} \alpha_j f_j(y|x) + \beta_j^{\top} \dfrac{\partial}{\partial \theta_1} f_j(y|x)+  \gamma_j^{\top} \dfrac{\partial}{\partial \theta_2} f_j(y|x) = 0, \nonumber
\end{eqnarray}
then $\alpha_j=0, \beta_j=0\in \Rbb^{d_1}, \gamma_j = 0 \in \Rbb^{d_2}$ for $j = 1, \dots, k$;
\item[(A2.)] (Second order identifiable) For any given $k$ distinct elements
$(\theta_{11},\theta_{21}), \ldots, $ $(\theta_{1k},\theta_{2k}) \in \Theta_1\times \Theta_2$ and $s_1, \dots, s_k\geq 1$, if there exist $\alpha_j\in \Rbb, \beta_j\in \Rbb^{d_1}, \gamma_j\in \Rbb^{d_2}$, and $\rho_{jt} \in \Rbb^{d_1}, \nu_{jt}\in \Rbb^{d_2}$ as $j = 1,\dots, k, t = 1, \dots, s$ such that for almost all $x, y$ (w.r.t. $\PX \times \nu$)
\begin{eqnarray}
\sum_{j=1}^{k} \alpha_j f_j(y|x) + \beta_j^{\top} \dfrac{\partial}{\partial \theta_1} f_j(y|x) + \gamma_j^{\top} \dfrac{\partial}{\partial \theta_2} f_j(y|x) + \sum_{t=1}^{s_j} \left(\rho_{jt}^{\top} \dfrac{\partial }{\partial \theta_1^2} f_j(y|x) \rho_{jt}\right) \nonumber\\
   + \sum_{t=1}^{s_j} \left(\nu_{jt}^{\top} \dfrac{\partial }{\partial \theta_2^2} f_j(y|x) \nu_{jt}\right) + \sum_{t=1}^{s_j} \left(\rho_{jt}^{\top} \dfrac{\partial }{\partial \theta_1 \partial \theta_2} f_j(y|x) \nu_{jt}\right)  = 0\nonumber,
\end{eqnarray}
then $\alpha_j=0, \beta_j = \rho_{jt} = 0\in \Rbb^{d_1}, \gamma_j = \nu_{jt} = 0 \in \Rbb^{d_2}$ for $t = 1, \dots, s_j, j = 1, \dots, k$.
\end{itemize}
\end{definition}

When we speak of strong identifiability without specifying the complexity level, it should be understood that the condition is satisfied for any complexity level $k\geq 1$. These strong identifiability conditions for conditional density functions are useful in deriving rates of convergence for the regression mixture model's parameters even when the associated Fisher information matrices are singular, e.g., when the model has redundant parameters. Indeed, when showing the convergence rate of an estimator $G$ to the true mixing measure $G_0$ in the over-fitted setting, there might exist several redundant atoms of $G$ converge to a common atom of $G_0$. The customary technique of applying the first-order Taylor expansion around $f_{G_0}(\cdot | X)$ may fail because the coefficients of these redundant components can be combined and canceled out. Instead, one needs to perform a Taylor expansion up to the second order around $f_{G_0}(\cdot | X)$, necessitating the second-order identifiability condition developed here.
It will be shown in the sequel that the strong identifiability conditions hold for most popular mixtures of regression models. There are notable exceptions which shall be discussed separately. For instance, a mixture of binomial regression models generally satisfies strong identifiability only up to a finite complexity level.

Since our model~\eqref{eq:true_model} is hierarchical with two levels of parameters:
{\fontsize{9.5}{10}\selectfont
\begin{equation}
    G_0 = \sum_{j=1}^{k_0} p_j^{0} \delta_{(\theta_{1j}, \theta_{2j})} 
    \, \mapsto \, 
    \sum_{j=1}^{k_0} p_j^{0} \delta_{(h_1(x, \theta_{1j}), h_2(x, \theta_{2j}))} 
    \,\mapsto \,
    \sum_{j=1}^{k_0} p_j^{0} f(y|h(x, \theta_{1j}), h(x, \theta_{2j})),
\end{equation}}
\noindent it is difficult to directly verify conditions (A1.) and (A2.). We will show in the following that they can be deduced from the identifiability conditions of a family of (unconditional) distribution $\{f(y|\mu,\phi):\mu, \phi\}$ and family of functions $(h_1, h_2)$. 
Recall from \cite{Ho-Nguyen-EJS-16, nguyen2013convergence}:   
\begin{definition} \label{def:strong_identifiability_f}
The family of (unconditional) distributions $\left\{f(y|\mu,\phi) : (\mu,\phi) \in H \right\}$ (or in short, $f$) is
\underline{identifiable} \underline{in order $r$ with complexity level $k$}, for some $r, k \geq 0$, if $f(y|\mu,\phi)$ is differentiable     up to order $r$ in $(\mu,\phi)$
and the following holds:
\begin{itemize}
\item[(A3.)] For any given $k$ distinct elements
$(\mu_{1},\phi_1), \ldots, (\mu_{k},\phi_k) \in H$, if for each pair of $n = (n_1, n_2)$, where $n_1\geq n_2\geq 0, n_1 + n_2 \leq r$, we have $\alpha_{n}^{(j)}\in \Rbb$ such that 
\begin{eqnarray}
\sum_{l=0}^{r} \sum \limits_{n_1+n_2=l}{\sum \limits_{j=1}^{k}{\alpha_{n}^{(j)}\dfrac{\partial^{n_1+n_2}{f}}{\partial{\mu^{n_1}}\partial{\phi^{n_2}}}(y|\mu_{j},\phi_i)}}=0 \nonumber
\end{eqnarray}
for almost all $y$, then $\alpha_{n}^{(j)}=0$ for all $1 \leq j \leq k$ and pair $n = (n_1, n_2)$.
\end{itemize}
\end{definition}
Condition (A3.) for $r=0$ simply ensures that the mixture of $f$ distributions model uniquely identifies the mixture components. The strong identifiability conditions ($r \geq 1$) are required to establish the convergence rates \cite{chen1995optimal, nguyen2013convergence}. In the model~\eqref{eq:true_model}, there is a hierarchically higher level of parameters $(\theta_1, \theta_2)$ that we want to learn, and it connects to the observations through the link functions $h_1, h_2$ as $\mu = h_1(x, \theta_1), \phi = h_2(x, \theta_2)$. To ensure that $\theta_1$ and $\theta_2$ can be learned efficiently, we also need suitable conditions for $h_1$ and $h_2$. 

\begin{definition}
\label{def:identifiable-functions}
The family of functions $\{(h_1(x,\theta_1), h_2(x, \theta_2)) : \theta_1 \in \Theta_1, \theta_2 \in \Theta_2\}$ is called \underline{identifiable with complexity level $k$} respect to $\mathbb{P}_X$ if the following conditions hold:
\begin{itemize}
\item[(A4.)] For every set of $k+1$ distinct elements $(\theta_{11}, \theta_{21}),..., (\theta_{1(k+1)}, \theta_{2(k+1)})\in \Theta_1 \times \Theta_2$, there exists a subset $A \subset \mathcal{X}$, $\mathbb{P}_X(A)>0$ such that\\ $(h_1(x,\theta_{11}), h_2(x,\theta_{21})),..., $ $(h_1(x,\theta_{1(k+1)}), h_2(x,\theta_{2(k+1)}))$ are distinct for every $x \in A$;
\item[(A5.)] Moreover, if there are vector $\beta_1 \in \mathbb{R}^{d_1}, \beta_2 \in \mathbb{R}^{d_2}$ such that
{\fontsize{9.5}{10}\selectfont
\begin{equation}
    \beta_1^{\top} \dfrac{\partial}{\partial \theta_1} h_1(x, \theta_{1j}) = 0, \quad \beta_2^{\top} \dfrac{\partial}{\partial \theta_2} h_2(x, \theta_{2j}) = 0 \quad \forall x \in A\setminus N, \quad j = 1,\dots, k+1,
    \nonumber
\end{equation}}
where $N$ is a zero-measure set (i.e., $\Pbb_X(N) = 0$), then $\beta_1 = 0$ and $\beta_2 = 0$. 
\end{itemize}
\end{definition}

\begin{remark}\label{remark:definition-identifiable-functions}
    \begin{enumerate}
        \item Condition (A4.) is necessary for identifying regression mixture components. Indeed, for two distinct pairs $(\theta_1, \theta_2)$ and $(\theta_1', \theta_2')$ in $\Theta_1\times \Theta_2$, there may exists some point $x\in \Xcal$ so that $h_1(x, \theta_1) = h_1(x, \theta_1'), h_2(x, \theta_2) = h_2(x, \theta_2')$. If we only observe data $(x, y)$ at such $x$, it is not possible to distinguish between $(\theta_1, \theta_2)$ and $(\theta_1', \theta_2')$. 
        \item In linear models, condition (A5.) reads that there is no multicollinearity: If we model $h_1(x, \theta) = \theta_1 \psi_1(x) + \dots + \theta_{d_1} \psi_{d_1}(x)$, where $\psi_i$'s are pre-defined functions, then by substitute this into condition (A5.), we have $\psi_1,\dots, \psi_{d_1}$ must be linearly independent as functions of $x$. Otherwise, the model is not identifiable with respect to parameters $\theta_j$'s.
        
        \item 
        (A4.) and (A5.) can be viewed as generalization (to non-linear) and population versions of condition (1b) and (2) in \cite{grun2008finite} (or condition in Theorem 2.2 in \cite{hennig2000identifiablity}).
    \end{enumerate}
\end{remark}
Hence, the two conditions in Definition~\ref{def:identifiable-functions} are necessary for learning parameters of the mixture of regression models. The following result shows that 
Definition~\ref{def:strong_identifiability_f} and Definition~\ref{def:identifiable-functions} give sufficient conditions to deduce the strong identifiability given by Definition~\ref{def:strong_identifiability_hierarchy}, where the chain rule plays an essential role in its proof.

\begin{theorem}\label{thm:identifiable-equivalent}
For any complexity level $k$, if the family of distributions $f$ is strongly identifiable in order $r$ (via (A3.)) and the family of functions $h$ is identifiable (via (A4.) and (A5.)), then the family of conditional density $f(y|x)$ is strongly identifiable in order $r$, where $r=1,2$.
\end{theorem}

\subsection{Characterization of strong identifiability} \label{subsec: strong identifiability}
Theorem \ref{thm:identifiable-equivalent} provides a simple recipe for establishing the strong identifiability of the conditional densities arising in regression mixture models~\eqref{eq:true_model} by checking the identifiability conditions of family $f$ and family $h$. 
In the following, we provide specific examples.

\begin{proposition}\label{prop:identifiable-f}
(a) The family of location normal distribution $\{\Ncal(y|\mu, \sigma^2) : \mu \in \mathbb{R}\}$ with fixed variance $\sigma^2$ is identifiable in the second order, for $\Ncal(y|\mu, \sigma^2) = \exp(-(y-\mu)^2/2\sigma^2)$. The location-scale family $\{\Ncal(y|\mu, \sigma^2) : \mu \in \mathbb{R}, \sigma^2 \in \Rbb_{+}\}$ is identifiable in the first order; \\
(b) The Poisson family $\{\Poi(y|\lambda) : \lambda \in \mathbb{R}^{+}\}$ is identifiable in the second order;\\ 
(c) The family of Binomial distributions $\{\Bin(y|N, q) : q \in [0, 1]\}$ with fixed number of trials $N$ is identifiable in the first order with complexity level $k$ if $2k \leq N+1$, and is identifiable in the second order with complexity level $k$ if $3k \leq N+1$;\\ 
(d) The family of negative binomial distributions $\{\NB(y|\mu,\phi): \mu\}$ with fixed $\phi \in \mathbb{R}_+$ is identifiable in the second order. 
\end{proposition}
The identifiability conditions (A4.) and (A5.) usually hold for parametric models, as we see below. We first define a general class of functions:
\begin{definition}\label{def:finite-dim-class}
    We say a family of functions $\{h(x, \theta): \theta\in \Theta\}$ is completely identifiable if for any $\theta \neq \theta' \in \Theta$, we have $h(x, \theta) \neq h(x, \theta')$ almost surely in $\Pbb_X$. 
\end{definition}

\begin{proposition}\label{prop:complete-ident}
    If $h_1$ and $h_2$ are both completely identifiable, then the family of functions $\{(h_1, h_2)\}$ satisfies condition (A4.). 
\end{proposition}
Most functions used in parametric regression mixture models are completely identifiable.
\begin{proposition}\label{prop:family complete_ident}
    Suppose that $\Pbb_X$ has a density with respect to Lebesgue measure on $\Xcal$, then the following families of functions are completely identifiable and satisfy condition (A5.):\\
    (a) Polynomial of finite dimensions 
    {\fontsize{9}{10}\selectfont $h(x, \theta) = \sum_{d_1+\dots+ d_p\leq d, d_i\geq 0} \theta_{(d_1,\dots, d_p)} x_1^{d_1}\dots x_{p}^{d_p}$}, where $d \in \mathbb{N}_+$ and $\theta = (\theta_{(d_1,\dots, d_p)}: d_i\geq 0, \sum_{i=1}^{p} d_i \leq d)$; \\
    (b) Trigonometric polynomials in $\Rbb$: {\fontsize{9}{10}\selectfont $h(x, \theta) = a_0 + \sum_{n=1}^{d} b_{n} \sin(n x) + \sum_{n=1}^{d} c_{n} \cos(n x)$}, where $\theta = (a_0, b_1, \dots, b_d, c_1, \dots, c_d)$;\\
    (c) Mixtures of polynomials and trigonometric polynomials as in (a) and (b): $h(x, \theta) = \sum_{n=0}^{d} a_{n} x^{n} + \sum_{n=1}^{d} b_{n} \sin(n x) + \sum_{n=1}^{d} c_{n} \cos(n x)$, where $\theta = (a_0,\dots, a_d,$ $ b_1, \dots, b_d, c_1, \dots, c_d)$;\\
    (d) $h(x, \theta) = g(p(x, \theta))$, where $g$ is a diffeomophism and $p(x, \theta)$ is completely identifiable and satisfies condition (A5.). 
\end{proposition}
\begin{remark}
In a general linear model,  $h(x, \theta) = \exp(\theta^{\top} x) \in \Rbb_{+}$ or $h(x, \theta) = \sigma(\theta^{\top} x)\in [0,1]$, where $\sigma$ is the sigmoid (inverse logit) function. Both the exponential function and sigmoid function are one-to-one, and $\theta^{\top} x$ is a first-order polynomial, so the above results apply.
\end{remark}

\subsection{Inverse bounds for mixture of regression models}\label{subsec:inverse-bounds} 
At the heart of our convergence theory for parameter learning in regression mixture models lies a set of inverse bounds, which are given as follows. 


\begin{theorem}\label{thm:inverse-bounds}
\begin{enumerate}
    \item [(a)] (Exact-fitted) Given $G_0 \in \Ecal_{k_0}(\Theta)$ for $k_0 \in \mathbb{N}_+$. Suppose that the family of conditional densities $\{f(\cdot|h_1, h_2)\}$ is identifiable in the first order, and the family of functions $(h_1, h_2)$ is identifiable (with the complexity level $k_0$). Then for all $G\in \Ecal_{k_0}(\Theta)$, there holds
\begin{equation}\label{eq:inverse-bound-exact-fitted}
    \Ebb_{X} d_{TV}  (f_{G}(\cdot|X), f_{G_0}(\cdot|X)) \succcurlyeq W_1(G, G_0),
\end{equation}
where the constant in this inequality depends only on $G_0, h_1, h_2, f, \Pbb_X$, and $\nu$ (but not on $G$).
\item  [(b)] (Over-fitted) Given $G_0 \in \Ecal_{k_0}(\Theta)$ for $k_0 \in \mathbb{N}_+$ and $k_0 \leq \Kup$ for some natural number $\Kup$. Suppose that the family of conditional densities $\{f(\cdot|h_1, h_2)\}$ is identifiable in the second order, and the family of functions $(h_1, h_2)$ is identifiable (with the complexity level $\Kup$). Then for all $G\in \Ocal_{\Kup}(\Theta)$, there holds
\begin{equation}\label{eq:inverse-bound-over-fitted}
    \Ebb_{X} d_{TV}(f_{G}(\cdot|X), f_{G_0}(\cdot|X)) \succcurlyeq W_2^2(G, G_0).
\end{equation}
where the constant in this inequality depends only on $G_0, h_1, h_2, f, \Pbb_X$, and $\nu$ (but not on $G$).
\end{enumerate}

\end{theorem}
If the true number of components $k_0$ is known, then Theorem \ref{thm:inverse-bounds} entails that the convergence rate for parameter estimations can be as fast as the convergence rate for conditional densities under the total variation distance. However, in practice, we may not know $k_0$ and fit the system by a large number $\Kup$. In this over-fitted regime, provided that the identifiability conditions for distribution $f$ and function $h$ are satisfied in the second order, the convergence rate for parameter estimation may be twice as slow as that of the conditional densities. 


%

Based on the convergence behavior of the regression mixture model's parameters, we can establish guarantees on the prediction error for the response variable.
The following bounds will be useful for deducing the prediction error bounds from that of parameter estimates. 
\begin{proposition}
\label{lem:inverse-bounds-prediction-error}
Suppose that the density $f$ and link functions $h_1, h_2$ are uniformly Lipschitz, then for all $G\in \Ocal_{\Kup}(\Theta), G = \sum_{j=1}^{\Kup} p_j \delta_{(\theta_{1j}, \theta_{2j})}$ and $r\geq 1$, we have
\begin{equation*}
    W_r(G, G_0) \succcurlyeq \Ebb_X W_r\left(\sum_{j=1}^{\Kup} p_j \delta_{(h_1(X, {\theta}_{1j}), h_2(X, {\theta}_{2j}))}, \sum_{j=1}^{k_0} {p}^{0}_j \delta_{(h_1(X, {\theta}^{0}_{1j}), h_2(X, {\theta}^{0}_{2j}))}\right), 
\end{equation*}
and
\begin{equation*}
    W_r(G, G_0) \succcurlyeq \Ebb_{X} \left|\sum_{j=1}^{\Kup} p_j h_u(X, \theta_{uj}) - \sum_{i=1}^{k_0} p_i^0 h_u(X, \theta^0_{ui}) \right|  \quad \forall u = 1,2,
\end{equation*}
where the constants in those inequalities only depend on Lipschitz constants of $h_1$ and $h_2$.
\end{proposition}

\subsection{Consequences of lack of strong identifiability}\label{subsec:non-strong-ident}

Strong identifiability notions characterize the favorable conditions under which efficient regression learning is possible in the mixture setting. Next, we turn our attention to the consequence of the lack of strong identifiability. Firstly, we note that the normal distributions satisfy the well-known heat equation: $\dfrac{\partial^2 }{\partial \mu^2} \Ncal(y|\mu, \sigma^2) = \dfrac{\partial }{\partial (\sigma^2)} \Ncal(y|\mu, \sigma^2), \quad \forall \mu\in \Rbb, \sigma^2\in \Rbb_{+}$,
so the mixture of normal regression model no longer satisfies the strong identifiability condition in the second order. This (Fisher information matrix's) singularity structure is universal (i.e., holds for all $\mu, \sigma^2$). Therefore, the inverse bounds presented in Theorem~\ref{thm:inverse-bounds} may not hold and potentially lead to slow convergence rates for Gaussian mixtures (see also \cite{ho2019convergence}).

Another interesting example arises in the negative binomial regression mixture models, which have been utilized in the traffic analysis of heterogeneous environments \cite{park2010bias, PARK2009683}. These authors observed via many empirical experiments that the quality of parameter estimates and the prediction performance may be affected by the (overlapped) sample-mean values obtained from the data. However, there was a lack of precise theoretical understanding. Our theoretical framework can be applied to shed light on the behavior of this class of regression mixture model. It starts with the observation that the mixture of negative binomial distributions does not satisfy the first-order strongly identifiable condition. Moreover, we can identify precisely the instances where strong identifiability fails to hold and investigate the impact on the quality of parameter estimates and the prediction performance in such instances. 

First, we note that the mean-dispersion negative binomial conditional density $\{\NB(y|\mu, \phi):\mu\in \Rbb_{+},\phi\in \Rbb_{+}\}$ satisfies the following equation: 
\begin{equation}\label{eq:weak-ident-NB}
    \dfrac{\partial }{\partial \mu} \NB(y|\mu, \phi) = \dfrac{\phi}{\mu} \NB\left(y|\mu\frac{\phi+1}{\phi}, \phi+1\right) - \dfrac{\phi}{\mu} \NB(y|\mu, \phi),\,\forall y\in \Nbb.
\end{equation}
Thus, a 2-mixture of negative binomial distributions {\fontsize{9}{10}\selectfont $\NB(y|\mu_1,\phi_1)$} and {\fontsize{9}{10}\selectfont $\NB(y|\mu_2,\phi_2)$} such that
\begin{equation}\label{eq: Non-strong Nega}
        \dfrac{\mu_1}{\phi_1} = \dfrac{\mu_2}{\phi_2} \,\,\text{ and }\,\,
        \phi_1 = \phi_2 + 1
\end{equation}
does not satisfy the strong identifiability condition in the first order, resulting in slow convergence for parameter estimation. 
It can be seen via the following minimax lower bound: 
\begin{theorem}\label{thm:minimax-NB}
    Consider a mixture of negative binomial regression model with link functions $h(x, \theta_1) = \exp(\theta_{10} + (\bar{\theta}_1)^{\top}x), h(x, \theta_2) = \theta_2$ for $\theta_1 = [\theta_{10}, \bar{\theta}_1]\in \Rbb^{p+1}$ and $\theta_2 \in \Rbb_{+}$. Then for any measurable estimate $\hat{G}_n$ of the mixing measure $G$, the following holds for any $r\geq 1$:
    \begin{equation}
        \inf_{\hat{G}_n\in \Ecal_{k_0}} \sup_{G\in \Ecal_{k_0}} \Ebb_{\Pbb_{G}} W_r(\hat{G}_n, G)\succcurlyeq n^{-1/(4r)}.
    \end{equation}
\end{theorem}
Fortunately, not all are bad news for negative binomial regression mixtures. We will show that non-identifiability occurs only in a Lebesgue measure zero subset of the parameter space.
\begin{proposition}
         Given $k$ distinct pairs $(\mu_1, \phi_1), \dots, (\mu_k, \phi_k) \in \mathbb{R}_{+}\times \mathbb{R}_{+}$ such that there does not exist two indices $i\neq j$ satisfying 
        $\dfrac{\mu_i}{\phi_i} = \dfrac{\mu_j}{\phi_j}$ and 
        $|\phi_i - \phi_j| = 1$,
then the mixture of negative binomials $(\NB(\mu_i, \phi_i))_{i=1}^{k}$ is strongly identifiable in the first order. If we further assume that there does not exist two indices $i\neq j$ satisfying 
        $\dfrac{\mu_i}{\phi_i} = \dfrac{\mu_j}{\phi_j}$ and 
        $|\phi_i - \phi_j| \in \{1,2\}$,
then the mixture of negative binomials $(\NB(\mu_i, \phi_i))_{i=1}^{k}$ is strongly identifiable in the second order.
\end{proposition} 

Finally, we note that the theory established earlier (Theorem~\ref{thm:identifiable-equivalent} and Theorem~\ref{thm:inverse-bounds}) represent sufficient conditions. There still may exist non-strongly identifiable families $f$ and $(h_1, h_2)$ that lead to strong identifiable $f(\cdot|h_1, h_2)$. For example, the mixture of two Binomial distributions $p_G(y) = p_1 \textrm{Bin}(y|1, q_1) + p_2 \textrm{Bin}(y|1, q_2)$
is not identifiable, because for instance, $p_{G_1} = p_{G_2}$ for $G_1 = 0.5 \delta_{0.3} +  0.5 \delta_{0.7}$ and $G_1 = 0.5 \delta_{0.2} +  0.5 \delta_{0.8}$. However, the mixture of two logistic regression models  $f_G(y|x) = p_1 \textrm{Bin}(y|1, \sigma(\theta_1^{\top}x)) + p_2 \textrm{Bin}(y|1, \sigma(\theta_2^{\top}x)))$
is strongly identifiable (see Proposition B.1 in Appendix B) and enjoys the inverse bound as well as standard convergence rates. Unfortunately, such a result is difficult to generalize. This again highlights our general theory developed in this section, which is applicable to a vast range of kernels $f$ and link functions $(h_1, h_2)$.
The pathological phenomena described will be revisited in Section \ref{sec:experiments}. 

\section{Statistical efficiency in learning regression mixtures }\label{sec:stats-efficiency}

Building on the previous section, we are ready to present convergence rates of the maximum likelihood estimator, and a Bayesian posterior contraction theory for the quantities of interest.

\subsection{Maximum (conditional) likelihood estimation}\label{sec:conv-rate-MLE}

Given $n$ i.i.d. observations $(x_1, y_1), \dots, (x_n, y_n)$, where $x_j \overset{i.i.d.}{\sim} \Pbb_X$ and $y_j | x_j \sim f_{G_0}(y|x), j = 1,\dots, n$, for $G_0 = \sum_{i=1}^{k_0} p_{j}^0\delta_{(\theta_{1j}^0, \theta_{2j}^0)}$. Denote the maximum likelihood estimate by
\begin{equation*}
    \widehat{G}_n := \argmax_{G \in \Ecal_{k_0}(\Theta)} \sum_{j=1}^{n}\log f_G(y_j|x_j),
\end{equation*}
in the exact-fitted setting, and we change the $\Ecal_{k_0}(\Theta)$ in the above formula to $\Ocal_K(\Theta)$, where $K \geq k_0$ in the over-fitted setting. It is implicitly assumed in this section that $\widehat{G}_n$ is measurable, otherwise a standard treatment using an outer measure of $\Pbb_{G_0}$ instead of $\Pbb_{G_0}$ can be invoked.
To obtain the rate of convergence of $\widehat{G}_n$ to $G_0$, we combine the inverse bounds above with the convergence of density estimates based on the standard theory of M-estimation for regression problems~\cite{Vandegeer}.  For conditional density estimation, the convergence behavior of $f_{\widehat{G}_n}$ to $f_{G_0}$ is evaluated in the sense of the \emph{expected} Hellinger distance:
\begin{align*}
    \overline{d}^2_{H}(f_{G}, f_{G'}) & := \Ebb_X d_{H}^2(f_{G}(\cdot|X), f_{G'}(\cdot|X)) \\
    & = \dfrac{1}{2}\int_{\Xcal} \int_{\Ycal}  (\sqrt{f_{G}(y|x)}-\sqrt{f_{G'}(y|x)})^2 d\nu(y) d\Pbb_X(x),
\end{align*}
for all $G, G' \in \cup_{k=1}^{\infty} \Ocal_{k}(\Theta)$. 
To this end, recall several basic notions related to the entropy numbers of a class of functions. For any $k \in \mathbb{N}$, set
\begin{equation*}
    \Fcal_k(\Theta) = \biggr \{f_G(y|x) : G \in \Ocal_k(\Theta) \biggr \},\quad  \overline{\Fcal}_k^{1/2}(\Theta) = \biggr \{f_{(G+G_0)/2}^{1/2}(y|x) : G \in \Ocal_k(\Theta) \biggr \},
\end{equation*}
and the Hellinger ball centered around $f_{G_0}$:
\begin{equation*}
    \overline{\Fcal}_k^{1/2}(\delta) = \overline{\Fcal}_k^{1/2}(\Theta, \delta) = \biggr \{f^{1/2}\in \overline{\Fcal}_k^{1/2}(\Theta) : \overline{d}_{H}(f, f_{G_0})\leq \delta \biggr \}.  
\end{equation*}
The complexity (richness) of this set is characterized in the following entropy integral:
\begin{equation}\label{eq:int_bracketing}
    \mathcal{J}(\delta) := \mathcal{J}(\delta, \overline{\Pcal}_k^{1/2}(\Theta, \delta)) = \int_{\delta^2/2^{13}}^{\delta} H_B^{1/2}(u, \overline{\Fcal}_k^{1/2}(\delta), L_2(\Pbb_X\times \nu)) du \vee \delta,
\end{equation}
where $H_B$ is the bracketing entropy number. An useful tool for establishing the rate of convergence under expected conditional density estimation by the MLE is given by the following theorem, which is an adaptation of Theorem 7.4. in \cite{Vandegeer} or Theorem 7.2.1. in \cite{gine2021mathematical}).
\begin{theorem}\label{thm:density_estimation_rate}
    Take $\Psi(\delta) \geq \mathcal{J}(\delta, \overline{\Fcal}_{k}^{1/2}(\delta))$ in such a way that $\Psi(\delta)/\delta^2$ is a non-increasing function of $\delta$. Then, for a universal constant $c$ and for 
    \begin{equation}\label{eq:entropy-condition-MLE}
        \sqrt{n} \delta_n^2 \geq c \Psi(\delta_n),
    \end{equation}
    we have for all $\delta \geq \delta_n$ that
        $\Pbb_{G_0}(\overline{d}_{H}(f_{\widehat{G}_n}, f_{G_0}) > \delta) \leq c \exp\left(-\dfrac{n\delta^2}{c^2}\right)$. 
\end{theorem}
Combining Theorem~\ref{thm:density_estimation_rate} with the inverse bounds established in Section \ref{sec:identifiable-inverse-bound}, we readily arrive at the following concentration inequalities for the MLE's parameter estimates based on the bracketing entropy numbers and its entropy integral given by Eq.~\eqref{eq:int_bracketing}.

\begin{theorem}\label{thm:convergence_parameter_estimation}
\begin{enumerate}
    \item [(a)] (Exact-fitted) Suppose that $k_0$ is known, the entropy condition~\eqref{eq:entropy-condition-MLE} holds, the family of conditional densities $f(\cdot|h_1, h_2)$ is identifiable in the first order, and $(h_1, h_2)$ is identifiable. Then, for any $G_0 \in \Ecal_{k_0}(\Theta)$, there exist a constant $C$ depending on $G_0$ and universal constant $c$ such that
    \begin{equation*}
        \Pbb_{G_0} \left(W_1(\widehat{G}_n, G_0) > C\delta\right) \leq c\exp(-n\delta^2/c^2).
    \end{equation*}
    \item [(b)] (Over-fitted) Suppose that $k_0$ is unknown but $k_0 < \Kup$ known, the entropy condition~\eqref{eq:entropy-condition-MLE} holds, the family of conditional densities $f(\cdot|h_1, h_2)$ is identifiable in the second order, and $(h_1, h_2)$ is identifiable. Then, for any $G_0 \in \Ecal_{k_0}(\Theta)$, there exist a constant $C$ depending on $G_0$ and $K$ and universal constant $c$ such that
    \begin{equation*}
        \Pbb_{G_0} \left(W_2^2(\widehat{G}_n, G_0) > C\delta\right) \leq c \exp(-n\delta^2/c^2).
    \end{equation*}
\end{enumerate} 
\end{theorem}


For concrete rates of convergence, we need to estimate the entropy integral \eqref{eq:int_bracketing}. 
For many parametric models, we will find that the convergence rate for the mixing measure $G$ is $(\log n/ n)^{1/2}$ under $W_1$ for the exact-fitted setting and $(\log n/n)^{1/4}$ under $W_2$ for the over-fitted setting. In the following, we shall present a set of mild assumptions and show that under these assumptions, such parametric rates can be established. 

\begin{enumerate}
    \item[(B1.)] (Assumptions on kernel densities) For the family of distribution $\{f(y|\mu, \phi) | \mu\in H_1, \phi \in H_2\}$,  $\norm{f(\cdot|\mu, \phi)}_{\infty}$ is uniformly bounded and uniformly light tail probability, i.e., there exist constant $\underline{c}, \overline{c}$ and constant $d_1, d_2, d_3 > 0$ such that $f(y| \mu, \phi) \leq d_1\exp(-d_2 |y|^{d_3}),$
    for all $y \geq \overline{c}$ or $y \leq \overline{c}$ and $\mu \in H_1, \phi \in H_2$. 
    \item[(B2.)] (Lipschitz assumptions) Both the kernel densities $\{f(y|\mu, \phi) | \mu\in H_1, \phi \in H_2\}$ and link functions $h_1$ and $h_2$ are uniformly Lipschitz in the sense of \eqref{eq:uniform-Lipschitz-f} and \eqref{eq:uniform-Lipschitz-h}.
\end{enumerate}
We note that:
\begin{proposition}\label{prop:satisfied_densities}
    The families of Normal, Poisson, Binomial, and negative binomial distribution satisfy condition (B1.).
\end{proposition}
The predictive performance of MLE for the regression mixture model is given as follows.

\begin{theorem}\label{thm:simple_density_estimation_rate}
    Given assumptions (B1.) and (B2.), and the dominating measure $\nu$ is Lebesgue over $\Rbb$ or counting measure on $\Zbb$. Then, there exists a constant $C$ depending on $\Theta, f, h_1, h_2$, and a universal constant $c$ such that
    \begin{equation*}
        \Pbb_{G_0} \left(\overline{d}_{H}(f_{\widehat{G}_n}, f_{G_0} ) > C\sqrt{\dfrac{\log(n)}{n}}\right) \leq c\exp(-\log n/c^2).
    \end{equation*}
\end{theorem}
Combining Theorem~\ref{thm:inverse-bounds} and Theorem~\ref{thm:simple_density_estimation_rate}, we arrive at the convergence rates for the maximum (conditional) likelihood estimates for the model parameters:
\begin{theorem}\label{thm:simple_convergence_parameter_estimation}
\begin{enumerate}
\item [(a)] (Exact-fitted) Suppose that $k_0$ is known, the family of conditional densities $f(\cdot|h_1, h_2)$ is identifiable in the first order, and the family of functions $(h_1,h_2)$ is identifiable. Furthermore, assume (B1.) and (B2.) hold, then for any $G_0 \in \Ecal_{k_0}(\Theta)$, there  exist constant $C$ depending on $G_0, \Theta, f, h_1, h_2$ and a universal constant $c$ such that
    \begin{equation*}
        \Pbb_{G_0} \left(W_1(\widehat{G}_n, G_0) > C(\log n / n)^{1/2}\right) \leq c\exp(-\log(n)/c^2).
    \end{equation*}

\item [(b)] (Over-fitted) Suppose that $k_0$ is unknown and is upper bounded by a known number $\Kup < +\infty$, the family of conditional densities $f(\cdot|h_1, h_2)$ is identifiable in the second order, and the family of functions $(h_1,h_2)$ is identifiable. Furthermore, assume (B1.) and (B2.) hold, there  exist constant $C$ depending on $G_0, \Kup, \Theta, f, h_1, h_2$ and a universal constant $c$ such that
    \begin{equation*}
        \Pbb_{G_0} \left(W_2(\widehat{G}_n, G_0) > C(\log n / n)^{1/4}\right) \leq c\exp(-\log(n)/c^2).
    \end{equation*}
\end{enumerate}
\end{theorem}
%

\subsection{Bayesian posterior contraction theorems for parameter inference}\label{sec:posterior-contraction-rate}



Given i.i.d. pairs $(x_1, y_1), \dots, (x_n, y_n)$ such that $y_i|x_i \sim f_{G_0}(y|x)$ for some true latent mixing measure $G_0 = \sum_{j=1}^{k_0} p_i^0 \delta_{(\theta_{1j}^0, \theta_{2j}^0)}$, and $x_i\overset{\text{i.i.d.}}{\sim} \Pbb_X $. In the Bayesian regime, we model the data as $ y_i | x_i \sim f_{G}(y|x),$
where $G \sim \Pi$ with $\Pi$ being some prior distribution on the space mixing measures. Let $\Gcal$ denote the support of the prior $\Pi$ on the mixing measure $G$. The nature of the prior distribution $\Pi$ depends on the several different settings that we will consider. In the exact-fitted setting, we assume $\Kup \equiv k_0  <+\infty$ is known, whereas the over-fitted setting means that the upper bound $k_0 \leq \Kup < +\infty$ is given, but $k_0$ unknown. In both cases, $\Gcal = \Ocal_{\Kup}$ so $\Pi$ is in effect a prior distribution on $(\Ocal_{\Kup}, \Bcal(\Ocal_{\Kup}))$. Later, we shall assume that neither $k_0$ nor an upper bound $\Kup$ is given; instead, a random variable $\Kvar$ is used to represent the number of mixture components and endowed with a prior distribution.

 By the Bayes' rule, the posterior distribution of the parameter $G$ is given by
\begin{equation*}
    \Pi(G\in B | x^{[n]}, y^{[n]}) = \dfrac{\int_B \prod_{i=1}^{n} f_G(y_i|x_i)  d\Pi(G)}{\int_{\Ocal_{\Kup}} \prod_{i=1}^{n} f_G(y_i|x_i)  d\Pi(G)}
\end{equation*}
for any measurable set $B \subset \Gcal$. Now, we want to study the posterior contraction rate of $\Pi( \cdot | x^{[n]}, y^{[n]})$ to the true latent mixing measure $G_0$ as $n\to \infty$. We proceed to describe several standard assumptions on the prior often employed in practice.

\subsection*{The case of known upper bound $\Kup$}

\begin{enumerate}
    \item[(B3.)] (Prior assumption) Prior $\Pi$ on space $\Ocal_{\Kup}$ is induced by prior $\Pi_p \times \Pi_{\theta_1}^{\Kup} \times \Pi_{\theta_2}^{\Kup}$ on $\{(p_1, \dots, p_{\Kup},$ $ \theta_{11}, \dots, \theta_{1\Kup}$, $\theta_{21}, \dots, \theta_{2\Kup})| p_i\geq 0, \sum_{i=1}^{\Kup}p_i=1, \theta_{1j} \in \Theta_1, \theta_{2j} \in \Theta_2\}$, where $\Pi_p$ is a prior distribution of $(p_{j})_{i=1}^{\Kup}$ on $\Delta^{\Kup-1}$, $\Pi_{\theta_1}$ is a prior distribution of $\theta_{1j}$ on $\Theta_1$, and $\Pi_{\theta_2}$ is a prior distribution of $\theta_{2j}$ on $\Theta_2$, independently for $i=1,\dots, \Kup$. We further assume that $\Pi_p, \Pi_{\theta_1}, \Pi_{\theta_2}$ have a density with respect to Lebesgue measure on $\Delta^{\Kup-1}, \Theta_1, \Theta_2$, respectively, which are bounded away from zero and infinity.  

    \item[(B4.)] There exists $\epsilon_0 >0$ such that for all $G \in \Ocal_{\Kup}(\Theta)$ satisfying $W_1(G, G_0) \leq \epsilon_0$, we have $\Ebb_{\Pbb_{G_0}} (f_{G_0} / f_{G})\leq M_0$ for $M_0$ only depends on $\epsilon_0, G_0, \Kup, \Theta$.
    
\end{enumerate}

\comment{
\begin{remark}
Assumptions (B3.) on $\Pi_p, \Pi_{\theta_1}, \Pi_{\theta_2}$ are standard when proving posterior contraction of Bayesian methods. We require the density functions of those prior distributions to be bounded from zero so that there is enough mass around the true parameters and bounded away from infinity so that it is possible to control the complexity of the family of induced densities.
\end{remark}
}
The posterior contraction behavior for conditional densities is given as follows.
\begin{theorem}\label{thm:posterior-contraction-rate-regression}
    Assume that (B2.)-(B4.) hold. For any $G_0 \in \Gcal$, there exists constant $C$ depending on $\Theta, f, h_1, h_2$ such that as $n \rightarrow \infty$,
    {\fontsize{9}{10}\selectfont
    \begin{equation}
        \Pi\left(G: \overline{d}_{H}(f_{G}(y|X), f_{G_0}(y|X)) \geq C\sqrt{\dfrac{\log(n)}{n}} \bigg| x^{[n]}, y^{[n]}\right)
        \rightarrow 0\; 
         \textrm{in}\; \otimes_{i=1}^{n} \Pbb_{G_0}\; \textrm{-probability}.
    \end{equation}}
\end{theorem}
Combining the above result with the inverse bounds developed in Section~\ref{sec:identifiable-inverse-bound} leads to the contraction rates of the posterior distribution in the exact-fitted and over-fitted settings. 
\begin{theorem}\label{thm:posterior-contraction-rate-parameter}
    Suppose that assumptions (A1.)-(A3.) and (B2.)-(B4.) hold. 
    Fix any $G_0 \in \Gcal$.
    \begin{enumerate}
        \item [(a)] (Exact-fitted) If $\Kup = k_0$, there exists some constant $C_1$ depending on $G_0, \Theta, f, h_1, h_2$ such that
        {\fontsize{9}{10}\selectfont
        \begin{equation*}
        \Pi\left(G: W_1(G, G_0) \geq C_1\left(\dfrac{\log(n)}{n}\right)^{1/2} \bigg| x^{[n]}, y^{[n]}\right) \stackrel{n\rightarrow \infty}{\longrightarrow} 0\; 
         \textrm{in}\; \otimes_{i=1}^{n} \Pbb_{G_0}\; \textrm{-probability}.
    \end{equation*}}
   
    \item [(b)] (Over-fitted) If $\Kup \geq k_0$, there exists some constant $C_2$ depending on $G_0, \Kup, \Theta, f, h_1, h_2$ such that
    {\fontsize{9}{10}\selectfont
        \begin{equation*}
        \Pi\left(G: W_2(G, G_0) \geq C_2 \left(\dfrac{\log(n)}{n}\right)^{1/4} \bigg| x^{[n]}, y^{[n]}\right) 
        \stackrel{n\rightarrow \infty}{\longrightarrow} 0\; 
         \textrm{in}\; \otimes_{i=1}^{n} \Pbb_{G_0}\; \textrm{-probability}.
    \end{equation*}}

    \end{enumerate}
\end{theorem}

\subsection*{The case of unknown $\Kup$} 
Finally, when the number of components $K$ and its upper bound are unknown, there are various approaches for prior specification. Here, we adopt the widely utilized "mixture-of-finite-mixtures" prior \cite{Miller-2016}. In particular, the prior distribution $\Pi$ on the space $\Gcal$ of mixing measures is induced by the following specification.

\begin{enumerate}
    \item[(B5.)]  A prior distribution $\Pi_\Kvar$ on $\Kvar$ with support in $\Nbb$, i.e., $\Pi_K(\Kvar = k) > 0$ for all $k\in \Nbb$. 
    \item[(B6.)] For each $k\in \Nbb$, given the event $K=k$, the conditional prior distribution of the mixing measure $G = \sum_{j=1}^{k}p_j \delta_{(\theta_{1j},\theta_{2j})}  \in \Ecal_{k}$ is induced by the following specification: $\Pi_p \times \Pi_{\theta_1}^{k} \times \Pi_{\theta_2}^{k}$ on $\{(p_1, \dots, p_{k}, \theta_{11}, \dots, \theta_{1k}$, $\theta_{21}, \dots, \theta_{2k})| p_i\geq 0, \sum_{i=1}^{k}p_i=1, \theta_{1j} \in \Theta_1, \theta_{2j} \in \Theta_2\}$, where $\Pi_p$ is a prior distribution of $(p_{j})_{i=1}^{k}$ on $\Delta^{k-1}$, $\Pi_{\theta_1}$ is a prior distribution of $\theta_{1j}$ on $\Theta_1$, and $\Pi_{\theta_2}$ is a prior distribution of $\theta_{2j}$ on $\Theta_2$, independently for $i=1,\dots, k$. Assume that $\Pi_p, \Pi_{\theta_1}, \Pi_{\theta_2}$ have a density with respect to Lebesgue measure on $\Delta^{k-1}, \Theta_1, \Theta_2$, respectively, which are bounded away from zero and infinity.  

    \item[(B7.)] For each $k\in \Nbb$, there exists $\epsilon_0 >0$ such that for all $G \in \Ecal_{k}(\Theta)$ satisfying $W_1(G, G_0) \leq \epsilon_0$, we have $\Ebb_{\Pbb_{G_0}} (f_{G_0} / f_{G})\leq M_0$ for $M_0$ only depends on $\epsilon_0, G_0, k, \Theta$.
\end{enumerate}  

\begin{theorem}\label{thm:MFM}
Assume that (A1.)-(A3.), (B2.), and (B5.)-(B7.) hold. There exists a subset $\Gcal_0 \subset \Gcal$, where $\Pi(\Gcal_0) = 1$ such that for all
$k_0 \in \Nbb$ and $G_0 \in \Gcal_0 \cap \Ecal_{k_0}$, there hold as $n\rightarrow \infty$
\begin{enumerate}
   \item [(a)] $\Pi(K = k_0 | x^{[n]}, y^{[n]}) \to 1$ a.s. under  $\otimes_{i=1}^{n}\Pbb_{G_0}$;
   \item [(b)] there is a constant $C$ depending on $G_0, \Theta, f, h_1, h_2$ such that
   {\fontsize{9}{10}\selectfont
     \begin{equation*}
        \Pi\left(G: W_1(G, G_0) \geq C\left(\dfrac{\log(n)}{n}\right)^{1/2} \bigg| x^{[n]}, y^{[n]}\right) \to 0\;
         \textrm{in}\; \otimes_{i=1}^{n} \Pbb_{G_0}\; \textrm{probability}.
    \end{equation*}
    }
    \end{enumerate}
\end{theorem}

\section{Simulations and data illustrations}\label{sec:experiments}

\paragraph{Regression mixtures vs unconditional mixtures}

The characterization results (Theorem~\ref{thm:identifiable-equivalent} and propositions in Section~\ref{subsec: strong identifiability}) provide easy-to-check sufficient conditions for strong identifiability (in the sense of Def.~\ref{def:strong_identifiability_hierarchy}). 
Part of the sufficient conditions requires that the kernel density $f$ be strongly identifiable up to a certain order, a standard condition considered in the asymptotic theory for finite (and unconditional) mixture models~\cite{nguyen2013convergence,Ho-Nguyen-Ann-16,ho2019singularity}.
It is noteworthy that the strong identifiability condition of a mixture of regression model given in Def.~\ref{def:strong_identifiability_hierarchy} is typically a weaker condition than that of a standard unconditional mixture model, because the presence of the covariate $x$ makes the conditional mixture model more constrained. Hence, it is possible that for an unconditional mixture of kernel densities $f$ strong identifiability and hence the inverse bound $V\succcurlyeq W_1$ may not hold, but when $f$ is utilized in a regression mixture model instead, the strong identifiability and hence the inverse bound still holds. 

We demonstrate this observation by a theoretical result given in Proposition B.1 for the mixture of binomial regression models. 
To illustrate this result by a simulation study, let
    $p_{G}(y) := \sum_{i=1}^{k} p_i \textrm{Bin}(y| N, q_i)$,
for $G = \sum_{i=1}^{k} p_i \delta_{q_i} \in \Ocal_{K}([0,1])$, where $N$ is a fixed natural number. From the discussion in Section~\ref{sec:identifiable-inverse-bound}, we know that this model is only strongly identifiable in the first order if $2k \leq N+1$. It means the inverse bound may not hold when $2k > N+1$. Let $k = 2$, $N=1$ so that $2k> N+1$, and $G_0 = 0.5 \delta_{0.3} + 0.5 \delta_{0.7}$, and then uniformly generate 2000 random samples of $G$ around $G_0$. We compare $W_1(G, G_0)$ against $d_{TV}  (  p_G, p_{G_0})$ to see if the inverse bound $d_{TV}  (  p_{G}, p_{G_0})\succcurlyeq W_1(G, G_0)$ holds or not. It can be seen in Fig.~\ref{fig:inverse-bounds}(a) that such an inverse bound does not hold. In contrast, for the mixture of binomial regression model under the same setting $k=2, N=1$ (a.k.a. mixture of two logistic regression):
    $f_{G}(y|x) = p_1 \textrm{Bin}(y|1, \sigma(\theta_1 x)) + p_2 \textrm{Bin}(y|1, \sigma(\theta_2 x))$
for $G = p_1\delta_{\theta_1} + p_2 \delta_{\theta_2}$ and $\sigma$ being the sigmoid function, the inverse bound as established by Theorem~\ref{thm:inverse-bounds} still holds. We uniformly sample 2000 measure $G$ around $G_0 = 0.5 \delta_{0.5} + 0.5 \delta_{5}$ and plot $W_1(G, G_0)$ against $\Ebb_{X} d_{TV}  (  f_G(y|X), f_{G_0}(y|X))$, for $X \sim \textrm{Uniform}([-6, 6])$. The relationship $\Ebb_{X} d_{TV}  (  f_G(y|X), f_{G_0}(y|X)) \succcurlyeq W_1(G, G_0)$ holds in this scenario (Fig.~\ref{fig:inverse-bounds}(b)). As a consequence, the mixture of logistic regression model still enjoys the convergence rate of $n^{-1/2}$ for parameter estimation in the exact-fitted setting. 

\begin{figure}[ht]
      \centering
      \subcaptionbox*{\scriptsize \centering(a) Mixture of binomial distributions \par}{\includegraphics[width = 0.49\textwidth]{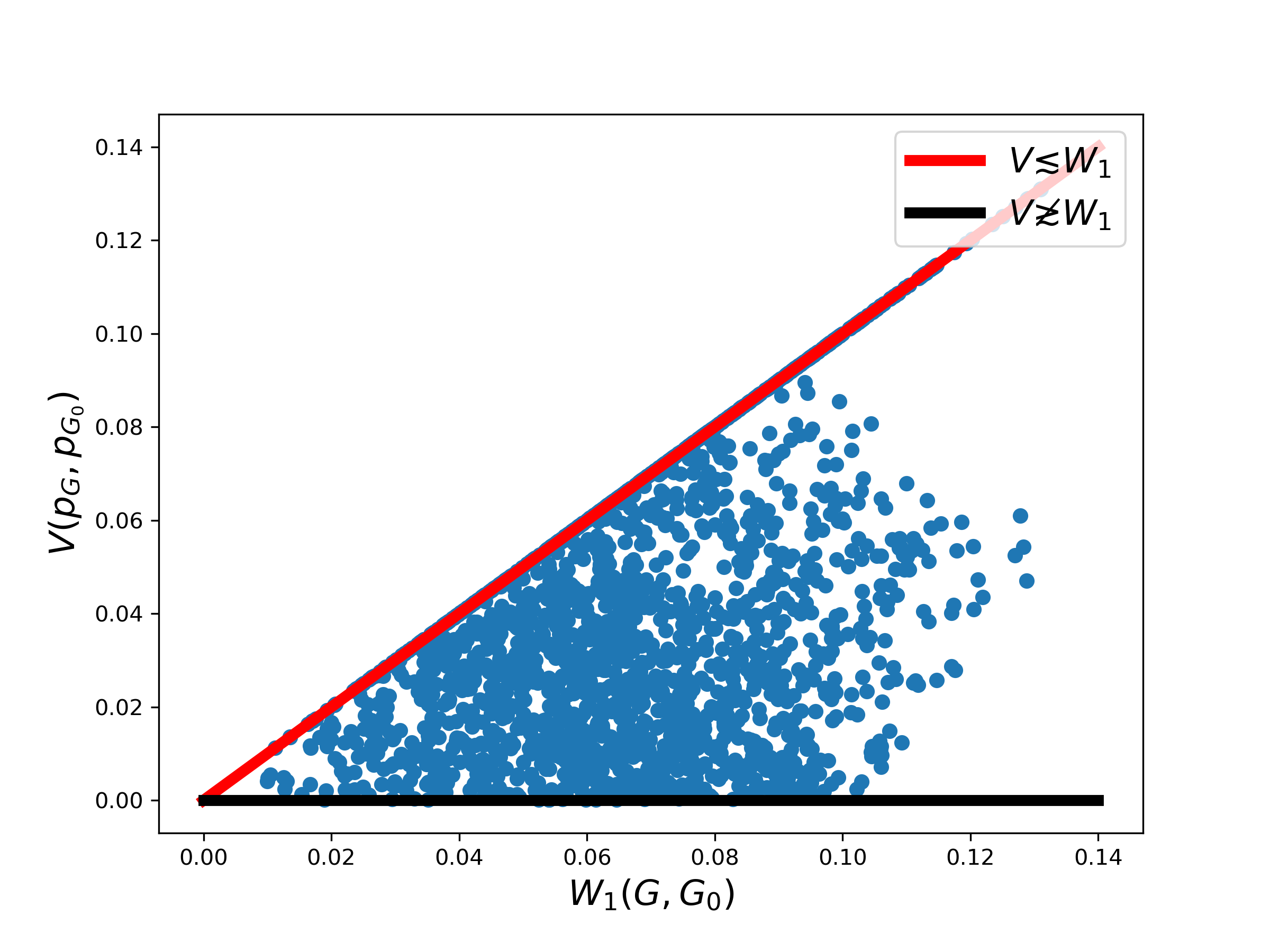}}
      \subcaptionbox*{\scriptsize \centering(b) Mixture of binomial regression models \par}{\includegraphics[width = 0.49
\textwidth]{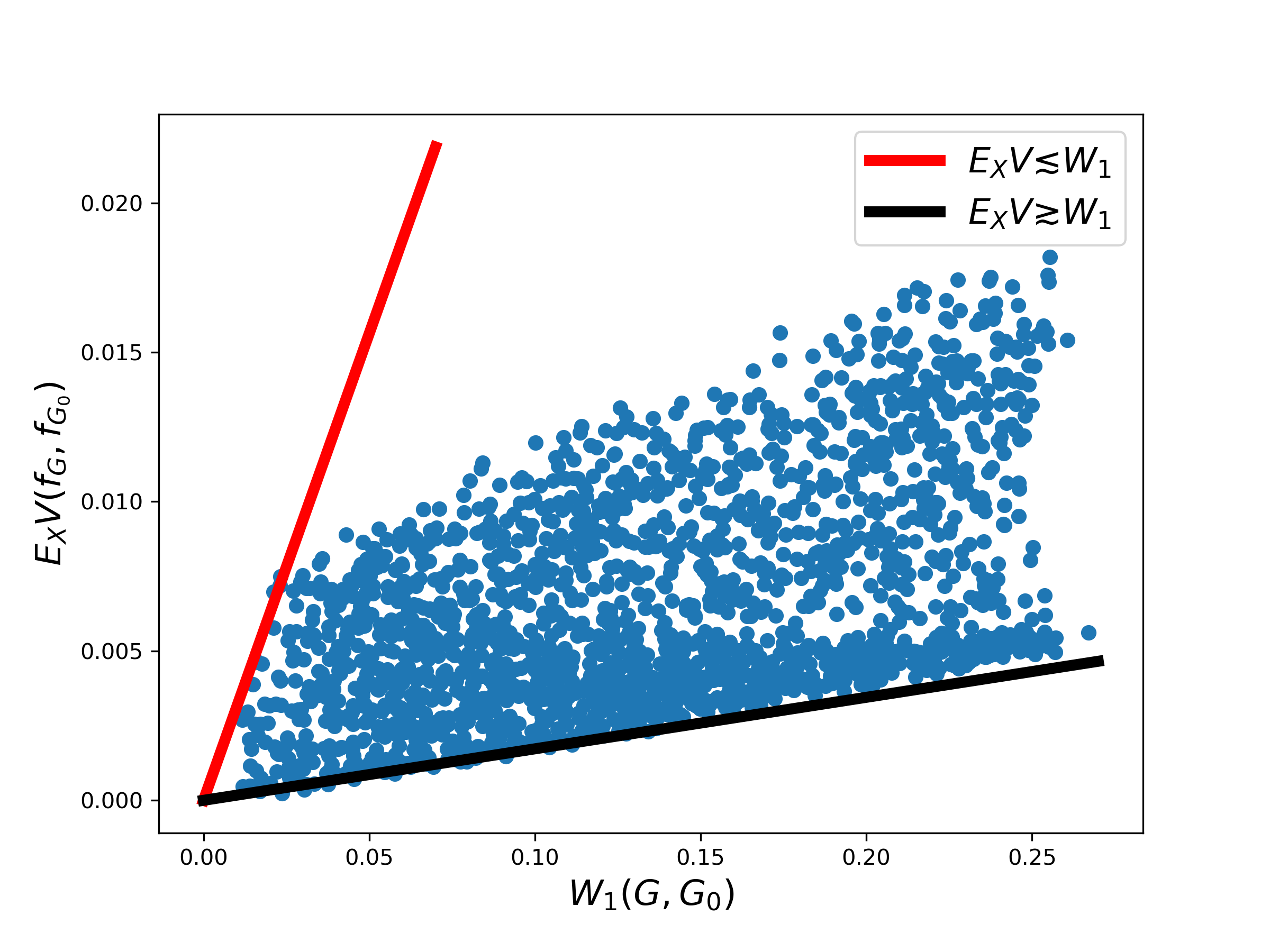}}
      \caption{Illustrations of the inverse bounds.} \label{fig:inverse-bounds}
    \end{figure}

\paragraph{Simulation studies for exact-fitted and over-fitted settings}

Next, we illustrate the parameter learning rates
under both exact and over-fitted settings. Consider a mixture of normal regression model with a polynomial link function and  fixed variance:
    $f_{G}(y|x) = \sum_{i=1}^{k} p_i \Ncal(y | h(x, \theta_i), \sigma^2)$, 
where $x \in \Rbb, G = \sum_{i=1}^{k} p_i \delta_{\theta_i}, \theta_i \in \Rbb^{3}, h(x, \theta_i) = \theta_{i1} + \theta_{i2} x + \theta_{i3} x^2$. We simulate data $(x_i)_{i=1}^n$ from an uniform distribution on $[-3, 3]$ and then generated $y_i$ given $x_i$ from this model using $G_0 = p_1^0 \delta_{\theta_1^0} + p_2^0 \delta_{\theta_2^0}$, where $\sigma=1, \theta^0_1=(1,-5,1),\theta^0_2=(2,5,2),p^0_1=p^0_2=0.5$.
Because the variances of both components are known, and the link function is a polynomial, this model is strongly identifiable in the second order. 
The maximum (conditional) likelihood estimate $\hat{G}_n$ of $G_0$ is obtained by the Expectation-Maximization (EM) algorithm 
. We considered two cases where the number of components of $G_0$ is known to be $2$ for the exact-fitted setting and set $K=3$ for the over-fitted setting. For each setting and each $n$, we run the experiment $16$ times to obtain the estimation error as measured by the Wasserstein distances. 
In Fig.~\ref{fig:exact over}, the orange line presents the interquartile range of the 16 replicates obtained at each sample size $n$. The $W_1$ error in the exact-fitted setting and the $W_2$ error in the over-fitted case are of order $(\log(n)/n)^{1/2}$ and $(\log(n)/n)^{1/4}$, respectively. These results are compatible with the theoretical results established in Theorem \ref{thm:convergence_parameter_estimation}. 
\begin{figure}[h!]
      \centering
      \subcaptionbox*{\scriptsize \centering(a) Convergence rate in exact-fitted case \par}{\includegraphics[width = 0.48\textwidth]{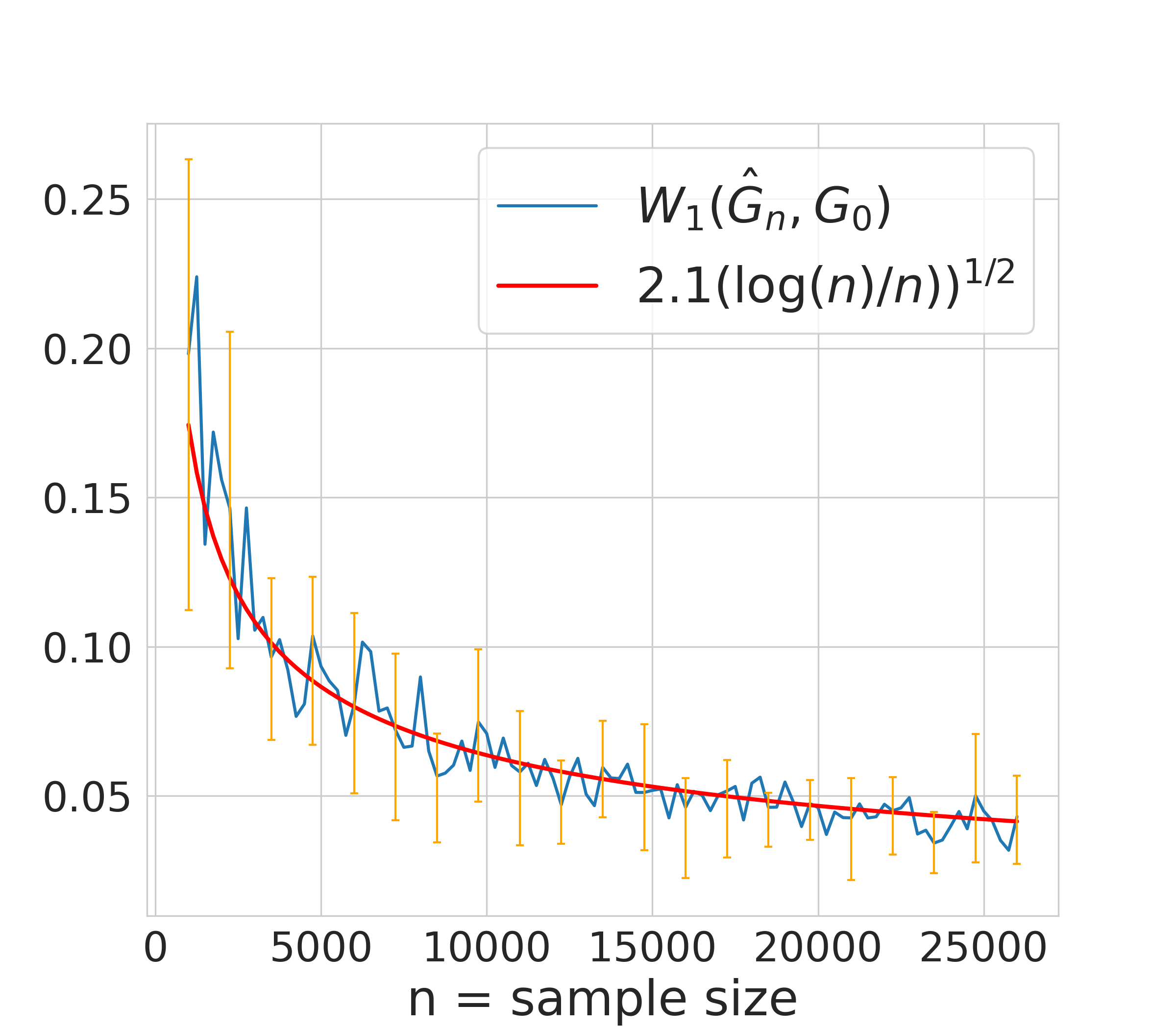}}
      \subcaptionbox*{\scriptsize \centering(b) Convergence rate in over-fitted case \par}{\includegraphics[width = 0.48\textwidth]{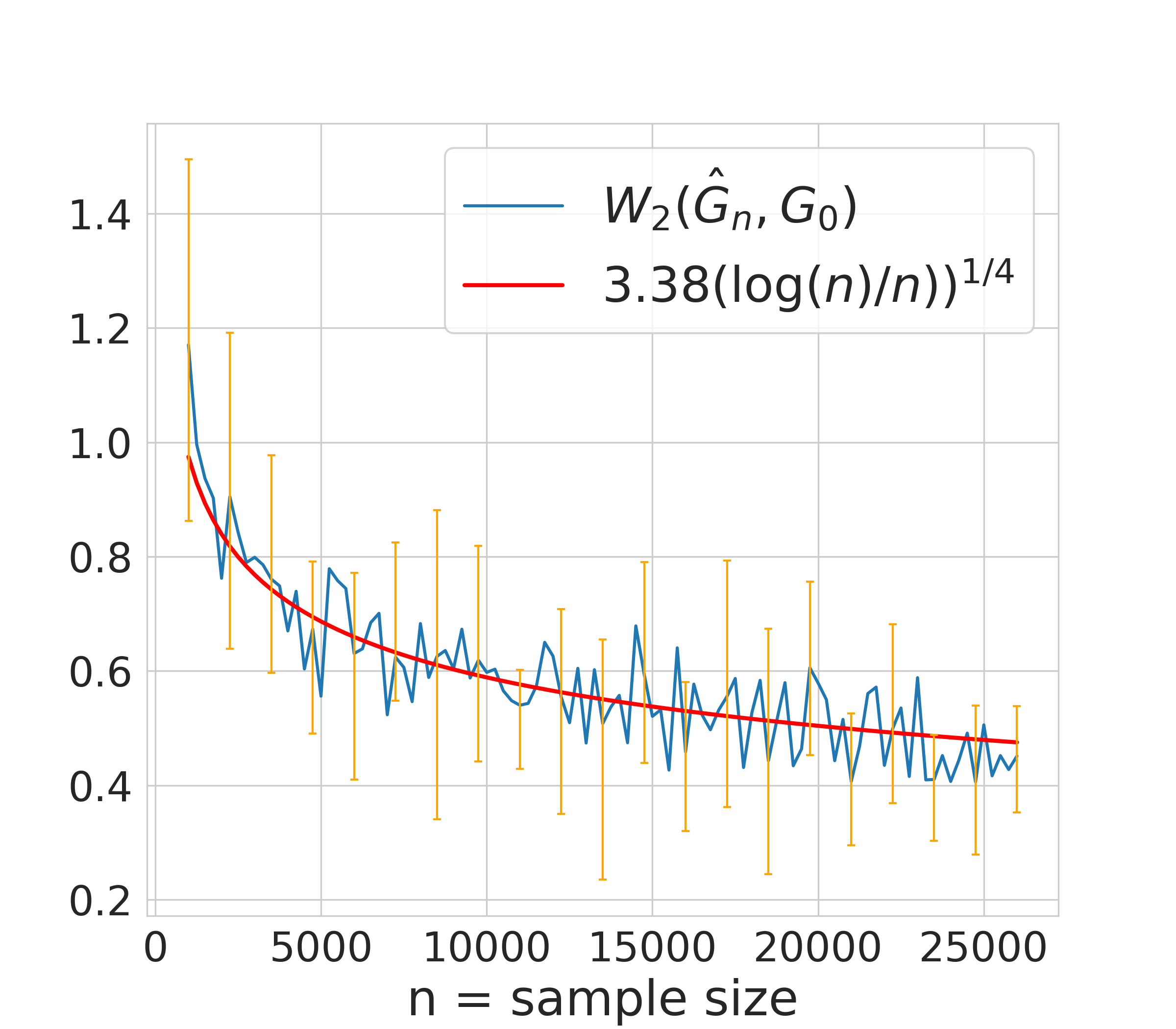}}
      \subcaptionbox*{\scriptsize(c) Comparison of the two cases \par}{\includegraphics[width = 1\textwidth]{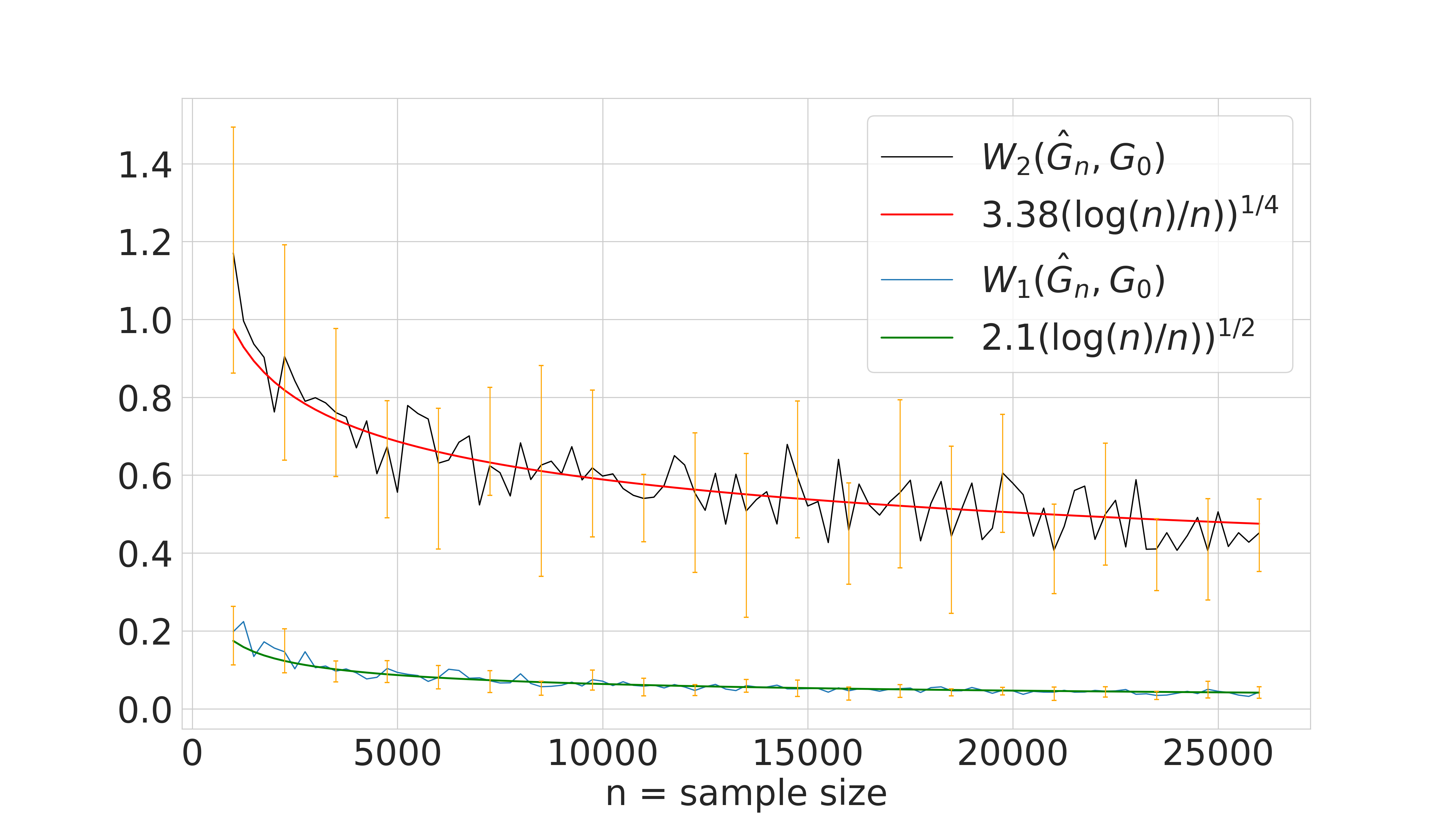}}
      \caption{\centering Convergence rates of parameters estimation in the exact-fitted and over-fitted setting.} \label{fig:exact over}
    \end{figure}

\paragraph{Investigating the lack of strong identifiability}

In Section \ref{subsec:non-strong-ident}, we discussed the lack of strong identifiability in the mixture of negative binomial regression models. Here, we conduct an experiment to show the posterior contraction rate of parameter estimation where the true model is not strongly identifiable (cf. Eq.~\eqref{eq: Non-strong Nega} holds). A dataset is drawn from a negative binomial regression mixture model:
$f_G(y|x)= \sum_{j=1}^{2}p_j\mathrm{NB}(y|h(x,\theta_j),\phi_j)$, 
where the covariate, $x_i$, is randomly generated from a uniform distribution over the interval $[0,5]$; the component means are $h(x_i, \theta_1)= \exp(X_i\theta_1)$, $h(x_i, \theta_2)= \exp(X_i\theta_2)$, where $X_i = (1,x_i), i=1,...,n $. The true latent mixing measure is $G_0 = p_1^0\delta_{\{\theta_1^0, \phi_1^0\}}+ p_2^0\delta_{\{\theta_2^0,\phi_2^0\}}$, where
the regression coefficients are $\theta_1^0 = (0,1)'$, $\theta_2^0 = (\log3,1)'$, the dispersion parameters are $\phi_1^0 = 0.5$, $\phi_2^0 = 1.5$ and the mixing proportions are $p_1^0 = 0.4$, $p_2^0 = 0.6$. Under this simulation, Eq. \eqref{eq: Non-strong Nega} is satisfied. The simulated dataset is illustrated in Figure \ref{fig:neg}(a). This model is not strongly identifiable. 

For model fitting, we adopt the Bayesian approach and consider the exact-fitted case. Let $Z_1,...Z_n$ be the indicator variables such that
$P(Z_i=j)=p_j, \mbox{ for } j=1,2$, and $p_1+p_2=1$. Then, we have $y_i|(x_i, Z_i=j) \sim \NB(h(x_i, \theta_j),\phi_j)$, for $i=1,...,n$ and $j=1,2$.
To proceed with estimating the
parameter set $\boldsymbol{\Psi}=(p_1,p_2,\theta_1, \theta_2, \phi_1, \phi_2)$ given the simulated data, $\{x_i, y_i\}_{i=1}^n$, we investigate the posterior distributions of $p= (p_1, p_2)$, $\theta_j$ and $\phi_j$ (for $j =1 ,2)$. Similar to \cite{PARK2009683}, we choose the prior $p$ to be $\textrm{Beta}(1, 1)$, $\theta_j \sim \Ncal(0, I_2)$ (normal distribution with identity covariance matrix), and $\phi_j^{-1} \sim \mathrm{Gamma}(0.01,0.01)$ (a non-informative gamma distribution) for $j=1,2$. 
%
%
The full posterior distribution is approximated using an MCMC algorithm with details given in Appendix F. 
%
For each different sample size $n$, we run the experiment 8 times. For each time running, we produce 2500 MCMC samples, discard the first 500, and use the remaining samples to estimate the expected Wasserstein distances to $G_0$.  
The estimation error averaged over the 8 runs is reported in Fig \ref{fig:neg}(b). The orange line presents an interquartile range of 8 results at each considered sample size. The $W_1$ error is plotted against the order $\dfrac{1}{\log(n)}$. It can be seen that the error tends to reduce at an extremely slow speed, as validated by the slow minimax bound. 

\begin{figure}[h!]
      \centering
      \subcaptionbox*{\scriptsize \centering(a) Simulated data from a NB mixture \par}{\includegraphics[width = 0.49\textwidth]{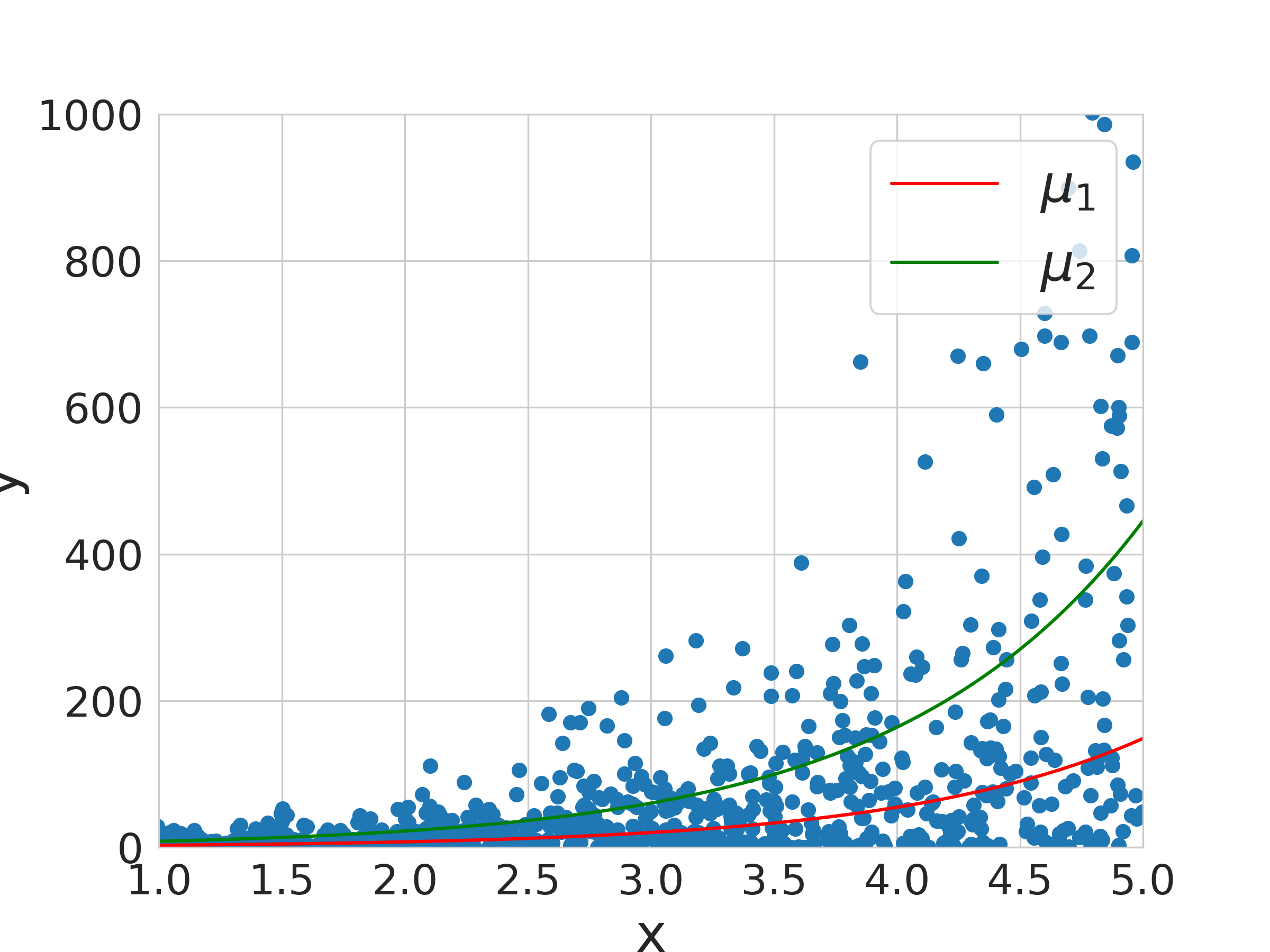}}
      \subcaptionbox*{\scriptsize \centering(b) Posterior contraction rate \par}{\includegraphics[width = 0.49\textwidth]{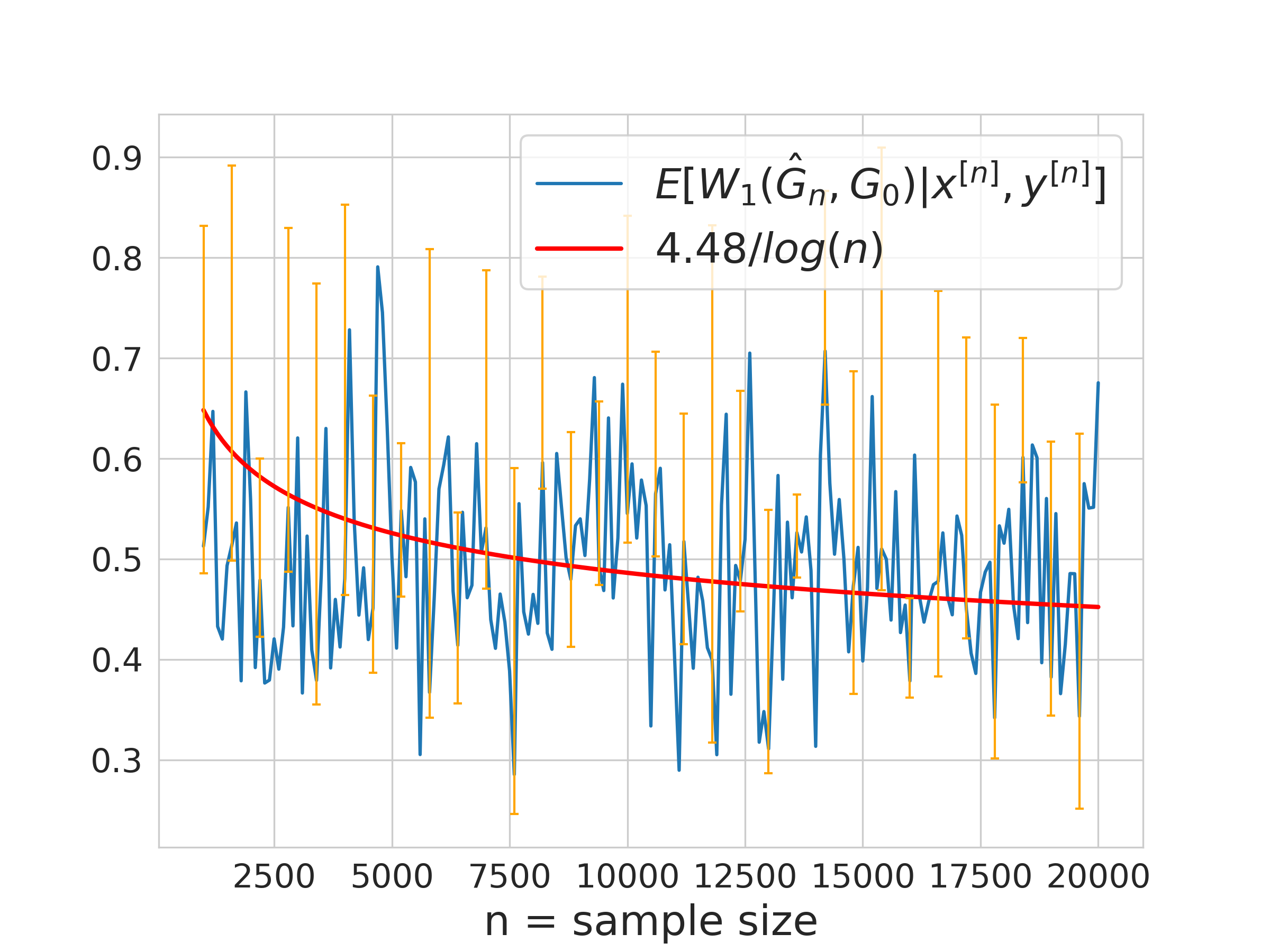}}
     \caption{\centering Convergence rate in exact-fitted case where the model is non-strongly-identifiable.} \label{fig:neg}
    \end{figure}

\paragraph{Analysis of crash data}

To validate the applicability of the proposed theoretical result in Section \ref{subsec: strong identifiability} illustrated via the simulation study above, we use the dataset collected in 1995 at urban 4-legged signalized intersections in Toronto, Canada. The same data has been explored and fitted by a mixture of negative binomial regression models by~\cite{PARK2009683}, where they showed the good quality of the dataset as well as the best-fitted model for it. This crash data set contains $868$ intersections, which have a total of $10,030$ reported crashes. In their paper, the authors explicated the heterogeneity in the dataset which came from the existence of several different sub-populations (i.e., the data collected from the different business environments, contains a mix of fixed and actuated traffic signals, and so on). 
Accordingly, 
the mean functional form has been used for each component as below:
\begin{align}
    \mu_{j,i} = h(\boldsymbol{F}_i, \theta_j) =  \theta_{j,0}F_{1i}^{\theta_{j,1}}F_{2i}^{\theta_{j,2}}, \text{ for } j=1,2 \text{ and } i=1,...,868,
\end{align}
where $\mu_{j,i}$ is the $j$th component's estimated number of crashes for intersection $i$; $F_{1i}$ counts the entering flows in vehicles/day from the major approaches at intersection $i$; $F_{2i}$ the entering flows in vehicles/day from the minor approaches at intersection $i$; and $\theta_j = (\log(\theta_{j,0}), \theta_{j,1}, \theta_{j,2})'$ the estimated regression for component $j$.
%
According to~\cite{PARK2009683}, the best model for describing the dataset is a two-mixture of negative binomial regression where $\phi_1 = 9.3692$, $\phi_2 =8.2437$, $p_1 = 0.43, p_2 = 0.57, \theta_1 = (-10.9407, 0.8588, 0.5056)'$, $\theta_2 = (-9.7842, 0.3987, 0.8703)'$. 

It can be seen that the two values of $\phi_1$ and $\phi_2$ nearly satisfy the second condition of the pathological case mentioned in Section \ref{subsec: strong identifiability} (i.e., $\phi_1 \approx \phi_2+1$). If the first condition holds (i.e.,  $\frac{\mu_{1}}{\phi_1}=\frac{\mu_{2}}{\phi_2}$), then we would be in a singular situation. To verify this, we calculate $\left(\dfrac{\mu_{1, i}}{\phi_1}-\dfrac{\mu_{2,i}}{\phi_2}\right)$ for all samples $(\boldsymbol{F}_i)_{i=1}^{868}$. The histogram of this difference can be seen in Fig.~\ref{fig:NB_real}(a).
By the Anderson-Darling test, we see that this distribution is significantly different from the degenerate distribution at $0$, with the calculated p-value for this test being $1.28\times10^{-13}$. 
Hence, we are quite far from the pathological situations of non-strong identifiability. In theory, the method should still enjoy the $n^{-1/2}$ convergence rate if the model is well-specified and exact-fitted. We further subsample this data and calculate the error of the estimator from the subsampled dataset to that of the whole dataset. For each sample size, we replicate the experiment 8 times and report the average error (in blue) and the interquartile error bar (in orange) in Fig.~\ref{fig:NB_real}(b). We can see that the error is approximate of the order $n^{-1/2}$. 

Finally, we conduct another subsampling experiment to focus on the data corresponding to $\left|\frac{\mu_{1,i}}{\phi_1}-\frac{\mu_{2,i}}{\phi_2}\right|\le0.3$. This data subset (with sample size $502$) represents a data population that is closer to the pathological cases of non-strong identifiability than that of the previous experiment. Note that the difference from the degenerate distribution at $0$ is still significant, so the $n^{-1/2}$ convergence rate is still achieved in theory. In particular, the black line in  Fig.~\ref{fig:NB_real}(b) represents the average errors in this case after an 8-time running of the experiment. A noteworthy observation is that the closer the data population is to a pathological situation, the slower the actual convergence to the true parameters will be. This result provides an interesting demonstration of the population theory given in the previous section.
\begin{figure}[h!]
      \centering
      \subcaptionbox*{\scriptsize \centering(a) Histogram of difference $\mu_1/\phi_1 - \mu_2/\phi_2$ \par}{\includegraphics[width = 0.51\textwidth]{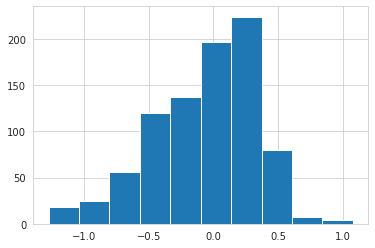}}
      \subcaptionbox*{\scriptsize(b) The blue line corresponds to the original data while the black line to subsamples chosen near a pathological situation.  \par}{\includegraphics[width = 0.48\textwidth]{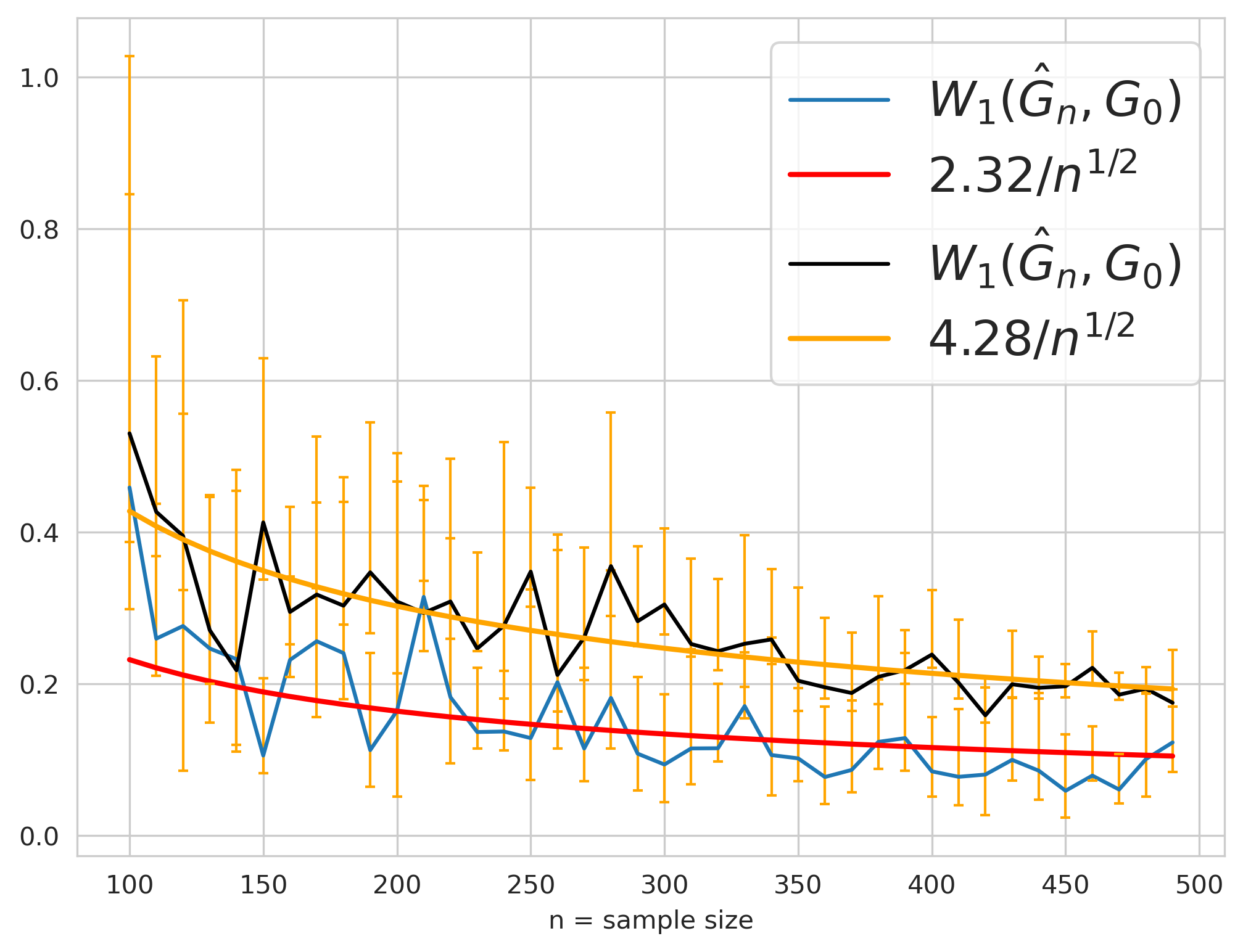}}
     \caption{\centering The impact of Crash data being near pathological cases of non-strong identifiability.} \label{fig:NB_real}
\end{figure}

\section{Conclusion}\label{sec:conclusion}

We developed a strong identifiability theory for general finite mixture of regression models and derived rates of convergence for density estimation and parameter estimation in both Bayesian and MLE frameworks.
This theory was shown to be applicable to a wide range of models employed in practice. 
It also invites interesting new questions.
First, in our models mixture weights $p_j$'s do not vary with covarate $x$. 
It would be interesting to extend the theory to the situation of co-varying weights.
Second, while our theory is applicable to the case of unknown but finite number of mixture components, it remains challenging to extend such a theory for infinite conditional mixtures motivated from Bayesian nonparametrics \cite{nguyen2013convergence,Hjort-etal-10}.
Finally, as demonstrated with the negative binomial mixtures, although the singularity situation (i.e., strong identifiability is violated) is rare, being in the vicinity of a singular model can be a far more common scenario. We would like to investigate more precisely the impact on parameter estimates when the true model is in the vicinity of a singular model, and to provide suitable statistical methods to overcome the inefficiency of inference in such situations.

\section*{Acknowledgments}
This work is supported in part by the NSF Grant DMS-2015361 and a research gift from Wells Fargo.


\bibliographystyle{plain} 
\bibliography{GLMMixture}       

\begin{thebibliography}{10}

\bibitem{bermudez2020modelling}
Llu{\'\i}s Berm{\'u}dez, Dimitris Karlis, and Isabel Morillo.
\newblock Modelling unobserved heterogeneity in claim counts using finite mixture models.
\newblock {\em Risks}, 8(1):10, 2020.

\bibitem{chen1995optimal}
Jiahua Chen.
\newblock Optimal rate of convergence for finite mixture models.
\newblock {\em The Annals of Statistics}, pages 221--233, 1995.

\bibitem{desarbo1988maximum}
Wayne~S DeSarbo and William~L Cron.
\newblock A maximum likelihood methodology for clusterwise linear regression.
\newblock {\em Journal of classification}, 5(2):249--282, 1988.

\bibitem{ding2006using}
Cody~S Ding.
\newblock Using regression mixture analysis in educational research.
\newblock {\em Practical Assessment, Research, and Evaluation}, 11(1):11, 2006.

\bibitem{do2022beyond}
Dat Do, Nhat Ho, and XuanLong Nguyen.
\newblock Beyond black box densities: Parameter learning for the deviated components.
\newblock {\em arXiv preprint arXiv:2202.02651}, 2022.

\bibitem{doob1949application}
Joseph~L Doob.
\newblock Application of the theory of martingales.
\newblock {\em Le calcul des probabilites et ses applications}, pages 23--27, 1949.

\bibitem{ghosal2017fundamentals}
Subhashis Ghosal and Aad Van~der Vaart.
\newblock {\em Fundamentals of nonparametric Bayesian inference}, volume~44.
\newblock Cambridge University Press, 2017.

\bibitem{ghosal2001entropies}
Subhashis Ghosal and Aad~W Van Der~Vaart.
\newblock Entropies and rates of convergence for maximum likelihood and bayes estimation for mixtures of normal densities.
\newblock {\em The Annals of Statistics}, 29(5):1233--1263, 2001.

\bibitem{gine2021mathematical}
Evarist Gin{\'e} and Richard Nickl.
\newblock {\em Mathematical foundations of infinite-dimensional statistical models}.
\newblock Cambridge university press, 2021.

\bibitem{GOLDFELD19733}
Stephen~M. Goldfeld and Richard~E. Quandt.
\newblock A markov model for switching regressions.
\newblock {\em Journal of Econometrics}, 1(1):3--15, 1973.

\bibitem{grun2008finite}
Bettina Gr{\"u}n and Friedrich Leisch.
\newblock Finite mixtures of generalized linear regression models.
\newblock In {\em Recent advances in linear models and related areas}, pages 205--230. Springer, 2008.

\bibitem{grun2008identifiability}
Bettina Gr{\"u}n and Friedrich Leisch.
\newblock Identifiability of finite mixtures of multinomial logit models with varying and fixed effects.
\newblock {\em Journal of classification}, 25(2):225--247, 2008.

\bibitem{hennig2000identifiablity}
Hennig and Christian.
\newblock Identifiablity of models for clusterwise linear regression.
\newblock {\em Journal of classification}, 17(2), 2000.

\bibitem{Hjort-etal-10}
N.~Hjort, C.~Holmes, P.~Mueller, and S.~Walker.
\newblock {\em Bayesian Nonparametrics: Principles and Practice}.
\newblock Cambridge University Press, 2010.

\bibitem{Ho-Nguyen-Ann-16}
N.~Ho and X.~Nguyen.
\newblock Convergence rates of parameter estimation for some weakly identifiable finite mixtures.
\newblock {\em Annals of Statistics}, 44:2726--2755, 2016.

\bibitem{Ho-Nguyen-EJS-16}
N.~Ho and X.~Nguyen.
\newblock On strong identifiability and convergence rates of parameter estimation in finite mixtures.
\newblock {\em Electronic Journal of Statistics}, 10:271--307, 2016.

\bibitem{ho2019singularity}
Nhat Ho and XuanLong Nguyen.
\newblock Singularity structures and impacts on parameter estimation in finite mixtures of distributions.
\newblock {\em SIAM Journal on Mathematics of Data Science}, 1(4):730--758, 2019.

\bibitem{ho2019convergence}
Nhat Ho, Chiao-Yu Yang, and Michael~I Jordan.
\newblock Convergence rates for gaussian mixtures of experts.
\newblock {\em arXiv preprint arXiv:1907.04377}, 2019.

\bibitem{hurn2003estimating}
Merrilee Hurn, Ana Justel, and Christian~P Robert.
\newblock Estimating mixtures of regressions.
\newblock {\em Journal of computational and graphical statistics}, 12(1):55--79, 2003.

\bibitem{jacobs1991adaptive}
Robert~A Jacobs, Michael~I Jordan, Steven~J Nowlan, and Geoffrey~E Hinton.
\newblock Adaptive mixtures of local experts.
\newblock {\em Neural computation}, 3(1):79--87, 1991.

\bibitem{jaki2019effects}
Thomas Jaki, Minjung Kim, Andrea Lamont, Melissa George, Chi Chang, Daniel Feaster, and M~Lee Van~Horn.
\newblock The effects of sample size on the estimation of regression mixture models.
\newblock {\em Educational and Psychological Measurement}, 79(2):358--384, 2019.

\bibitem{jiang1999identifiability}
Wenxin Jiang and Martin~A Tanner.
\newblock On the identifiability of mixtures-of-experts.
\newblock {\em Neural Networks}, 12(9):1253--1258, 1999.

\bibitem{khalili2007variable}
Abbas Khalili and Jiahua Chen.
\newblock Variable selection in finite mixture of regression models.
\newblock {\em Journal of the american Statistical association}, 102(479):1025--1038, 2007.

\bibitem{Miller-2016}
J.~W. Miller and M.~T. Harrison.
\newblock Mixture models with a prior on the number of components.
\newblock {\em Journal of the American Statistical Association}, 113, 2018.

\bibitem{miller2018detailed}
Jeffrey~W Miller.
\newblock A detailed treatment of doob's theorem.
\newblock {\em arXiv preprint arXiv:1801.03122}, 2018.

\bibitem{mufudza2016poisson}
Chipo Mufudza and Hamza Erol.
\newblock Poisson mixture regression models for heart disease prediction.
\newblock {\em Computational and Mathematical Methods in Medicine}, 2016, 2016.

\bibitem{nguyen2013convergence}
XuanLong Nguyen.
\newblock Convergence of latent mixing measures in finite and infinite mixture models.
\newblock {\em The Annals of Statistics}, 41(1):370--400, 2013.

\bibitem{PARK2009683}
Byung-Jung Park and Dominique Lord.
\newblock Application of finite mixture models for vehicle crash data analysis.
\newblock {\em Accident Analysis and Prevention}, 41(4):683--691, 2009.

\bibitem{park2010bias}
Byung-Jung Park, Dominique Lord, and Jeffrey~D Hart.
\newblock Bias properties of bayesian statistics in finite mixture of negative binomial regression models in crash data analysis.
\newblock {\em Accident Analysis and Prevention}, 42(2):741--749, 2010.

\bibitem{Rai1993}
S.~N. Rai and D.~E. Matthews.
\newblock Improving the em algorithm.
\newblock {\em Biometrics}, 49(2):587--591, 1993.

\bibitem{sarstedt2008model}
Marko Sarstedt and Manfred Schwaiger.
\newblock Model selection in mixture regression analysis--a monte carlo simulation study.
\newblock In {\em Data analysis, machine learning and applications}, pages 61--68. Springer, 2008.

\bibitem{spath1979algorithm}
Helmuth Sp{\"a}th.
\newblock Algorithm 39 clusterwise linear regression.
\newblock {\em Computing}, 22(4):367--373, 1979.

\bibitem{stein2010complex}
Elias~M Stein and Rami Shakarchi.
\newblock {\em Complex analysis}, volume~2.
\newblock Princeton University Press, 2010.

\bibitem{teicher1963identifiability}
Henry Teicher.
\newblock Identifiability of finite mixtures.
\newblock {\em The annals of Mathematical statistics}, pages 1265--1269, 1963.

\bibitem{Vandegeer}
S.~van~de Geer.
\newblock {\em Empirical Processes in M-estimation}.
\newblock Cambridge University Press, 2000.

\bibitem{viele2002modeling}
Kert Viele and Barbara Tong.
\newblock Modeling with mixtures of linear regressions.
\newblock {\em Statistics and Computing}, 12(4):315--330, 2002.

\bibitem{wang2014note}
Shaoli Wang, Weixin Yao, and Mian Huang.
\newblock A note on the identifiability of nonparametric and semiparametric mixtures of glms.
\newblock {\em Statistics and Probability Letters}, 93:41--45, 2014.

\bibitem{wei2020convergence}
Yun Wei and XuanLong Nguyen.
\newblock Convergence of de finetti's mixing measure in latent structure models for observed exchangeable sequences.
\newblock {\em arXiv preprint arXiv:2004.05542}, 2020.

\bibitem{wong1995probability}
Wing~Hung Wong and Xiaotong Shen.
\newblock Probability inequalities for likelihood ratios and convergence rates of sieve mles.
\newblock {\em The Annals of Statistics}, pages 339--362, 1995.

\bibitem{yu1997assouad}
Bin Yu.
\newblock Assouad, fano, and le cam.
\newblock In {\em Festschrift for Lucien Le Cam}, pages 423--435. Springer, 1997.

\end{thebibliography}
\newpage
\appendix
\begin{center}
{\bf \large Supplement to "Strong identifiability and parameter learning
in regression with heterogeneous response"}
\end{center}
In the supplementary material, we collect the proofs and additional information deferred from the main text. Section~\ref{sec:proof} includes the proofs for all main theoretical results. Section~\ref{sec:remain-results} includes the proofs for all remaining theoretical results. Section~\ref{sec:conv-rate-m-estimator-theory} and Section~\ref{sec:posterior-contraction-rate-theory} provide a general theory of M-estimators convergence rates and posterior contraction rates in the regression setting, respectively. 
The EM algorithms for mixtures of regression are described in Section~\ref{sec:EM}. 


\section{Proofs of main results}\label{sec:proof}
\subsection{Identifiablity conditions and inverse bounds}\label{subsec:ident-proofs}
\begin{proof}[Proof of Theorem~\ref{thm:identifiable-equivalent}]
We will divide the proof into first order and second order identifiability cases. 
\paragraph{The case $r=1$:} For some $k \in \mathbb{N}$, suppose that there exist $k$ distinct elements
$(\theta_{11},\theta_{21}), \ldots, (\theta_{1k},\theta_{2k})$ 
$\in \Theta_1\times \Theta_2$, and $\alpha_j\in \Rbb, \beta_j\in \Rbb^{d_1}, \gamma_j\in \Rbb^{d_2}$ as $j = 1,\dots, K$ such that for almost all $x, y$ (w.r.t. $\PX \times \nu$)
\begin{eqnarray}
\sum_{j=1}^{k} \alpha_j f_j(y|x) + \beta_j^{\top} \dfrac{\partial}{\partial \theta_1} f_j(y|x)+  \gamma_j^{\top} \dfrac{\partial}{\partial \theta_2} f_j(y|x) = 0, \nonumber
\end{eqnarray}
then we want to show that $\alpha_j=0, \beta_j=0\in \Rbb^{d_1}, \gamma_j = 0 \in \Rbb^{d_2}$ for $j = 1, \dots, k$. Indeed, by the chain rule, the left-hand side of the above expression is equal to
{\fontsize{10}{10}\selectfont 
\begin{align*}
    \sum_{j=1}^{k} \alpha_j f_j(y|x) + \beta_j^{\top} \dfrac{\partial h_1(x, \theta_{1j})}{\theta_{1j}} \dfrac{\partial}{\partial \mu} f(y|h_1(x, \theta_{1j}), h_2(x, \theta_{2j})) \\
    + \gamma_j^{\top} \dfrac{\partial h_2(x, \theta_{2j})}{\partial \theta_{2j}} \dfrac{\partial}{\partial \phi} f(y|h_1(x, \theta_{1j}), h_2(x, \theta_{2j})).
\end{align*}}
Because of the identifiability of $(h_1, h_2)$ (condition (A4.)), there exists a set $A$ with $\Pbb_X(A) > 0$ such that $(h_1(x, \theta_{1j}), h_2(x, \theta_{2j}))_{j=1}^{k}$ are distinct. Therefore, by the first order identifiability of $f$, we have 
\begin{align*}
    \beta_j^{\top} \dfrac{\partial h_1(x, \theta_1)}{\partial \theta_1} = 0,\quad \gamma_j^{\top} \dfrac{\partial h_2(x, \theta_2)}{\partial \theta_2} = 0,
\end{align*}
for all $x\in A$ (possibly except a $\PX$ zero-measure set). Hence, by condition (A5.), we have $\beta_j = 0, \gamma_j = 0$ for all $j = 1,\dots, k$. This concludes the first-order identifiability of the family of conditional densities $f(\cdot|h_1, h_2)$.  
\paragraph{The case $r=2$:} For any $k, s \geq 1$, given $k$ distinct elements
$(\theta_{11},\theta_{21}), \dots$, $ (\theta_{1k},\theta_{2k}) \in \Theta_1\times \Theta_2$, if there exist $\alpha_j\in \Rbb, \beta_j\in \Rbb^{d_1}, \gamma_j\in \Rbb^{d_2}$, and $\rho_{jt} \in \Rbb^{d_1}, \nu_{jt}\in \Rbb^{d_2}$ as $t = 1, \dots, s_j$ and $j = 1,\dots, k$ such that for almost all $x, y$ (w.r.t. $\PX \times \nu$)
\begin{eqnarray}
\sum_{j=1}^{k} \alpha_j f_j(y|x) + \beta_j^{\top} \dfrac{\partial}{\partial \theta_1} f_j(y|x) + \gamma_j^{\top} \dfrac{\partial}{\partial \theta_2} f_j(y|x) + \sum_{t=1}^{s_j} \left(\rho_{jt}^{\top} \dfrac{\partial }{\partial \theta_1^2} f_j(y|x) \rho_{jt}\right) \nonumber\\
\vspace{-3cm} + \sum_{t=1}^{s_j} \left(\nu_{jt}^{\top} \dfrac{\partial }{\partial \theta_2^2} f_j(y|x) \nu_{jt}\right) + \sum_{t=1}^{s_j} \left(\rho_{jt}^{\top} \dfrac{\partial }{\partial \theta_1 \partial \theta_2} f_j(y|x) \nu_{jt}\right)  = 0\nonumber,
\end{eqnarray}
then we want to show that $\alpha_j=0, \beta_j = \rho_{jt} = 0\in \Rbb^{d_1}, \gamma_j = \nu_{jt} = 0 \in \Rbb^{d_2}$ for $j = 1, \dots, k, t = 1, \dots, s$. Indeed, again, by the chain rule, the expression above is equivalent to
\begin{align}
\sum_{j=1}^{k} \alpha_j f_j(y|x) + \left(\beta_j^{\top} \dfrac{\partial h_1(x, \theta_1)}{\partial \theta_1} + \sum_{t=1}^{s_j} \rho_{jt}^{\top} \dfrac{\partial^2 h_1(x, \theta_{1j})}{\partial \theta_1^2} \rho_{jt} \right) \dfrac{\partial }{\partial \mu} f_j(y|x)  \nonumber \\ + \sum_{t=1}^{s_j} \left(\rho_{jt}^{\top} \dfrac{\partial h_1(x, \theta_{1j})}{\partial \theta_1} \right)^2 \dfrac{\partial }{\partial \mu^2} f_j(y|x) \nonumber \\ + \left(\gamma_j^{\top} \dfrac{\partial h_2(x, \theta_2)}{\partial \theta_2} + \sum_{t=1}^{s_j} \nu_{jt}^{\top} \dfrac{\partial^2 h_2(x, \theta_{2j})}{\partial \theta_2^2} \nu_{jt} \right) \dfrac{\partial }{\partial \phi} f_j(y|x)\nonumber\\
\vspace{-2cm}   + \sum_{t=1}^{s_j} \left(\nu_{jt}^{\top} \dfrac{\partial h_2(x, \theta_{2j})}{\partial \theta_2} \right)^2 \dfrac{\partial^2 }{\partial \phi^2} f_j(y|x) \nonumber \\ 
+ 2\sum_{t=1}^{s_j} \left(\rho_{jt}^{\top} \dfrac{\partial^2 }{\partial \theta_1 \partial \theta_2} f_j(y|x) \nu_{jt}\right) \dfrac{\partial^2 }{\partial \mu\partial \phi} f_j(y|x) = 0.
\end{align}
Because of the identifiability of $(h_1, h_2)$ (condition (A4.)), there exists a set $A$ with $\Pbb_X(A) > 0$ such that $(h_1(x, \theta_{1j}), h_2(x, \theta_{2j}))_{j=1}^{k}$ are distinct. By the second-order identifiability of $f$, then we have 
\begin{align*}
    \alpha_j = 0, \nonumber \\
    \beta_j^{\top} \dfrac{\partial h_1(x, \theta_1)}{\theta_1} + \sum_{t=1}^{s_j} \rho_{jt}^{\top} \dfrac{\partial^2 h_1(x, \theta_{1j})}{\partial \theta_1^2} \rho_{jt}= 0,
    \quad 
    \sum_{t=1}^{s_j} \left(\rho_{jt}^{\top} \dfrac{\partial h_1(x, \theta_{1j})}{\partial \theta_1} \right)^2 = 0,\\
    \gamma_j^{\top} \dfrac{\partial h_2(x, \theta_2)}{\theta_2} + \sum_{t=1}^{s_j} \nu_{jt}^{\top} \dfrac{\partial^2 h_2(x, \theta_{2j})}{\partial \theta_2^2} \nu_{jt}= 0,
    \quad
    \sum_{t=1}^{s_j} \left(\nu_{jt}^{\top} \dfrac{\partial h_2(x, \theta_{2j})}{\partial \theta_2} \right)^2 = 0,\\
    \sum_{t=1}^{s_j} \left(\rho_{jt}^{\top} \dfrac{\partial^2 }{\partial \theta_1 \partial \theta_2} f_j(y|x) \nu_{jt}\right) = 0,
\end{align*}
for all $x\in A$, possibly except a $\PX$ zero-measure set. Hence, by condition (A5.), from the third and fifth equation above, we have $\alpha_j = 0, \rho_{jt} = 0, \nu_{jt} = 0$ for all $t=1, \dots, s_j, j = 1,\dots, k$. These together with the second and fourth equation further imply that $\beta_j =0, \gamma_j = 0$ for all $j = 1,\dots, k$. We arrive at the second-order identifiability of family of conditional densities $f(y|x)$ as desired.
\end{proof}

\begin{proof} [Proof of Theorem~\ref{thm:inverse-bounds}]

In order to prove part (a) of the theorem, we divide it into two regimes, local and global, and establish two following claims: for any $\epsilon' > 0$, 
\begin{equation}
	\inf \limits_{G \in \Ecal_{k_{0}} ( \Theta): W_{1}( G,   G_{0}) > \epsilon'} \dfrac{\Ebb_X d_{TV}(f_{ G}(\cdot|X), f_{  G_{0}}(\cdot|X) )}{W_{1}( G,   G_{0})} > 0, \label{eq:claim_global_exact}
\end{equation}

\begin{equation}
    \lim \limits_{\epsilon \to 0} \inf \limits_{G \in \Ecal_{k_{0}}( \Theta)}{\left\{\dfrac{\Ebb_X d_{TV}(f_{ G}(\cdot|X), f_{  G_{0}}(\cdot|X) )}{W_{1}( G, G_{0})}: \ W_{1}( G, G_{0}) \leq \epsilon\right\}} > 0. \label{eq:claim_local_exact}
\end{equation}
\paragraph{Proof of claim~\eqref{eq:claim_global_exact}:} Suppose that this is not true, that is, there exists a sequence $G_n \in \Ecal_k(\Theta)$ such that $W_1(G_n, G_0) > \epsilon'$ for all $n$ and as $n\rightarrow \infty$,
\begin{equation*}
    \dfrac{\Ebb_X d_{TV}(f_{G_n}(\cdot|X), f_{G_{0}}(\cdot|X) )}{W_1(G_n, G_0)} \to 0.
\end{equation*}
This implies $\Ebb_X d_{TV}(f_{G_n}(\cdot|X), f_{G_{0}}(\cdot|X) )\to 0$. Since $\Delta^{k_0-1}$, $\Theta_1$ and $\Theta_2$ are compact, there exists a subsequence of $(G_n)_n$ (which is assumed to be $(G_n)_n$ itself without loss of generality) that converges weakly to some $G' = \sum_{j=1}^{k'} p_j' \delta_{(\theta_{1j}', \theta_{2j}')} \in \Ocal_{k_0}(\Theta_1 \times \Theta_2)$, where $(\theta_{11}', \theta_{21}'), \dots, (\theta_{1k'}', \theta_{2k'}')$ are distinct. Hence, $W_1(G', G_0) > \epsilon'$, and $\Ebb_X d_{TV}(f_{G'}(\cdot|X), f_{G_0}(\cdot|X)) = 0$. Because $d_{TV}(f_{G'}(\cdot|x), f_{G_0}(\cdot|x))\geq 0$ for all $x$, this implies
\begin{equation*}
    f_{G'}(y|x) = f_{G_0}(y|x)\quad \text{a.s. in }x, y,
\end{equation*}
which means
\begin{equation*}
    \sum_{j=1}^{k'} p_j' f(y|h_1(x, \theta_{1j}'), h_2(x, \theta_{2j}')) = \sum_{j=1}^{k_0} p_j^0 f(y|h_1(x, \theta_{2j}^0), h_2(x, \theta_{2j}^0)) \quad \text{a.s. in }x, y.
\end{equation*}
By the zero-order identifiability of $f$, we conclude that
\begin{equation}\label{eq:global-ident-contradiction}
    \sum_{j=1}^{k'} p_j' \delta_{(h_1(x, \theta_{1j}'), h_2(x, \theta_{2j}'))} = \sum_{j=1}^{k_0} p_j^0 \delta_{(h_1(x, \theta_{1j}^0), h_2(x, \theta_{2j}^0))} \quad \text{a.s. in }x.
\end{equation}
If family $(h_1, h_2)$ is identifiable, then we argue that the set $\{(\theta_{11}', \theta_{21}'), \dots, (\theta_{1k'}', \theta_{1k'}')\}$ must equal $\{(\theta_{11}^0, \theta_{21}^0), \dots, (\theta_{1k_0}^0, \theta_{2k_0}^{0})\}$. Otherwise, we can assume (without loss of generality) that $(\theta_{11}', \theta_{21}') \not \in \{(\theta_{11}^0, \theta_{21}^0), \dots, (\theta_{1k_0}^0, \theta_{2k_0}^0)\}$, which means there exists a set $A$ with $\PX(A) > 0$ such that 
$(h_1(x, \theta_{11}'),$ $ h_2(x, \theta_{21}'))$ is distinct from any pair among $(h_1(x, \theta_{11}^0), h_2(x, \theta_{21}^0)) \dots, (h_1(x, \theta_{1k_0}^0), h_2(x, \theta_{2k_0}^0))$ for all $x\in A$. But this contradicts~\eqref{eq:global-ident-contradiction}, because the left-hand side put a positive weight to the atom $(h(x, \theta_{11}'), h(x, \theta_{21}'))$, which the right-hand side does not have. Hence, we have the set $\{(\theta_{11}', \theta_{21}'), \dots, (\theta_{1k_0}', \theta_{1k_0}')\}$ equals $\{(\theta_{11}^0, \theta_{21}^0), \dots, (\theta_{1k_0}^0, \theta_{2k_0}^{0})\}$. Now, without loss of generality, we can assign $(\theta_{11}', \theta_{21}') = (\theta_{11}^0, \theta_{21}^0), \dots, (\theta_{1k_0}', \theta_{2k_0}') = (\theta_{1k_0}^0, \theta_{2k_0}^0)$. Because $(\theta_{11}^0, \theta_{21}^0), \dots, (\theta_{1k_0}^0, \theta_{2k_0}^0)$ are distinct, there exists $x'$ such that $$(h_1(x', \theta_{11}^0), h_2(x', \theta_{21}^0)), \dots, (h_1(x', \theta_{1k_0}^0), h_2(x', \theta_{2k_0}^0))$$ are distinct, which together with Eq. \eqref{eq:global-ident-contradiction} entail that $p_i' = p_i^0$ for all $i\in \{1, \dots, k_0\}$. Thus, $G' = G_0$, while $W_1(G', G_0) > \epsilon'$, a contradiction. 
Hence, claim~\eqref{eq:claim_global_exact} is proved.

\paragraph{Proof of claim~\eqref{eq:claim_local_exact}:} Suppose this does not hold. Then there exists a sequence $G_n \in \Ecal_{k_0}(\Theta)$ such that 
\begin{equation}\label{eq:contradiction_exact}
    W_1(G_n, G_0) \to 0,\quad \dfrac{\Ebb_X d_{TV}(f_{G_n}(\cdot|X), f_{G_{0}}(\cdot|X) )}{W_1(G_n, G_0)} \to 0.
\end{equation}
We can relabel the atoms and weights of $G_{n}$ such that it admits the following form:
\begin{align}
	G_{n} = \sum_{j = 1}^{k_{0}} p_{j}^{n} \delta_{(\theta_{1j}^{n}, \theta_{2j}^{n})}, \label{eq:relabel_measure}
\end{align}
where $p_{j}^{n} \to p_{j}^{0}$, $\theta_{j}^{n} \to \theta_{j}^{0}$ and $\theta_{2j}^{n} \to \theta_{2j}^{0}$ for all $i \in [k_{0}]$. To ease the ensuing presentation, we denote $\Delta p_{j}^{n} : = p_{j}^{n} - p_{j}^{0}$, $\Delta \theta_{1i}^{n} : = \theta_{1j}^{n} - \theta_{1j}^{0}$ and $\Delta \theta_{2j}^{n} : = \theta_{2j}^{n} - \theta_{2j}^{0}$ for $i \in [k_{0}]$. Then, using the coupling between $G_n$ and $G_0$ such that it put mass $\min\{p_i^n, p_i^0\}$ on $\delta_{((\theta_{1j}^n, \theta_{2j}^n), (\theta_{1j}^0, \theta_{2j}^0))}$, we can verify that 
\begin{align}
	W_{1}( G_{n}, G_{0}) \preccurlyeq \sum_{i = 1}^{k_{0}} \abss{ \Delta p_{j}^{n}} + p_{j}^{n} (\norm{\Delta \theta_{1j}^{n}} + \norm{\Delta \theta_{2j}^{n}}) =: D_1(G_n, G_0). \label{eq:wasserstein_1_equivalence}
\end{align} 
The remainder of the proof is composed of three steps.
\paragraph{Step 1} (Taylor expansion) To ease the notation, we write for short
\begin{align*}
    f_j^0(y|x) = f(y| h_1(x,\theta_{1j}^{0}), h_2(x,\theta_{2j}^{0})), 
\end{align*}
and 
\begin{align*}
    f_j^n(y|x) = f(y| h_1(x,\theta_{1j}^{n}), h_2(x,\theta_{2j}^{n})),
\end{align*}
for all $j = 1,\dots, k_0$. Because $f(y|h_1(x,\theta_1), h_2(x, \theta_2))$ is differentiable with respect to $\theta$ for all $x,y$, by applying Taylor expansion up to the first order and the chain rule, we find that for all $j = 1,\dots, k_0$,
\begin{align*}
	f_{j}^{n}(y|x) - f_{j}^{0}(y|x) & = (\Delta \theta_{1j}^{n})^{ \top} \frac{\partial{}}{\partial{ \theta_1}} f_j^0(y|x) + (\Delta \theta_{2j}^{n})^{ \top} \frac{\partial{}}{\partial{ \theta_2}} f_j^{0}(y|x) + R_{j}(x, y),
\end{align*}
where $R_{j}(x)$ is Taylor remainder such that $R_{j}(x, y) = o ( \norm{\Delta \theta_{1j}^{n}}+ \norm{\Delta \theta_{2j}^{n}})$ for $i \in [k_{0}]$. Combine the above expression for $j = 1,\dots, k_0$, we have
\begin{align}
	f_{G_{n}}(y|x) - f_{G_{0}}(y|x) & =  \sum_{j = 1}^{k_{0}} \parenth{ \Delta p_{j}^{n} } f_{j}^{0}(y|x) + p_{j}^{n} (\Delta \theta_{1j}^{n})^{ \top} \frac{\partial{}}{\partial{ \theta_1}} f_{j}^{0}(y|x) \nonumber \\
    & + p_{j}^{n} (\Delta \theta_{2j}^{n})^{ \top}\frac{\partial{}}{\partial{ \theta_2}} f_{j}^{0}(y|x) + R(x, y), \label{eq:Taylor_exact}
\end{align}
where $R(x, y) = \sum_{i = 1}^{n} p_{j}^{n} R_{i}(x,y) = o \parenth{ \sum_{i = 1}^{k_{0}} p_{j}^{n} \left(\norm{ \Delta \theta_{1j}^{n}} + \norm{ \Delta \theta_{2j}^{n}} \right)}$. From Eq.~\eqref{eq:wasserstein_1_equivalence}, we have $R(x, y)/ {D}_{1}( G_{n},  G_{0}) \to 0$ as $n \to \infty$ for all $x,y$. 

\paragraph{Step 2} (Extracting non-vanishing coefficients) From Eq.~\eqref{eq:contradiction_exact} and \eqref{eq:wasserstein_1_equivalence}, we have that
\begin{equation}
    \dfrac{\Ebb_X d_{TV}(f_{G_n}(\cdot|X), f_{G_{0}}(\cdot|X) )}{D_1(G_n, G_0)} \to 0 \quad (n\to \infty). 
\end{equation}
We can write 
\begin{align}
	\dfrac{f_{G_{n}}(y|x) - f_{G_{0}}(y|x)}{D_1(G_n, G_0)} & = \sum_{j = 1}^{k_{0}} \alpha_{j}^{n} f_{j}^{0}(y|x) + (\beta_{j}^{n})^{ \top} \frac{\partial{}}{\partial{ \theta_1}} f_{j}^{0}(y|x) +  (\gamma_i^{n})^{\top} \frac{\partial{}}{\partial{ \theta_2}} f_{j}^{0}(y|x) \nonumber \\
    & + \dfrac{R(x, y)}{{D_1(G_n, G_0)}},
\end{align}
where $\alpha_i^{n} = \dfrac{\parenth{ \Delta p_{j}^{n}}}{D_1(G_n, G_0)}\in \mathbb{R}$, $\beta_i^{n} =\dfrac{p_{j}^{n} (\Delta \theta_{1j}^{n})}{{D_1(G_n, G_0)}}\in \mathbb{R}^{d_1}$ and $\gamma_i^{n} =\dfrac{p_{j}^{n} (\Delta \theta_{2j}^{n})}{{D_1(G_n, G_0)}}\in \mathbb{R}^{d_2}$. From the definition of $D_1(G_n, G_0)$, we have 
\begin{equation*}
    \sum_{i=1}^{k_0} |\alpha_i^n| + \sum_{i=1}^{k_0} \|\beta_i^n\| + \sum_{i=1}^{k_0} \|\gamma_i^n\| = 1, 
\end{equation*}
therefore $(\alpha_i^n)$ is in $[-1,1]$, $(\beta_i^n)$ is in $[-1,1]^{d_1}$ and $(\gamma_i^n)$ is in $[-1,1]^{d_2}$, by compactness of those sets, there exist subsequences of $(\alpha_i^n)$, $\beta_i^n$ and $\gamma_i^n$ (without loss of generality, we assume it is the whole sequence itself) such that $\alpha_i^n \to \alpha_i\in [-1,1], \beta_i^n \to \beta_i\in [-1,1]^d$ as $n\to \infty$ for all $i=1,\dots, k_0$. As $\sum_{i=1}^{k_0} |\alpha_i| + \sum_{i=1}^{k_0} \|\beta_i\| + \sum_{i=1}^{k_0} \|\gamma_i\|  = 1$, at least one of them is not zero.
\paragraph{Step 3} (Deriving contradiction via Fatou's lemma) By Fatou's lemma, we have
\begin{align*}
    & \liminf_{n\to \infty} \dfrac{2\Ebb_X d_{TV}(f_{G_n}(\cdot|X), f_{G_{0}}(\cdot|X) )}{D_1(G_n, G_0)} \\
    & = \liminf_{n\to \infty} \int_{\Xcal} d\PX(x) \int_{\Ycal} d\nu(y)\left| \dfrac{f_{G_{n}}(y|x) - f_{G_{0}}(y|x)}{D_1(G_n, G_0)} \right|\\
    & \geq \int_{\Xcal} d\PX(x) \int_{\Ycal} d\nu(y) \left(\liminf_{n\to \infty} \left| \dfrac{f_{G_{n}}(y|x) - f_{G_{0}}(y|x)}{D_1(G_n, G_0)}\right|\right)\\
    & \geq \int_{\Xcal} d\PX(x) \int_{\Ycal} d\nu(y) \left| \liminf_{n\to \infty} \dfrac{f_{G_{n}}(y|x) - f_{G_{0}}(y|x)}{D_1(G_n, G_0)}\right|\\
    & = \int_{\Xcal} d\PX(x) \int_{\Ycal} d\nu(y) \left|\sum_{j = 1}^{k_{0}} \alpha_{j} f_{j}^{0}(y|x) + (\beta_{j})^{ \top} \frac{\partial{}}{\partial{ \theta_1}} f_{j}^{0}(y|x) +  (\gamma_i)^{\top} \frac{\partial{}}{\partial{ \theta_2}} f_{j}^{0}(y|x) \right|.
\end{align*}
Since $\lim_{n\to \infty} \dfrac{\Ebb_X d_{TV}(f_{G_n}(\cdot|X), f_{G_{0}}(\cdot|X) )}{D_1(G_n, G_0)} = 0$, we have
\begin{align}\label{eq:contradiction_eq_exact}
    \sum_{j = 1}^{k_{0}} \alpha_{j} f_{j}^{0}(y|x) + (\beta_{j})^{ \top} \frac{\partial{}}{\partial{ \theta_1}} f_{j}^{0}(y|x)  +  (\gamma_i)^{\top} \frac{\partial{}}{\partial{ \theta_2}} f_{j}^{0}(y|x) = 0, \,\, \text{a.s. in } x, y,
\end{align}
where at least one of $\alpha_i, \beta_i, \gamma_i$ are not 0. But by the identifiability of family of conditional densities $f$ (condition (A1.)), we have $\alpha_1=\dots = \alpha_{k_0}=0$, $\beta_1=\dots = \beta_{k_0}=0$ and $\gamma_1=\dots = \gamma_{k_0}=0$. Hence, we arrive at a contradiction and conclude claim~\eqref{eq:claim_local_exact}.

For part (b) of the theorem, in a similar spirit we can achieve the conclusion by proving the following claims:
\begin{equation}
	\inf \limits_{G \in \Ocal_{K} ( \Theta): W_{2}(G, G_{0}) > \epsilon'} \dfrac{\Ebb_{X} d_{TV}(f_{G}(\cdot|X), f_{  G_{0}}(\cdot|X))}{W_{2}^2( G,   G_{0})} > 0, \label{eq:claim_global_over}
\end{equation}
for any $\epsilon' > 0$, and
\begin{equation}
    \lim \limits_{\epsilon \to 0} \inf \limits_{G \in \Ocal_{K}( \Theta)}{\left\{\dfrac{\Ebb_{X} d_{TV}(f_{G}(\cdot|X), f_{  G_{0}}(\cdot|X))}{W_{2}^2( G, G_{0})}: \ W_{2}( G, G_{0}) \leq \epsilon\right\}} > 0. \label{eq:claim_local_over}
\end{equation}
The proof of  claim~\eqref{eq:claim_global_over} is similar to that of claim~\eqref{eq:claim_global_exact} and is omitted. Now we proceed to prove claim~\eqref{eq:claim_local_over}. Suppose this does not hold. So there exists a sequence $G_n \in \Ocal_{K}(\Theta)$ such that 
\begin{equation}\label{eq:contradiction_over}
    W_2(G_n, G_0) \to 0,\quad \dfrac{\Ebb_X d_{TV}(p_{G_n}(\cdot|X), p_{G_{0}}(\cdot|X) )}{W_2^2(G_n, G_0)} \to 0.
\end{equation}
We can assume that all $G_n$ have the same number of atoms (by extracting a subsequence if needed) and relabel the atoms and weights of $G_{n}$ such that it admits the following form:
\begin{align}
	G_{n} = \sum_{j = 1}^{k_{0}+l}\sum_{t=1}^{s_j} p_{jt}^{n} \delta_{(\theta_{1jt}^{n}, \theta_{2jt}^{n})}, \label{eq:relabel_measure_over}
\end{align}
where $\sum_{t=1}^{s_j} p_{jt}^{n} \to p_{j}^{0}$, $\theta_{1jt}^{n} \to \theta_{1j}^{0}$ and $\theta_{2jt}^{n} \to \theta_{2j}^{0}$ for all $j \in [k_{0}+l]$,  $p_j^0=0$ for all $j > k_0$, and $(\theta_{11}^0, \theta_{21}^0), \dots, (\theta^0_{1(k_0+l)}, \theta^0_{2(k_0+l)})$ are distinct. For all $j, t$, we denote $\Delta \theta_{1jt}^{n} : = \theta_{1jt}^{n} - \theta_{1j}^{0}$, $\Delta \theta_{2jt}^{n} : = \theta_{2jt}^{n} - \theta_{2j}^{0}$, and $\Delta p_{j}^{n} : = \sum_{t=1}^{s_j}p_{jt}^{n} - p_{j}^{0}$. We have
\begin{align}
	W_{2}^{2}(G_{n}, G_{0}) \preccurlyeq \sum_{j = 1}^{k_{0} + l}\left( \abss{ \Delta p_{j}^{n}} + \sum_{t=1}^{s_j} p_{jt}^{n} \left(\norm{\Delta \theta_{1jt}^{n}}^2 + \norm{\Delta \theta_{2jt}^{n}}^2\right)\right) =: D_2(G_n, G_0) \label{eq:wasserstein_2_equivalence_over}
\end{align} 
As in part (a) the remainder of the proof is divided into three steps.
\paragraph{Step 1} (Taylor expansion) To ease the notation, we write for short
\begin{align*}
    f_j^0(y|x) = f(y| h_1(x,\theta_{1j}^{0}), h_2(x,\theta_{2j}^{0})),\quad  f_{jt}^n(y|x) = f(y| h_1(x,\theta_{1jt}^{n}), h_2(x,\theta_{2jt}^{n})),
\end{align*}
for all $t = 1,\dots, s_j, j = 1,\dots, k_0+l$. Because $f(y|h_1(x, \theta_1), h_2(x, \theta_2))$ is differentiable up to the second order with respect to $\theta_1, \theta_2$ for all $x,y$, by applying Taylor expansion up to the second order and the chain rule, we find that
\begin{align*}
	f_{jt}^{n}(y|x) - f_j^0(y|x) & = (\Delta \theta_{1jt}^{n})^{ \top} \frac{\partial{}}{\partial{ \theta_1}} f_{j}^{0}(y|x) + (\Delta \theta_{2jt}^{n})^{ \top}\frac{\partial{}}{\partial{ \theta_2}} f_{j}^{0}(y|x) \\
    & \hspace{-.7cm} + \dfrac{1}{2} (\Delta \theta_{1jt}^{n})^{ \top} \frac{\partial^2}{\partial  \theta_1^2} f_{j}^{0}(y|x)  (\Delta \theta_{1jt}^{n}) + \dfrac{1}{2} (\Delta \theta_{2jt}^{n})^{ \top} \frac{\partial^2}{\partial  \theta_2^2} f_{j}^{0}(y|x) (\Delta \theta_{2jt}^{n}) \\
	& \hspace{-.7cm} + (\Delta \theta_{1jt}^{n})^{\top} \frac{\partial^2}{\partial  \theta_1 \partial \theta_2} f_{j}^{0}(y|x) (\Delta \theta_{2jt}^{n})  + R_{i}(x, y)
\end{align*}
where $R_{i}(x, y)$ is Taylor remainder such that $R_{ij}(x, y) = o \left( \norm{ \Delta \theta_{1jt}^{n}}^2\right)$ for $i \in [k_{0}+l]$. Therefore,
\begin{align}
	f_{G_{n}}(y|x) - f_{G_{0}}(y|x) & = \sum_{j = 1}^{k_{0}+l} \parenth{ \Delta p_{j}^{n} } f_{j}^{0}(y|x) + \sum_{j = 1}^{k_{0}+l}\sum_{t=1}^{s_j} p_{jt}^n (f_{jt}^{n}(y|x) - f_{j}^{0}(y|x)) \nonumber \\
	& = \sum_{j = 1}^{k_{0}+l} \parenth{ \Delta p_{j}^{n} } f_{j}^{0}(y|x) + \sum_{j=1}^{k_0+l} \left(\sum_{t=1}^{s_j} p_{ij}^{n} (\Delta \theta_{1jt}^{n})^{ \top} \right) \frac{\partial{}}{\partial{ \theta_1}} f_{j}^{0}(y|x) \nonumber \\
	& + 
	\sum_{j=1}^{k_0+l} \left(\sum_{t=1}^{s_j} p_{ij}^{n} (\Delta \theta_{2jt}^{n})^{ \top} \right) \frac{\partial{}}{\partial{ \theta_2}} f_{j}^{0}(y|x)
	\nonumber\\
	& + \sum_{j=1}^{k_0+l}\Big(\sum_{t=1}^{s_j} \dfrac{p_{ij}^{n}}{2} (\Delta \theta_{1jt}^{n})^{\top} \dfrac{\partial^2}{\partial \theta_1^2} f_{j}^{0}(y|x) (\Delta \theta_{1jt}^{n}) \Big)
	\nonumber\\
	& + \sum_{j=1}^{k_0+l}\Big(\sum_{t=1}^{s_j} \dfrac{p_{ij}^{n}}{2} (\Delta \theta_{2jt}^{n})^{\top} \dfrac{\partial^2}{\partial \theta_2^2} f_{j}^{0}(y|x) (\Delta \theta_{2jt}^{n}) \Big)
	\nonumber\\
	& + \sum_{j=1}^{k_0+l}\Big(\sum_{t=1}^{s_j} \dfrac{p_{ij}^{n}}{2} (\Delta \theta_{1jt}^{n})^{\top} \dfrac{\partial^2}{\partial \theta_1 \partial \theta_2} f_{j}^{0}(y|x) (\Delta \theta_{2jt}^{n}) \Big) + R(x,y),
\end{align}
where $R(x, y) = \sum_{j,t} p_{jt}^{n} R_{jt}(x,y) = o \parenth{ \sum_{j,t} p_{jt}^{n} \parenth{\norm{ \Delta \theta_{1jt}^{n}}^2 + \norm{ \Delta \theta_{2jt}^{n}}^2 }}$. From the expression in Eq.~\eqref{eq:wasserstein_2_equivalence_over}, we have $R(x, y)/ {D}_{2}( G_{n},  G_{0}) \to 0$ as $n \to \infty$ for all $x,y$. 

\paragraph{Step 2} (Extracting non-vanishing coefficients) From Eq.~\eqref{eq:contradiction_over} and \eqref{eq:wasserstein_2_equivalence_over}, we have that
\begin{equation}
    \dfrac{\Ebb_X d_{TV}(f_{G_n}(\cdot|X), f_{G_{0}}(\cdot|X) )}{D_2(G_n, G_0)} \to 0 \quad (n\to \infty). 
\end{equation}
We can write 
\begin{align}
	\dfrac{f_{G_{n}}(y|x) - f_{G_{0}}(y|x)}{D_2(G_n, G_0)}
	& = \sum_{i = 1}^{k_{0}+l} a_j^{n} f_{j}^{0}(y|x) + \sum_{j=1}^{k_0+l} b_j^n \frac{\partial{}}{\partial{ \theta_1}} f_{j}^{0}(y|x) + 
	\sum_{j=1}^{k_0+l} c_j^n \frac{\partial{}}{\partial{ \theta_2}} f_{j}^{0}(y|x)
	\nonumber\\
	& \hspace{-2cm}+ \sum_{j=1}^{k_0+l}\Big(\sum_{t=1}^{s_j} (r_{jt}^{n})^{\top} \dfrac{\partial^2}{\partial \theta_1^2} f_{j}^{0}(y|x) (r_{jt}^{n}) \Big)
	+ \sum_{j=1}^{k_0+l}\Big(\sum_{t=1}^{s_j} (v_{jt}^{n})^{\top} \dfrac{\partial^2}{\partial \theta_2^2} f_{j}^{0}(y|x) (v_{jt}^{n}) \Big)
	\nonumber\\
	& \hspace{-2cm}+ 2\sum_{j=1}^{k_0+l}\Big(\sum_{t=1}^{s_j} (r_{jt}^{n})^{\top} \dfrac{\partial^2}{\partial \theta_1 \partial \theta_2} f_{j}^{0}(y|x) (v_{jt}^{n}) \Big) + R(x,y), \label{eq:Taylor_over}
\end{align}
where 
\begin{equation*}
    a_j^{n} = \dfrac{\parenth{\Delta p_{j}^{n}}}{D_2(G_n, G_0)}, \, \,
    b_{j}^{n} = \dfrac{\sum_{t=1}^{s_j} p_{jt}^{n} (\Delta \theta_{1jt}^{n})}{ D_2(G_n, G_0)}, \, \,
    c_{j}^{n} = \dfrac{\sum_{t=1}^{s_j} p_{jt}^{n} (\Delta \theta_{2jt}^{n})}{ D_2(G_n, G_0)},
\end{equation*}
and
\begin{equation*}
    r_{jt}^{n} =  \dfrac{\sqrt{p_{jt}^{n}}(\Delta \theta_{1jt}^{n})}{\sqrt{2 D_2(G_n, G_0)}}, \, \,
    v_{jt}^{n} =  \dfrac{\sqrt{p_{jt}^{n}}(\Delta \theta_{2jt}^{n})}{\sqrt{2 D_2(G_n, G_0)}},
\end{equation*}
for all $j\in [k_0 + l]$. 
From the definition of $D_2(G_n, G_0)$, we have 
\begin{equation*}
    \sum_{j=1}^{k_0+l} |a_j^n| + 2\sum_{j=1}^{k_0+l} \sum_{t=1}^{s_i} \norm{r_{jt}^{n}}^2 + 2\sum_{j=1}^{k_0+l} \sum_{t=1}^{s_i} \norm{v_{jt}^{n}}^2 = 1, 
\end{equation*}
so that $M_n := \max_{j, t}\{|a_j^{n}|, \|{b_j^n}\|, \|{c_j^n}\|, \|{r_{jt}^{n}}\|^2, \|{v_{jt}^{n}}\|^2\}$ is always bounded below by $\dfrac{1}{5\Kup}$ for all $n$, and does not converge to 0. Denote 
\begin{align*}
    \alpha_j^{n} = a_{j}^{n} / M_n,
    \,\, 
    \beta_{j}^{n} = b_{j}^{n} / M_n, \,\,
    \gamma_{j}^{n} = c_{j}^{n} / M_n, \,\,
    \rho_{jt}^{n} = r_{jt}^{n} / \sqrt{M_n}, \,\,
    \nu_{jt}^{n} = v_{jt}^{n}/\sqrt{M_n}.
\end{align*}
for all $t =1, \dots, s_j, j=1,\dots, k_0+l$. By compactness and subsequence argument, we can have that $\alpha_j^n \to \alpha_j\in [-1,1], \beta_{j}^n \to \beta_{j}\in [-1,1]^{d_1}$ and $\gamma_{j, t}^{n} \to \gamma_{j}\in [-1,1]^{d_2}$, and $\rho_{jt}^{n} \to \rho_{jt}\in [-1,1]^{d_1}, \nu_{jt}^{n} \to \nu_{jt} \in \Rbb^{d_2}$ as $n\to \infty$ for all $t, j$, and at least one of those limits is not zero. 
\paragraph{Step 3} (Deriving contradiction via Fatou's lemma) By Fatou's lemma, we have
\begin{align*}
    & \liminf_{n\to \infty} \dfrac{2}{M_n}\dfrac{\Ebb_X d_{TV}(f_{G_n}(\cdot|X), f_{G_{0}}(\cdot|X) )}{D_2(G_n, G_0)} \\
    & = \liminf_{n\to \infty} \dfrac{2}{M_n}\int_{\Xcal} d\PX(x) \int_{\Ycal} d\nu(y)\left| \dfrac{f_{G_{n}}(y|x) - f_{G_{0}}(y|x)}{D_2(G_n, G_0)} \right|\\
     &  \geq \int_{\Xcal} d\PX(x) \int_{\Ycal} d\nu(y) \left(\liminf_{n\to \infty} \dfrac{1}{M_n} \left| \dfrac{f_{G_{n}}(y|x) - f_{G_{0}}(y|x)}{D_2(G_n, G_0)}\right|\right)\\
     &  \geq \int_{\Xcal} d\PX(x) \int_{\Ycal} d\nu(y) \left| \liminf_{n\to \infty}\dfrac{1}{M_n} \dfrac{f_{G_{n}}(y|x) - f_{G_{0}}(y|x)}{D_2(G_n, G_0)}\right|\\
     &  = \int_{\Xcal} d\PX(x) \int_{\Ycal} d\nu(y) \bigg|\sum_{j = 1}^{k_{0}+l} \alpha_j f_{j}^{0}(y|x) + (\beta_j)^{\top} \dfrac{\partial}{\partial \theta_1} f_{j}^{0}(y|x) + (\gamma_j)^{\top} \dfrac{\partial}{\partial \theta_2} f_{j}^{0}(y|x)\\
     &  + \sum_{t=1}^{s_j} (\rho_{jt})^{\top} \dfrac{\partial}{\partial \theta_1^2} f_{j}^{0}(y|x) (\rho_{jt}) + \sum_{t=1}^{s_j} (\nu_{jt})^{\top} \dfrac{\partial}{\partial \theta_2^2} f_{j}^{0}(y|x) (\nu_{jt}) \\
     & + 2\sum_{t=1}^{s_j} (\rho_{jt})^{\top} \dfrac{\partial}{\partial \theta_1\partial \theta_2} f_{j}^{0}(y|x) (\nu_{jt}) \bigg|.
\end{align*}
Since $\lim_{n\to \infty} \dfrac{1}{M_n}\dfrac{\Ebb_X d_{TV}(f_{G_n}(\cdot|X), f_{G_{0}}(\cdot|X) )}{D_2(G_n, G_0)} = 0$, the integrand in the right hand side of the above display is 0 for almost all $x, y$. By the second order identifiability of $f(y|x)$, all the coefficients are 0, which contradicts with the fact derived in the end of Step 2. We arrive at the conclusion of claim~\eqref{eq:claim_local_over}.

\end{proof}

\begin{proof}[Proof of Proposition~\ref{prop:complete-ident}] We want to prove that for $h_1$ and $h_2$ being completely identifiable, then for any $ k \geq 1$ and distinct pairs $(\theta_{11}, \theta_{21}), \dots, (\theta_{1k}, \theta_{2k})$ we have 
$$(h_1(x, \theta_{11}), h_2(x, \theta_{21})), \dots, (h_1(x, \theta_{1k}), h_2(x, \theta_{2k}))$$
are distinct almost surely. 
For any $i \neq j$, because $(\theta_{1i}, \theta_{2i})\neq (\theta_{1j}, \theta_{2j})$, we have either $\theta_{1i}\neq \theta_{1j}$ or $\theta_{2i}\neq \theta_{2j}$. By the complete identifiability of $h_1$ and $h_2$, we have either
\begin{equation*}
    \Pbb_X(\{x: h_1(x, \theta_{1i}) = h_1(x, \theta_{1j})\}) = 0, 
\end{equation*}
or 
\begin{equation*}
    \Pbb_X(\{x: h_2(x, \theta_{2i}) = h_2(x, \theta_{2j})\}) = 0.
\end{equation*}
Hence,
{\fontsize{10}{10}\selectfont 
\begin{align*}
    & \Pbb_X(\{x: (h_1(x, \theta_{1i}), h_2(x, \theta_{2i})) = (h_1(x, \theta_{1j}), h_2(x, \theta_{2j}))\}) \\
    & =  \Pbb_X(\{x: h_1(x, \theta_{1i}) = h_1(x, \theta_{1j}), h_2(x, \theta_{2i}) = h_2(x, \theta_{2j}) \}) \\
    & \leq  \min\{\Pbb_X(\{x: h_1(x, \theta_{1i}) = h_1(x, \theta_{1j})\}), \Pbb_X(\{x: h_2(x, \theta_{2i}) = h_2(x, \theta_{2j})\})\} \\
    & = 0. 
\end{align*}}
Now consider the set 
$$A = \cup_{1\leq i < j\leq k} \{x : (h_1(x, \theta_{1i}), h_2(x, \theta_{2i})) = (h_1(x, \theta_{1j}), h_2(x, \theta_{2j}))\},$$
we have 
\begin{equation*}
    \mu(A) \leq \sum_{1\leq i < j\leq k} \mu(\{x: (h_1(x, \theta_{1i}), h_2(x, \theta_{2i})) = (h_1(x, \theta_{1j}), h_2(x, \theta_{2j}))\}) = 0.
\end{equation*}
Therefore, $(h_1(x, \theta_{11}), h_2(x, \theta_{21})), \dots, (h_1(x, \theta_{1k}), h_2(x, \theta_{2k}))$ are distinct on $A^c$, where $\Pbb_{X}(A^c) = 1$.
\end{proof}

\begin{proof}[Proof of Proposition~\ref{lem:inverse-bounds-prediction-error}]
(a) This comes directly from the fact that if $h_1$ and $h_2$ are Lipschitz. For any $G_0 = \sum_{j=1}^{k_0} p_i^0 \delta_{(\theta_{1j}^0, \theta_{2j}^0)}\in \Ecal_{k_0}(\Theta), G= \sum_{i=1}^{\Kup} p_i\delta_{(\theta_{1i}, \theta_{2i})} \in \Ecal_{\Kup}(\Theta)$, we have 
$$|h_1(x, \theta_{1j}^0) - h_1(x, \theta_{1i})| \leq c_1 \norm{\theta_{1j}^0 - \theta_{1i}},$$
and 
$$|h_2(x, \theta_{2j}^0) - h_2(x, \theta_{2i})| \leq c_2 \norm{\theta_{2j}^0 - \theta_{2i}},$$
for any $j=1, \dots, k_0, i=1,\dots, \Kup$ and all $x\in \Xcal$, where $c_1$ and $c_2$ are constants which only depend on $h_1$ and $h_2$. Hence, for any coupling $(q_{ij})_{i,j=1}^{\Kup, k_0}$ of $(p_i)_{i=1}^{\Kup}$ and $(p_j^0)_{j=1}^{k_0}$, we have
\begin{align*}
    &\sum_{i, j} q_{ij} (\norm{\theta_{1i} - \theta_{1j}^0} + \norm{\theta_{2i} - \theta_{2j}^0}) 
    \\ 
    &\geq \overline{C}_1 \sum_{i, j} q_{ij} (|h(x, \theta_{1i}) - h(x, \theta_{1j}^0)| + |h(x, \theta_{1i}) - h(x, \theta_{2j}^0)|)
\end{align*}
for $\overline{C}_1= 1 / \max\{c_1, c_2\}$ and for all $x\in \Xcal$. Taking infimum with respect to the LHS, we have
$$W_r(G, G_0)\geq \overline{C}_1  W_r\left(\sum_{j=1}^{k_0} {p}^{0}_j \delta_{(h_1(x, {\theta}^{0}_{1j}), h_2(x, {\theta}^{0}_{2j}))},  \sum_{j=1}^{K} p_j \delta_{(h_1(x, {\theta}_{1j}), h_2(x, {\theta}_{2j}))}\right),$$
Taking the expectation with respect to $\Pbb_{X}$ we obtain
\begin{equation*}
    W_r(G, G_0) \geq \overline{C}_1\Ebb_X W_r\left(\sum_{j=1}^{k_0} {p}^{0}_j \delta_{(h_1(X, {\theta}^{0}_{1j}), h_2(X, {\theta}^{0}_{2j}))},  \sum_{j=1}^{K} p_j \delta_{(h_1(X, {\theta}_{1j}), h_2(X, {\theta}_{2j}))}\right).
\end{equation*}

(b) For any coupling $(q_{ij})_{i,j=1}^{K, k_0}$ of $(p_i)_{i=1}^{K}$ and $(p_j^0)_{j=1}^{k_0}$ we have
\begin{align*}
    \Ebb_X \left|\sum_{i=1}^{K} p_i h_1(X, \theta_{1i}) - \sum_{j=1}^{k_0} p_j^0 h_1(X, \theta_{1j}^0) \right|& \leq \sum_{i, j=1}^{K, k_0} q_{ij} | h_1(X, \theta_{1i}) -  h_1(X, \theta_{1j}^0)|\\
    & 
     \leq \sum_{i, j=1}^{K, k_0} q_{ij} c_1 \norm{\theta_{1i} - \theta_{1j}^0}.
\end{align*}
Taking infimum of all feasible $(q_{ij})_{i, j}$, this implies
$$\Ebb_X \left|\sum_{i=1}^{K} p_i h_1(X, \theta_{1i}) - \sum_{j=1}^{k_0} p_j^0 h_1(X, \theta_{1j}^0) \right|\leq c_1 W_1(G, G_0).$$
Doing similarly for $h_2$, we have the conclusion. 
\end{proof}

\subsection{Convergence rates for conditional density estimation and parameter estimation}\label{subsec:convergence-rate-proofs}
Firstly, we combine the inverse bounds (Theorem~\ref{thm:inverse-bounds}) with the convergence theory for density estimation to derive convergence rates for parameter estimation that arise in regression mixture models as presented in Theorem~\ref{thm:convergence_parameter_estimation}.
\begin{proof}[Proof of Theorem~\ref{thm:convergence_parameter_estimation}] Recall that with the assumptions in this theorem, we have
$$\Ebb_{X} d_{TV}(f_{G}, f_{G_0}) \geq C_1 W_1(G, G_0)\quad \forall G \in \Ecal_{k_0}(\Theta),$$
and for any $K> k_0$,
$$\Ebb_{X} d_{TV}(f_{G}, f_{G_0}) \geq C_2 W_2^2(G, G_0)\quad \forall G \in \Ocal_{K}(\Theta),$$
for $C_1, C_2 > 0$ only depend on $\Theta, G_0, f, h_1, h_2$, and $K$. Besides, 
$$\sqrt{2}\overline{d}_{H}(f_{G}, f_{G_0})\geq \Ebb_{X} d_{TV}(f_{G}, f_{G_0}) \quad \forall G, G_0.$$
Combining these inequalities with the concentration inequality given in Theorem~\ref{thm:density_estimation_rate} to have 
$$\Pbb_{G_0} (W_1(\widehat{G}_n, G_0) > C\delta) \leq \Pbb_{G_0} (\overline{d}_{H}(f_{G}, f_{G_0}) > \sqrt{2}C C_1\delta) \leq c \exp(-n\delta^2 / c^2),$$
for the exact-fitted setting, since $\hat{G}_n\in \Ecal_{k_0}(\Theta)$. 
In the over-fitted setting, a similar argument yields
$$\Pbb_{G_0} (W_2^2(\widehat{G}_n, G_0) > C\delta) \leq \Pbb_{G_0} (\overline{d}_{H}(f_{G}, f_{G_0}) >\sqrt{2} C C_2\delta) \leq c \exp(-n\delta^2 / c^2).$$
\end{proof}


Next, we proceed to prove Theorem \ref{thm:simple_density_estimation_rate} which is concerned with the convergence rates of conditional density estimation.

\begin{proof}[Proof of Theorem~\ref{thm:simple_density_estimation_rate}] 
    The proof is a generalization of proof of Theorem 3.1 in  \cite{ghosal2001entropies}. First we prove that if for any fixed $k$ and for all $\epsilon\in (0,1/2)$, these claims hold
    \begin{align}
        \log N(\epsilon, \Fcal_{k}(\Theta), \norm{\cdot}_{\infty}) &  \preccurlyeq \log(1/\epsilon)\label{claim:bounded_covering_inf_norm} \\
        H_B(\epsilon, \Fcal_{k}(\Theta), \overline{d}_H) & \preccurlyeq \log(1/\epsilon), \label{claim:bounded_bracketing_hellinger} 
    \end{align}
    then by applying Theorem~\ref{thm:density_estimation_rate}, we can arrive at our conclusion. Indeed, since $$\norm{\left(\dfrac{f+f_0}{2}\right)^{1/2} - \left(\dfrac{g+f_0}{2}\right)^{1/2}}_2 \leq d_H(f, g)$$ 
    for all densities $f, g, f_0$, we have 
    \begin{equation*}
        H_B(u, \overline{\Fcal}_k^{1/2}(\Theta, u), \norm{\cdot}_{2}) \leq H_B(u, \Pcal_k(\Omega), \overline{d}_H),
    \end{equation*}
    for all $u>0$. Thus, we can bound the bracketing entropy integral as follows
    \begin{align*}
     \mathcal{J}(\delta, \overline{\Fcal}_k^{1/2}(\Theta, \delta)) & \leq \int_{\delta^2/2^{13}}^{\delta} H_B^{1/2}(u, \Fcal_k(\Omega), \overline{d}_H) du \vee \delta \\
        & \preccurlyeq \int_{\delta^2/2^{13}}^{\delta} \log(1/u) du \vee \delta\\
        & \leq \delta \log(2^{13}/\delta^2) \vee \delta\\
        &\preccurlyeq \delta \log(1/\delta).
    \end{align*}
    Hence, if we choose $\Psi(\delta) = \delta \log (1/\delta)$, then $\Psi(\delta)\succcurlyeq \mathcal{J}(\delta, \overline{\Pcal}_k^{1/2}(\Theta, \delta))$, $\Psi(\delta) / \delta = \log (1/\delta) (1/\delta)$ is a non-increasing function, and for $\delta_n = O((\log(n) / n)^{1/2})$, we have
    \begin{equation*}
        \sqrt{n} \delta_n^2 \succcurlyeq \log(n) / \sqrt{n} \succcurlyeq \Psi(\delta_n).
    \end{equation*}
    Therefore, the result of Theorem~\ref{thm:density_estimation_rate} says that there exist constant $C$ and $c$ such that
    \begin{equation*}
        \Pbb \left(\overline{d}_H(f_{\widehat{G}_n}, f_{G_0} ) > C\sqrt{\dfrac{\log(n)}{n}}\right) \preccurlyeq \exp(-c \log(n)),
    \end{equation*}
    which is the conclusion. It remains to verify~\eqref{claim:bounded_covering_inf_norm} and~\eqref{claim:bounded_bracketing_hellinger}. 

\paragraph{Proof of claim~\eqref{claim:bounded_covering_inf_norm}} Since $\Theta_1$ and $\Theta_2$ are compact, for all $\epsilon>0$, there exists a $\epsilon$-net $B_1$ of $(\Theta_1, \norm{\cdot})$ and $B_2$ of $(\Theta_2, \norm{\cdot})$ with the cardinality $|B_1|\leq \left(\dfrac{\diam(\Theta_1)}{\epsilon}\right)^{d_1}$ and $|B_2|\leq \left(\dfrac{\diam(\Theta_2)}{\epsilon}\right)^{d_2}$. We also know that there exists a $\epsilon$-net $A$ for $(\Delta^{k-1}, \norm{\cdot}_{\infty})$ such that $|A| \leq \left(\dfrac{5}{\epsilon}\right)^k$ (\cite{ghosal2001entropies}). We consider the following subset of $\Fcal_{k}(\Theta)$
\begin{equation*}
    C = \{p_G: G = \sum_{i=1}^{k} p_i \delta_{(\theta_{1j}, \theta_{2j})}, (p_i)_{i=1}^{k} \in A, \theta_{1j} \in B_1, \theta_{2j} \in B_2\forall \, i\}.
\end{equation*}
We can see that 
\begin{equation*}
    |C| = |A|\times |B_1|^k \times |B_2|^k \leq \left(\dfrac{5}{\epsilon}\right)^{k} \left(\dfrac{\diam(\Theta_1)}{\epsilon}\right)^{d_1k} \left(\dfrac{\diam(\Theta_2)}{\epsilon}\right)^{d_2k}.
\end{equation*}
For any $G = \sum_{i=1}^{k} p_i \delta_{\theta_{1j}} \in \Ocal_{k}(\Theta)$, there exist $(\tilde{p}_i)_{i=1}^{k} \in A$ and $\tilde{\theta}_i \in B$ such that $|p_i - \tilde{p}_i|\leq \epsilon$ and $\norm{\theta_{1j} - \tilde{\theta}_i}\leq \epsilon$ for all $i$. Let $\tilde{G} = \sum_{i=1}^{k} \tilde{p}_i \delta_{\tilde{\theta}_i}$ and $G' = \sum_{i=1}^{k} \tilde{p}_i \delta_{\theta_{1j}}$, by triangle inequality, we have
{\fontsize{10}{10}\selectfont 
\begin{align*}
    & \norm{f_{G}(y|x) - f_{\tilde{G}}(y|x)}_{\infty} \leq \norm{f_{G}(y|x) - f_{G'}(y|x)}_{\infty} + \norm{f_{G'}(y|x) - f_{\tilde{G}}(y|x)}_{\infty}\\
    & \leq \sum_{j=1}^{k} |p_j - \tilde{p}_j|\norm{f(y|h_1(x, \theta_{1j}), h_2(x, \theta_{2j}))}_{\infty} \\
    &  + \sum_{j=1}^{k} \tilde{p}_j \norm{(f(y|h_1(x, \theta_{1j}),h_2(x, \theta_{2j})) - f(y|h_1(x, \tilde{\theta}_{1j}),h_2(x, \tilde{\theta}_{2j}))}_{\infty}\\
    & \preccurlyeq \sum_{j=1}^{k} |p_j - \tilde{p}_j|\norm{f(y|h_1(x, \theta_{1j}), h_2(x, \theta_{2j})) }_{\infty}\\
    & + \sum_{i=1}^{k} \tilde{p}_j \left(\norm{\theta_{1j} - \tilde{\theta}_{1j}} + \norm{\theta_{2j} - \tilde{\theta}_{2j}} \right)\\
    & \preccurlyeq \epsilon,
\end{align*}}
where we apply the assumptions that $\norm{f(y|\mu, \phi)}_{\infty}$ is bounded uniformly in $(\mu, \phi)\in H$, and the uniform Lipschitz of $f$ and $h_1, h_2$. Hence $C$ forms a $\epsilon$-net of $\Fcal_k(\Theta)$. This implies that
\begin{equation*}
    N(\epsilon, \Fcal_k(\Theta), \norm{\cdot}_{\infty})\preccurlyeq \left(\dfrac{5}{\epsilon}\right)^{k} \left(\dfrac{\diam(\Theta_1)}{\epsilon}\right)^{d_1k} \left(\dfrac{\diam(\Theta_2)}{\epsilon}\right)^{d_2k}.
\end{equation*}
Taking the logarithm of both sides, we arrive at the conclusion of claim~\eqref{claim:bounded_covering_inf_norm}.
\paragraph{Proof of claim~\eqref{claim:bounded_bracketing_hellinger}} We first construct an $\epsilon$-bracketing for $\Fcal_{k}(\Theta)$ under $\ell_1$ norm. Let $\eta$ be a small number that we can choose later, and $f_1, \dots, f_M$ is a $\eta$-net for $\Fcal_k(\Theta)$ under $\norm{\cdot}_{\infty}$, for $M \preccurlyeq \log(1/\epsilon)$. Denote by $C_1$ an upper bound for $\norm{f(y|\mu, \phi)}_{\infty}$ for all $(\mu, \phi)\in H$. From our assumptions, we can construct an envelope for $\Fcal_{k}(\Theta)$ as follows
\begin{equation*}
    H(x, y) = \begin{cases}
    d_1\exp(-d_2 |y|^{d_3}) , & \forall y > \overline{c} \text{ or } y<\underline{c}\\
    C_1  , & \forall y\in [\underline{c}, \overline{c}],
    \end{cases}
\end{equation*}
where we can assume that $\overline{c} > 0$ and $\underline{c} < 0$. Then, we can create the brackets $[f_i^{L}(x, y), f_i^{U}(x, y)]_{i=1}^{M}$ by 
\begin{equation*}
    f_i^{L}(x, y) :=\max\{f_i(y|x) -\eta, 0\}, \quad f_i^{U}(x, y) :=\max\{f_i(y|x) + \eta, H(x, y)\}.
\end{equation*}
Because for all $f\in \Fcal_k(\Theta)$, we have a $f_i$ such that $\norm{f-f_i}_{\infty}\leq \eta$, it can be seen that $f( y|x) \in [f_i^{L}(x, y), f_i^{U}(x, y)]$ for all $x, y$. Therefore, $\Fcal_k(\Theta)\subset \cup_{i=1}^{M}[f_i^{L}, f_i^{U}]$. Besides, for any $\overline{C} > \overline{c}$ and $\underline{C} < \underline{c}$, we have
\begin{align}\label{eq:l1_gap_bracket}
    & \int (f_i^U - f_i^{L}) d\Pbb(x) d\nu(y) \leq \int_{x}\int_{y=\underline{C}}^{y=\overline{C}} (f_i^U - f_i^{L}) d\Pbb(x) d\nu(y) \nonumber \\
    & + \int_{x}\int_{\{y<\underline{C}\} \cup \{y >\overline{C}\}}(f_i^U - f_i^{L}) d\Pbb(x) d\nu(y)\nonumber\\
    &\leq \eta (\overline{C} - \underline{C}) + \int_{\{y<\underline{C}\} \cup \{y >\overline{C}\}} d_1 \exp(-d_2 |y|^{d_3})d\nu(y)\nonumber\\
    &\leq \eta (\overline{C} - \underline{C}) + \int_{\{u<d_2|\underline{C}
    |^{d_3}\} \cup \{u > d_2\overline{C}^{d_3}\}} d_1 \exp(-|u|) |u|^{1/d_3-1}d\nu(y)\nonumber\\
    & \preccurlyeq \eta (\overline{C} - \underline{C}) + \overline{C}\exp(-d_2 \overline{C}^{d_3}) - \underline{C}\exp(d_2 \underline{C}^{d_3}),
\end{align}
where we use the change of variable formula $u = d_2 |y|^{d_3}$, and the fact that when $\nu$ is the Lebesgue measure:
\begin{align*}
    \int_{u \geq z} \exp(-u)\left(u\right)^{1/d_3-1} du & = z^{1/d_3} e^{-z} \int_{0}^{\infty} (1+s)^{1/d_3 - 1} e^{-zs} \\
    &\leq z^{1/d_3} e^{-z} \dfrac{1}{z - 1/d_3 + 1} < z^{1/d_3} e^{-z},
\end{align*}
for all $z \geq 0$. Notice that if $f$ is a probability mass function (i.e., $\nu$ is discrete), we can change the integral to sum, and the result still holds because
\begin{equation*}
    \sum_{y=\overline{C}+1}^{\infty} \exp(-d_2 |y|^{d_3})\leq \int_{y = \overline{C}}^{\infty} \exp(-d_2 |y|^{d_3}) dy. 
\end{equation*}
Now, let $\overline{C} = \overline{c} (\log(1/\eta))^{1/d_3}, \underline{C} = \underline{c} (\log(1/\eta))^{1/d_3}$, we have
\begin{equation*}
    \norm{f_i^{U} - f_i^{L}}_1\preccurlyeq \eta^{d_4} \left(\log\left(\dfrac{1}{\eta}\right) \right)^{1/d_3},
\end{equation*}
where $d_4 = \max\{1, d_2\overline{c}^{d_3}, d_2|\underline{c}|^{d_3}\}$. Hence, there exists a positive constant $c$ which does not depend on $\eta$ such that
\begin{equation*}
    H_B(c\eta^{d_4} (\log(1/\eta))^{1/d_3}, \Fcal_k(\Theta), \norm{\cdot}_1) \preccurlyeq \log(1/\eta). 
\end{equation*}
Let $\epsilon = c\eta^{d_4} (\log(1/\eta))^{1/d_3}$, we have $\log(1/\epsilon) \asymp \log(1/\eta)$. Combining with inequality $\norm{\cdot}_1\leq h^2$ yields
\begin{equation*}
    H_B(\epsilon, \mathcal{F}_k(\Theta), h)\leq H_B(\epsilon^2, \mathcal{F}_k(\Theta), \norm{\cdot}_1)\preccurlyeq \log(1/\epsilon^2) \preccurlyeq \log(1/\epsilon).
\end{equation*}
Thus, we have proved claim~\eqref{claim:bounded_bracketing_hellinger}.
\end{proof}

Finally, we obtain upper bounds on the tail probability for some popular family of distributions in order to verify that they satisfy all conditions of Theorem~\ref{thm:density_estimation_rate}.

\begin{proof}[Proof of Proposition~\ref{prop:satisfied_densities}] Since the parameter space $\Lambda$ is compact, we can assume it is a subset of some $[\underline{\lambda}, \overline{\lambda}]$, where $\overline{\lambda} >0$ and $\underline{\lambda} < 0$. If the family of distribution is discrete, then obviously its probability mass function is bounded uniformly by 1. 

\noindent
(a) For the family of normal distribution $\{f(y|\mu, \sigma^2) : \mu \in [\overline{\lambda}, \underline{\lambda}], \sigma^2\}$, we have that
\begin{equation}
    f(y|\mu, \sigma^2) \leq \dfrac{1}{\sqrt{2\pi \sigma^2}} \exp(-y/8\sigma^2), 
\end{equation}
for all $y > 2 \overline{\lambda}$ or $y < 2 \underline{\lambda}$.

\noindent
(b) For the family of Binomial distribution $\textrm{Bin}(N, q)$, we can see that it is discrete and domain of $q$ is  bounded. Therefore the conclusion is immediate.

\noindent
(c) For the family of Poisson distribution $f(y|\lambda)$, we have that $f(y|\lambda) = 0 \forall \, y < 0$ and 
\begin{equation*}
    f(y|\lambda) = \dfrac{e^{-\lambda}\lambda^y}{y!} \leq \exp(-y),
\end{equation*}
for all $y\geq 2 (\overline{\lambda} e)^2$ due to the inequality $y! \geq \left(\dfrac{y}{2}\right)^{y/2}$.

\noindent
(d) For the family of negative binomial distribution $f(y|\mu, \phi)$, we also have $f(y|\mu, \phi) = 0 \forall \, y < 0$, and 
\begin{equation*}
    f(y|\mu, \phi) \preccurlyeq y^{[\phi]+1} \left(\dfrac{\mu}{\mu+\theta}\right)^{y}\leq y^{[\phi]+1} \left(\dfrac{\overline{\mu}}{\overline{\mu}+\theta}\right)^{y}\leq \left(\dfrac{\overline{\mu}}{\overline{\mu}+\theta}\right)^{y/2}.
\end{equation*}
for all $y$ large enough compared to $\overline{\mu}$ and $\phi$.
\end{proof}

\begin{proof}[Proof of Theorem~\ref{thm:simple_convergence_parameter_estimation}] Similar to the proof of Theorem~\ref{thm:convergence_parameter_estimation}, with $\delta = \sqrt{\log (n)/ n}$ (and using Theorem~\ref{thm:simple_density_estimation_rate} instead of Theorem~\ref{thm:density_estimation_rate}).
\end{proof}

\subsection{Posterior contraction theorems}\label{subsec:contraction-rate-proofs}
\begin{proof}[Proof of Theorem~\ref{thm:posterior-contraction-rate-regression}]
    It suffices to verify conditions (i) and (ii) of Theorem~\ref{thm:basic-contraction-rate} in Appendix~\ref{sec:posterior-contraction-rate-theory} in order to arrive at the conclusion, with $\Fcal = \Fcal_n = \{f_G : G\in \Ocal_K\}$ and $\epsilon_n = (\log(n)/n)^{1/2}$.
    
    \paragraph{Checking condition (i):} We need to show that the prior distribution puts enough mass around the true (conditional) density function $f_{G_0}$, i.e., to obtain a lower bound for $\Pi(B_2(f_{G_0}, \epsilon_n))$. First, consider the ball $\{G\in\Ocal_{K}(\Theta): W_1(G, G_0)\leq C\epsilon_n^2 \}$ for a constant $C$ to be chosen later. By Lemma~\ref{lem:basic-bounds}, we have $\Ebb_{X} d_{H}^2(f_{G_0}, f_{G})\leq C_1 C\epsilon_n^2$, where $C_1$ depends on $\Theta$. Because $C C_1 \epsilon_n^2 \leq \epsilon_0$ for all sufficiently large $n$, we have $\Ebb_{\Pbb_X \times f_{G_0}} (f_{G_0} / f_{G})\leq M$. By Theorem 5 in~\cite{wong1995probability}, we have
    \begin{align*}
        \Ebb_{\Pbb_X} K(f_{G_0}, f_{G}) &  \preccurlyeq \epsilon_n^2 \log(\sqrt{M} / \sqrt{CC_1}\epsilon_n)\\
        \Ebb_{\Pbb_X} K_2(f_{G_0}, f_{G}) &  \preccurlyeq \epsilon_n^2 \log(\sqrt{M} / \sqrt{CC_1}\epsilon_n)^2.
    \end{align*}
    Hence, for $\overline{M} = \log(\sqrt{M} / \sqrt{CC_1})$, we have
    \begin{equation*}
        \Pi(B_2(f_{G_0}, \overline{M}\epsilon_n)) \geq \Pi(W_1(G, G_0) \leq C\epsilon_n).
    \end{equation*}
    However, for all $G = \sum_{i=1}^{k_0} p_i \delta_{(\theta_{1i}, \theta_{2i})}$ such that $\|\theta_{1i} - \theta_{1i}^0\|\leq \epsilon_n / (4k_0), \|\theta_{2i} - \theta_{2i}^0\|\leq \epsilon_n / (4k_0), |p_i - p_i^0|\leq \epsilon_n / (4k_0 \diam(\Theta_1)\times \diam(\Theta_2))$, we have
    \begin{align*}
        W_1(G_0, G) &\leq \sum_{i=1}^{k_0} (p_i^0 \wedge p_i) (\norm{\theta_{1i} - \theta_{1i}^0}+\norm{\theta_{2i} - \theta_{2i}^0}) + |p_i - p_i^0| (\diam(\Theta_1) \diam(\Theta_2))\\
        & \leq \epsilon_n.
    \end{align*}
    Due to assumption (B1.), the prior measure of this set is asymptotically greater than $\epsilon_n$ Hence
    \begin{align*}
        \Pi(W_1(G, G_0) \leq C\epsilon_n) \succcurlyeq \epsilon_n \succcurlyeq e^{-c n \epsilon_n^2},
    \end{align*}
    as $\epsilon_n = (\log(n)/n)^{1/2}$.
    
    \paragraph{Checking condition (ii):}We need to provide an upper bound for the entropy number $\log N(\Fcal, \overline{d}_H, \epsilon_n)$. By Lemma~\ref{lem:basic-bounds}, 
    \begin{equation*}
        \overline{d}^2_H(f_{G}, f_{G_0})\leq \Ebb_{X} d_{TV}(f_{G}, f_{G_0}) \preccurlyeq W_1(G, G_0)
    \end{equation*}
    We use the same strategy as in the proof of Theorem~\ref{thm:simple_density_estimation_rate}. Since $\Theta_1$ and $\Theta_2$ are compact, for all $\epsilon>0$, there exists an $\epsilon$-net $B_1$ of $(\Theta_1, \norm{\cdot})$ and $B_2$ of $(\Theta_2, \norm{\cdot})$ with the cardinality $|B_1|\leq \left(\dfrac{\diam(\Theta_1)}{\epsilon}\right)^{d_1}$ and $|B_2|\leq \left(\dfrac{\diam(\Theta_2)}{\epsilon}\right)^{d_2}$. Moreover, there exists an $\epsilon$-net $A$ for $(\Delta^{k-1}, \norm{\cdot}_{\infty})$ such that $|A| \leq \left(\dfrac{5}{\epsilon}\right)^k$. We consider the following subset of $\Fcal$
\begin{equation*}
    C = \{G: G = \sum_{i=1}^{k} p_j \delta_{(\theta_{1j}, \theta_{2j})}, (p_j)_{j=1}^{k} \in A, \theta_{1j} \in B_1, \theta_{2j} \in B_2\forall \, j\}.
\end{equation*}
Note that
\begin{equation*}
    |C| = |A|\times |B_1|^k \times |B_2|^k \leq \left(\dfrac{5}{\epsilon}\right)^{k} \left(\dfrac{\diam(\Theta_1)}{\epsilon}\right)^{d_1k} \left(\dfrac{\diam(\Theta_2)}{\epsilon}\right)^{d_2k}.
\end{equation*}
For any $G = \sum_{i=1}^{k} p_i \delta_{(\theta_{1j}, \theta_{2j})} \in \Ocal_{k}(\Theta)$, there exist $(\tilde{p}_j)_{j=1}^{k} \in A$ and $\tilde{\theta}_j \in B$ such that $|p_j - \tilde{p}_j|\leq \epsilon_n$ and $\norm{\theta_{1j} - \tilde{\theta}_{2j}}\leq \epsilon_n$, $\norm{\theta_{2j} - \tilde{\theta}_{2j}}\leq \epsilon_n$ for all $j$. Let $\tilde{G} = \sum_{j=1}^{k} \tilde{p}_i \delta_{(\tilde{\theta}_{1j}, \tilde{\theta}_{2j})}$ and $G' = \sum_{i=1}^{k} \tilde{p}_i \delta_{(\theta_{1j}, \theta_{2j})}$, by the triangle inequality, we have
\begin{align*}
    W_1(G, \tilde{G}) & \leq W_1(G, G') + W_1(G', \tilde{G})\\
    & \leq \sum_{j=1}^{k} |p_j - \tilde{p}_j| 2(\diam(\Theta_1) + \diam(\Theta_2))   \\
    & + \sum_{j=1}^{k} \tilde{p}_j (\norm{\theta_{1j} - \tilde{\theta}_{1j}}+ \norm{\theta_{2j} - \tilde{\theta}_{2j}})
     \preccurlyeq \epsilon_n,
\end{align*}
This implies that the covering number
\begin{equation*}
    N(\epsilon_n, \Fcal, \overline{d}_H)\preccurlyeq \left(\dfrac{5}{\epsilon_n}\right)^{k} \left(\dfrac{\diam(\Theta_1)}{\epsilon_n}\right)^{d_1k} \left(\dfrac{\diam(\Theta_2)}{\epsilon_n}\right)^{d_2k}.
\end{equation*}
Taking logarithm of both sides, we obtain $\log N(\epsilon_n, \Fcal, \overline{d}_H) \preccurlyeq \log(1/\epsilon_n)\leq n\epsilon_n^2$. 
Now, we are ready to apply Theorem~\ref{thm:basic-contraction-rate} to conclude the proof.
\end{proof}


\begin{proof}[Proof of Theorem~\ref{thm:posterior-contraction-rate-parameter}]
The proof of this theorem is similar to that of Theorem~\ref{thm:convergence_parameter_estimation}. It is a direct consequence of Theorem~\ref{thm:posterior-contraction-rate-regression}, where we proved that the posterior contraction rate of $\overline{d}_H(f_{G}, f_{G_0})$ is $(\log(n)/n)^{1/2}$, and the inverse bounds (Theorem~\ref{thm:inverse-bounds}), where we showed that $\overline{d}_H(f_{G}, f_{G_0}) \succcurlyeq W_1(G, G_0)$ in the exact-fitted case and $\overline{d}_H(f_{G}, f_{G_0}) \succcurlyeq W_2^2(G, G_0)$ in the over-fitted case.
\end{proof}

Now we are to establish the consistency of the number of parameters and the posterior contraction rate of the latent mixing measure in a Bayesian estimation setting, where the regression mixture model is endowed with a "mixture of finite mixture" prior. The proof makes a crucial usage of Doob's consistency theorem (\cite{ghosal2017fundamentals} Theorem 6.9, or \cite{miller2018detailed} Theorem 2.2).

\begin{proof}[Proof of Theorem~\ref{thm:MFM}] For each latent mixing measure $G$, we write $k(G)$ as its number of (distinct) support points. Recall that we have a prior $\Pi$ on $\mathcal{G} = \cup_{k=1}^{\infty} \mathcal{E}_{k}$, which is a subset of the complete and separable Wasserstein space endowed with metric $W_1$. By assumption, $G$ (and hence $k(G)$) is identifiable.  By Doob's consistency theorem~\cite{doob1949application} (or \cite{ghosal2017fundamentals} Theorem 6.9), there exists $\mathcal{G}_0 \subset \overline{\mathcal{G}}$ such that $\Pi(\mathcal{G}_0) = 1$ and for any $G_0 \in \mathcal{G}_0 \cap \Ecal_{k_0}$, i.e. those $G_0 \in \Gcal_0$ that that have $k_0$ supporting atoms, we have
$$\Pbb(k(G) = k_0 |x^{[n]}, y^{[n]}) = \Ebb[1(k(G) = k_0) |x^{[n]}, y^{[n]}] \to 1(k(G_0) = k_0) = 1,$$ 
almost surely in $\otimes_{i=1}^{\infty}\Pbb_{G_0}$. For the mixture of finite mixtures prior, $K$ represents the (random) number of components. Moreover, by assumption, given $K=k$, the prior distributions on $p=(p_{j})_{j=1}^{k}$ and $(\theta_{j})_{j=1}^{k}$ are absolutely continuous, and set $G = \sum_{j=1}^{k} p_j \delta_{\theta_j}$. Thus, under the induced prior $\Pi$ on the mixing measure, we have $k(G) = K$ for $\Pi$-almost all $G$. This entails that
 there exists $\mathcal{G}'_0 \subset \overline{\mathcal{G}}$ such that $\Pi(\mathcal{G}'_0) = 1$ and for any $G_0 \in \mathcal{G}'_0$ we have
$$\Pbb(k(G) = K |x^{[n]}, y^{[n]}) = 1 \quad \forall n \geq 1 \quad \text{a.s }\otimes_{i=1}^{\infty}\Pbb_{G_0}.$$ 
Now, for any $G_0 \in \Gcal_0 \cap \Gcal'_0$, by the calculus of probabilities
\begin{align*}
\Pbb(K = k_0 |x^{[n]}, y^{[n]}) & \geq \Pbb(K = k_0, k(G) = k_0 |x^{[n]}, y^{[n]}) \\
& = \Pbb(K = k(G), k(G) = k_0 |x^{[n]}, y^{[n]}) \\
& \geq \Pbb(k(G) = k_0 |x^{[n]}, y^{[n]}) - \Pbb(k(G) \neq K |x^{[n]}, y^{[n]})\\
& = \Pbb(k(G) = k_0 |x^{[n]}, y^{[n]}).
\end{align*}
Thus, $\Pbb(K = k_0 |x^{[n]}, y^{[n]}) \to 1 \text{ a.s }\otimes_{i=1}^{\infty}\Pbb_{G_0}$, provided that $G_0 \in \Gcal_0 \cap \Gcal'_{0} \cap \Ecal_{k_0}$. Then, with $\epsilon_n = \sqrt{\log(n)/n}$, we can bound:
\begin{align*}
    \Pi(G: W_1(G, G_0)\succcurlyeq \epsilon_n|x^{[n]}, y^{[n]}) & = \sum_{k=1}^{\infty} \Pi(G \in \Ecal_{k}(\Theta): W_1(G, G_0)\succcurlyeq \epsilon_n|x^{[n]}, y^{[n]})\\
    & \leq \Pi(K \neq k_0 | x^{[n]}, y^{[n]})\\
    & + \Pi(G \in \Ecal_{k_0}(\Theta): W_1(G, G_0)\succcurlyeq \epsilon_n|x^{[n]}, y^{[n]}).
\end{align*}
The first term goes to 0 $\Pbb_{G_0}$-almost surely, thanks to the argument above. For the second term, we apply the first part of Theorem~\ref{thm:posterior-contraction-rate-parameter} to conclude that it tends to 0 in $\Pbb_{G_0}$-probability.

\end{proof}

\section{Proofs of remaining main results}\label{sec:remain-results}

\subsection{Basic inequalities}\label{subsec:basic-inequality-proof}
\begin{proof}[Proof of Lemma~\ref{lem:basic-bounds}] 
Let $G = \sum_{i=1}^{K}p_i \delta_{(\theta_{1i}, \theta_{2i})}$ and recall that $G_0 = \sum_{j=1}^{k_0}p_i^0 \delta_{(\theta^0_{1j}, \theta^0_{2j})}$. To ease the presentation, denote $f_i(y|x) = f(y|h_1(x, \theta_{1i}), h_2(x, \theta_{2i}))$ and $f_j^0(y|x) = f(y|h_1(x, \theta_{1j}^0), h_2(x, \theta_{2j}^0))$ for $i= 1,\dots, K, j=1, \dots, k_0$. We have
\begin{align*}
    \Ebb_X d_{TV}(f_G(\cdot|X), f_{G_0}(\cdot|X)) & = \int_{\Xcal} d\Pbb_X \int_{\Ycal} d\nu(y) \left|\sum_{i=1}^{K} p_i f_i(y|x) - \sum_{j=1}^{k_0} p_j^0 f_j^0(y|x) \right|  \\
    & = \int_{\Xcal} d\Pbb_X \int_{\Ycal} d\nu(y) \left|\sum_{i,j=1}^{K, k_0}  q_{ij} (f_i(y|x) - f_j^0(y|x))\right|\\
    &\leq \sum_{i,j=1}^{K, k_0}  q_{ij} \int_{\Xcal} d\Pbb_X \int_{\Ycal} d\nu(y) \left|  f_i(y|x) - f_j^0(y|x) \right|,
\end{align*}
for any coupling $(q_{ij})_{i,j=1}^{K, k_0}$ of $(p_i)_{i=1}^{K}$ and $(p_j^0)_{j=1}^{k_0}$. But because of the uniform Lipschitz assumption of $f$ and $h_1, h_2$, we have
$$|f_i(y|x) - f_j^0(y|x)| \leq c (|h_1(x, \theta_{1i}) - h_1(x, \theta_{1j}^0)| + |h_2(x, \theta_{2i}) - h_2(x, \theta_{2j}^0)|),$$
and then
$$|f_i(y|x) - f_j^0(y|x)| \leq c c_1 \norm{\theta_{1i} - \theta_{1j}^0} + c c_2 \norm{\theta_{2i} - \theta_{2j}^0}\quad \forall x, y.$$ 
Therefore,
$$\Ebb_X d_{TV}(f_G(\cdot|X), f_{G_0}(\cdot|X)) \leq c \max\{c_1, c_2\} \sum_{i, j=1}^{K, k_0} q_{ij} (\norm{\theta_{1i} - \theta_{1j}^0} + \norm{\theta_{2i} - \theta_{2j}^0}),$$
for all $x, y$. Taking infimum of all feasible $(q_{ij})_{i, j}$ to obtain
$$\Ebb_X d_{TV}(f_G(\cdot|X), f_{G_0}(\cdot|X)) \preccurlyeq W_1(G, G_0).$$
\end{proof}

\begin{remark}
    By inspecting the proof above, we see that the results still hold if we change the uniform Lipschitz condition of $h_1$ and $h_2$ to the integrability of the Lipschitz constants, i.e.  there exist $c_1(x), c_2(x)$ for all $x\in \mathcal{X}$ such that
    $$h_1(x, \theta_1) - h_1(x, \theta_1') \leq c_1(x) \norm{\theta_1- \theta_1'},\quad h_2(x, \theta_2) - h_2(x, \theta_2') \leq c_2(x) \norm{\theta_2 - \theta_2'},$$
    for all $\theta_1, \theta_2, \theta_1',\theta_2'$, and $\Ebb_X c_1(X) < \infty, \Ebb_X c_2(X) < \infty$. This condition is weaker than the uniformly Lipschitz condition in $x$. 
\end{remark}

\subsection{Identifiability results}\label{subsec:identifiability-example-proofs}
\begin{proof}[Proof of Proposition~\ref{prop:identifiable-f}]
(a), (b): Can be found in \cite{chen1995optimal,Ho-Nguyen-EJS-16}.

(c) First, we will establish the first order identifiability condition when $2K\leq N+1$. Suppose that $q_1, q_2,\dots q_K \in [0,1]$ are distinct and there exist $\alpha_1, \dots, \alpha_K$, $\beta_1, \dots, \beta_K$ such that
    \begin{equation}\label{eq:first-order-contradiction}
        \alpha_1 \Bin (y | q_1) + \dots + \alpha_K \Bin (y | q_K) + \beta_1 \dfrac{\partial}{\partial  q}\Bin (y | q_1) + \dots + \beta_K \dfrac{\partial}{\partial  q}\Bin(y | q_K) = 0,
    \end{equation}
    for all $y=0,1,\dots, N$. Direct calculation gives 
    \begin{equation}
        \sum_{i=1}^{K} q_i^{y}(1-q_i)^{N-y} \alpha_{i} + \sum_{i=1}^{K} \dfrac{\partial}{\partial  q} q_i^{y}(1-q_i)^{N-y} \beta_{i} = 0, \quad \forall y=0,\dots, N.
    \end{equation}
Because this is a system of linear equations of $(\alpha_i, \beta_i)_{i=1}^{K}$, it suffices to show that the following $(N+1) \times 2K$ matrix has independent columns
\begin{equation*}
    \begin{pmatrix}
        (1-q_1)^{N} & \cdots & (1-q_K)^{N} & \frac{\partial}{\partial q} (1-q_1)^{N}& \cdots & \frac{\partial}{\partial q}(1-q_K)^{N} \\
        \vdots & \ddots & \vdots & \vdots & \ddots & \vdots \\
        q_1^{N} & \cdots & q_K^{N} & \frac{\partial}{\partial q} q_1^{N}& \cdots & \frac{\partial}{\partial q}q_K^{N}
    \end{pmatrix}.
\end{equation*}
Multiplying this matrix with the following upper triangular matrix
\begin{equation*}
    \begin{pmatrix}
        1 & \binom{N}{1} & \binom{N}{2} & \cdots &  \binom{N}{N}\\
        0 & 1 & \binom{N-1}{1} & \cdots & \binom{N-1}{N-1} \\
        0 & 0 & 1 & \cdots & \binom{N-2}{N-2} \\
        \vdots & \vdots & \vdots & \ddots & \vdots \\
        0 & 0 & 0 & \cdots & 1
    \end{pmatrix},
\end{equation*}
we only need to prove the following $(N+1) \times 2K$ matrix 
\begin{equation*}
    \begin{pmatrix}
        1 & \cdots & 1 & 0& \cdots & 0 \\
        q_1 & \cdots & q_K & 1& \cdots & 1 \\
        q_1^2 & \cdots & q_2^2 & 2 q_1& \cdots & 2q_2 \\
        \vdots & \ddots & \vdots & \vdots & \ddots & \vdots \\
        q_1^{N} & \cdots & q_K^{N} & N q_1^{N-1}& \cdots & Nq_K^{N-1}
    \end{pmatrix}
\end{equation*}
has independent columns. Because $2K \leq N+1$, it suffices to prove that $\det(D_1) \neq 0$, for
\begin{equation}\label{eq:matrix-vandermonde-d1}
    D_1 =  \begin{pmatrix}
        1 & \cdots & 1 & 0& \cdots & 0 \\
        q_1 & \cdots & q_K & 1& \cdots & 1 \\
        q_1^2 & \cdots & q_2^2 & 2 q_1& \cdots & 2q_2 \\
        \vdots & \ddots & \vdots & \vdots & \ddots & \vdots \\
        q_1^{2K-1} & \cdots & q_K^{2K-1}  & (2K-1) q_1^{2K-2}& \cdots & (2K-1)q_K^{2K-2}
    \end{pmatrix}.
\end{equation}
In the following, we will prove that  $\det (D_1) = \prod_{1\leq i < j \leq K}(q_i-q_j)^4$, so that it is different from 0 if $q_1, \dots, q_N$ are distinct as in our assumption. We borrow an idea in calculating the determinant of the Vandermonde matrix. Note that $\det (D_1)$ is a polynomial of $q_1, q_2, \dots, q_K$, with the degree of each $q_i$ no more than $4K-4$. Let us treat $q_1 = x$ as a variable, while $q_2, \dots, q_K$ as constants, and prove that 
{\fontsize{8}{10}\selectfont 
\begin{equation*}
    f_1(x) = \det \begin{pmatrix}
        1 & 1 & \cdots & 1 & 0 & 0 & \cdots & 0 \\
        x & q_2 & \cdots & q_K & 1& 1 & \cdots & 1 \\
        x^2 & q_2^2 & \cdots & q_K^2 & 2x & 2q_2 & \cdots & 2q_K \\
        \vdots & \vdots & \ddots & \vdots & \vdots & \vdots & \ddots & \vdots \\
        x^{2K-1} & q_2^{2K-1} & \cdots & q_K^{2K-1}  & (2K-1) x^{2K-2}& (2K-1) q_2^{2K-2}&  \cdots & (2K-1)q_K^{2K-2}
    \end{pmatrix}
\end{equation*}}
is a polynomial having degree $4(K-1)$ of $x$ and can be factorized as $\prod_{i=2}^{K} (x-q_i)^4$. It suffices for us to show $f(x), f'(x), f''(x), f^{(3)}(x)$ all attains $q_2$ as solutions, and similar for other $q_i$'s. It can be seen that $f_1(q_2)$ is a determinant of a matrix with identical first two columns, therefore $f_1(q_2)=0$. For the derivative of $f_1$, we use the derivative rule for product $(fg)' = f' g + g'f$ to have that $f_1'(x)$ equals\\
{\fontsize{9}{10}\selectfont 
$\det \begin{pmatrix}
        0 & 1 & \cdots & 1 & 0 & 0 & \cdots & 0 \\
        1 & q_2 & \cdots & q_K & 1& 1 & \cdots & 1 \\
        2x & q_2^2 & \cdots & q_K^2 & 2x & 2q_2 & \cdots & 2q_K \\
        \vdots & \vdots & \ddots & \vdots & \vdots & \vdots & \ddots & \vdots \\
        (2K-1)x^{2K-2} & q_2^{2K-1} & \cdots & q_K^{2K-1}  & (2K-1) x^{2K-2}& (2K-1) q_2^{2K-2}&  \cdots & (2K-1)q_K^{2K-2}
        \end{pmatrix}$\\
    $+$\\
    $ \det \begin{pmatrix}
        1 & 1 & \cdots & 1 & 0 & 0 & \cdots & 0 \\
        x & q_2 & \cdots & q_K & 0& 1 & \cdots & 1 \\
        x^2 & q_2^2 & \cdots & q_K^2 & 2 & 2q_2 & \cdots & 2q_K \\
        \vdots & \vdots & \ddots & \vdots & \vdots & \vdots & \ddots & \vdots \\
        x^{2K-1} & q_2^{2K-1} & \cdots & q_K^{2K-1}  & (2K-1)(2K-2) x^{2K-3}& (2K-1) q_2^{2K-2}&  \cdots & (2K-1)q_K^{2K-2}
        \end{pmatrix}$}
As the first matrix has identical first and $(K+1)$-th columns, its determinant equals 0. Hence, $f_1'(x)$ equals {\fontsize{9}{10}\selectfont 
\begin{eqnarray*}
\det \begin{pmatrix}
        1 & 1 & \cdots & 1 & 0 & 0 & \cdots & 0 \\
        x & q_2 & \cdots & q_K & 0& 1 & \cdots & 1 \\
        x^2 & q_2^2 & \cdots & q_K^2 & 2 & 2q_2 & \cdots & 2q_K \\
        \vdots & \vdots & \ddots & \vdots & \vdots & \vdots & \ddots & \vdots \\
        x^{2K-1} & q_2^{2K-1} & \cdots & q_K^{2K-1}  & (2K-1)(2K-2) x^{2K-3}& (2K-1) q_2^{2K-2}&  \cdots & (2K-1)q_K^{2K-2}
        \end{pmatrix}
\end{eqnarray*}}
Now consider $f_1'(q_2)$. It is the determinant of a matrix that has identical first two columns, so $f_1'(q_2)=0$. Continuing to apply the derivative rule for products of functions, we have that $f_1''(x)$ equals
{\fontsize{9}{10}\selectfont 
\begin{align*}
    & \det \begin{pmatrix}
        0 & 1 & \cdots & 1 & 0 & 0 & \cdots & 0 \\
        1 & q_2 & \cdots & q_K & 0& 1 & \cdots & 1 \\
        2x & q_2^2 & \cdots & q_K^2 & 2 & 2q_2 & \cdots & 2q_K \\
        \vdots & \vdots & \ddots & \vdots & \vdots & \vdots & \ddots & \vdots \\
        (2K-1)x^{2K-2} & q_2^{2K-1} & \cdots & q_K^{2K-1}  & \prod_{i=1}^{2}(2K-i) x^{2K-3}& (2K-1) q_2^{2K-2}&  \cdots & (2K-1)q_K^{2K-2}
        \end{pmatrix}\\
    & + \det \begin{pmatrix}
        1 & 1 & \cdots & 1 & 0 & 0 & \cdots & 0 \\
        x & q_2 & \cdots & q_K & 0& 1 & \cdots & 1 \\
        x^2 & q_2^2 & \cdots & q_K^2 & 0 & 2q_2 & \cdots & 2q_K \\
        \vdots & \vdots & \ddots & \vdots & \vdots & \vdots & \ddots & \vdots \\
        x^{2K-1} & q_2^{2K-1} & \cdots & q_K^{2K-1}  & \prod_{i=1}^{3}(2K-i) x^{2K-4}& (2K-1) q_2^{2K-2}&  \cdots & (2K-1)q_K^{2K-2}
        \end{pmatrix}
\end{align*}}
Substitute $x = q_2$ in the formula above, the first matrix has identical first and $(K+2)$-th column, and the second matrix has identical first two columns. Hence, $f_1''(q_2)=0$. Continue applying derivative one more time, we have $f_1^{(3)}(x)$ equals
{\fontsize{9}{10}\selectfont 
\begin{align*}
    & \det \begin{pmatrix}
        0 & 1 & \cdots & 1 & 0 & 0 & \cdots & 0 \\
        0 & q_2 & \cdots & q_K & 0& 1 & \cdots & 1 \\
        2 & q_2^2 & \cdots & q_K^2 & 2 & 2q_2 & \cdots & 2q_K \\
        \vdots & \vdots & \ddots & \vdots & \vdots & \vdots & \ddots & \vdots \\
        \prod_{i=1}^{2}(2K-i) x^{2K-3} & q_2^{2K-1} & \cdots & q_K^{2K-1}  & \prod_{i=1}^{2}(2K-i) x^{2K-3}& (2K-1) q_2^{2K-2}&  \cdots & (2K-1)q_K^{2K-2}
        \end{pmatrix}\\
        & +\det \begin{pmatrix}
        0 & 1 & \cdots & 1 & 0 & 0 & \cdots & 0 \\
        1 & q_2 & \cdots & q_K & 0& 1 & \cdots & 1 \\
        2x & q_2^2 & \cdots & q_K^2 & 2 & 2q_2 & \cdots & 2q_K \\
        \vdots & \vdots & \ddots & \vdots & \vdots & \vdots & \ddots & \vdots \\
        (2K-1)x^{2K-2} & q_2^{2K-1} & \cdots & q_K^{2K-1}  & \prod_{i=1}^{3}(2K-i) x^{2K-4}& (2K-1) q_2^{2K-2}&  \cdots & (2K-1)q_K^{2K-2}
        \end{pmatrix}\\
    & + \det \begin{pmatrix}
        0 & 1 & \cdots & 1 & 0 & 0 & \cdots & 0 \\
        1 & q_2 & \cdots & q_K & 0& 1 & \cdots & 1 \\
        2x & q_2^2 & \cdots & q_K^2 & 0 & 2q_2 & \cdots & 2q_K \\
        \vdots & \vdots & \ddots & \vdots & \vdots & \vdots & \ddots & \vdots \\
        (2K-1)x^{2K-2} & q_2^{2K-1} & \cdots & q_K^{2K-1}  & \prod_{i=1}^{3}(2K-i) x^{2K-4}& (2K-1) q_2^{2K-2}&  \cdots & (2K-1)q_K^{2K-2}
        \end{pmatrix}\\
        & + \det \begin{pmatrix}
        1 & 1 & \cdots & 1 & 0 & 0 & \cdots & 0 \\
        x & q_2 & \cdots & q_K & 0& 1 & \cdots & 1 \\
        x^2 & q_2^2 & \cdots & q_K^2 & 0 & 2q_2 & \cdots & 2q_K \\
        \vdots & \vdots & \ddots & \vdots & \vdots & \vdots & \ddots & \vdots \\
        x^{2K-1} & q_2^{2K-1} & \cdots & q_K^{2K-1}  & \prod_{i=1}^{4}(2K-i) x^{2K-5}& (2K-1) q_2^{2K-2}&  \cdots & (2K-1)q_K^{2K-2}
        \end{pmatrix}.
\end{align*}
}
The first matrix has two identical columns, so its determinant is 0. Meanwhile, when we substitute $x=q_2$ into the other three matrices, each also has identical columns. Hence, $f^{(3)}(q_2) = 0$. We obtain that $f_1(x) \propto \prod_{i=2}^{K}(x-q_i)^{4}$. By treating $q_2, \dots, q_K$ as variables respectively and applying the same argument, we have 
\begin{equation*}
    \det(D_1) = \prod_{1\leq i < j\neq K} (q_i - q_j)^4\neq 0,
\end{equation*}
whenever $q_1, q_2, \dots, q_{K}$ are distinct.

The proof for establishing the second order identifiability when $3K\leq N+1$ is similar, where the determinant of the derived $3K\times 3K$ matrix is $\prod_{1\leq i < j\neq K} (q_i - q_j)^6\neq 0$.

(d) For the family of negative binomial distributions, the density is given as
$$\NB(y|\mu,\phi)= \dfrac{\Gamma(\phi + y)}{\Gamma(\phi) y!} \left(\dfrac{\mu}{\phi + \mu}\right)^{y} \left(\dfrac{\phi}{\phi + \mu}\right)^{\phi}.$$ 
Suppose that $\mu_1,...,\mu_K$ are distinct, and there exist $\alpha_1,...,\alpha_K, \beta_1,..., \beta_K, \gamma_1,...,\gamma_K$ such that for every $y\in \mathbb{N}$
\begin{align}\label{eq: second order NegBin}
    \sum_{i=1}^K\alpha_i \NB(y | \mu_i,\phi)+\sum_{i=1}^{K} \beta_i \dfrac{\partial}{\partial  \mu} \NB(y | \mu_i,\phi) +\sum_{i=1}^K \gamma_i \dfrac{\partial^2}{\partial  \mu^2} \NB(y | \mu_i,\phi) = 0.
\end{align}
We will show that $\alpha_1=\dots=\alpha_K=\beta_1=\dots=\beta_K=\gamma_1=\dots\gamma_K=0$.
Indeed, Eq. \eqref{eq: second order NegBin} is simplified as below
\begin{align}
    \sum_{i=1}^K \alpha_i\left(\dfrac{\mu_i}{\phi+\mu_i}\right)^y\left(\dfrac{\phi}{\phi+\mu_i}\right)^\phi+\sum_{i=1}^K \beta_i\left(\dfrac{\mu_i}{\phi+\mu_i}\right)^{y-1}\left(\dfrac{\phi}{\phi+\mu_i}\right)^{\phi+1}\dfrac{y-\mu_i}{\phi+\mu_i} \nonumber\\
   \label{eq:simplfy second order NegBin} +\sum_{i=1}^K\gamma_i \left(\dfrac{\mu_i}{\phi+\mu_i}\right)^{y-2}\left(\dfrac{\phi}{\phi+\mu_i}\right)^{\phi+1}\dfrac{1}{(\phi+\mu_i)^3}[\phi(y-\mu_i)^2-y(2\mu_i+\phi)+\mu_i^2]\nonumber\\
   =0,
\end{align}
for all $y\in \Nbb$. Without loss of generality, assume $\mu_1$ is the largest value in the set of $\{\mu_1,...,\mu_K\}$. This implies that $\dfrac{\mu_1}{\phi+\mu_1}$ is also the largest value in the set of $\left\{\dfrac{\mu_1}{\phi+\mu_1},...,\dfrac{\mu_K}{\phi+\mu_K}\right\}$. Dividing both sides of Eq. \eqref{eq:simplfy second order NegBin} by $\left(\dfrac{\mu_1}{\phi+\mu_1}\right)^{y-2}[\phi(y-\mu_1)^2-y(2\mu_1+\phi)+\mu_1^2]=\left(\dfrac{\mu_1}{\phi+\mu_1}\right)^{y-2}A_1(y)$, we obtain
\begin{align}
   &\sum_{i=1}^K \alpha_i\left(\dfrac{\mu_i(\phi+\mu_1)}{\mu_1(\phi+\mu_i)}\right)^{y-2}\left(\dfrac{\mu_i}{\phi+\mu_i}\right)^2\left(\dfrac{\phi}{\phi+\mu_i}\right)^\phi\dfrac{1}{A_1(y)} \nonumber\\
   +&\sum_{i=1}^K \beta_i\left(\dfrac{\mu_i(\phi+\mu_1)}{\mu_1(\phi+\mu_i)}\right)^{y-2}\left(\dfrac{\mu_i}{\phi+\mu_i}\right) \left(\dfrac{\phi}{\phi+\mu_i}\right)^{\phi+1}\dfrac{y-\mu_i}{(\phi+\mu_i)A_1(y)} \nonumber \\
   \label{eq:second order NegBin3} +&\sum_{i=2}^K\gamma_i \left(\dfrac{\mu_i(\phi+\mu_1)}{\mu_1(\phi+\mu_i)}\right)^{y-2}\left(\dfrac{\phi}{\phi+\mu_i}\right)^{\phi+1}\dfrac{A_i(y)}{(\phi+\mu_i)^3A_1(y)}\nonumber \\
   & +\gamma_1\left(\dfrac{\phi}{\phi+\mu_1}\right)^{\phi+1}\dfrac{1}{(\phi+\mu_1)^3}=0, \forall y \in \mathbb{N}.  
\end{align}
Let $y \rightarrow \infty$ in \eqref{eq:second order NegBin3}, we get $\gamma_1=0$. After dropping $\gamma_1$ in \eqref{eq:simplfy second order NegBin}, the remaining terms of the equation is divided by $\left(\dfrac{\mu_1}{\phi+\mu_1}\right)^{y-1}(y-\mu_1)$, we have
\begin{align}
    &\sum_{i=1}^K \alpha_i\left(\dfrac{\mu_i(\phi+\mu_1)}{\mu_1(\phi+\mu_i)}\right)^{y-1}\left(\dfrac{\mu_i}{\phi+\mu_i}\right)\left(\dfrac{\phi}{\phi+\mu_i}\right)^\phi\dfrac{1}{y-\mu_1}\nonumber\\
   +&\sum_{i=2}^K \beta_i\left(\dfrac{\mu_i(\phi+\mu_1)}{\mu_1(\phi+\mu_i)}\right)^{y-1} \left(\dfrac{\phi}{\phi+\mu_i}\right)^{\phi+1}\left(\dfrac{y-\mu_i}{y-\mu_1}\right) \dfrac{1}{\phi+\mu_i} \nonumber \\
   \label{eq:second order NegBin4} +&\sum_{i=2}^K\gamma_i \left(\dfrac{\mu_i(\phi+\mu_1)}{\mu_1(\phi+\mu_i)}\right)^{y-2}\left(\dfrac{\phi}{\phi+\mu_i}\right)^{\phi+1}\dfrac{1}{(\phi+\mu_i)^3}\dfrac{A_i(y)}{y-\mu_1} \nonumber \\
   & + \beta_1\left(\dfrac{\phi}{\phi+\mu_1}\right)^{\phi+1}\dfrac{1}{\phi+\mu_1}=0, \forall y \in \mathbb{N}.  
\end{align}
Taking the limit $y\rightarrow\infty$ both sides of Eq. \eqref{eq:second order NegBin4}, we get $\beta_1=0$. Continuing this procedure, we set $\beta_1=0$ and $\gamma_1=0$ in Eq. \eqref{eq: second order NegBin}, then divide $\left(\dfrac{\mu_1}{\phi+\mu_1}\right)^{y}$ on both sides of the remaining equation. The final result leads to the following:
\begin{align}\label{eq:second order NegBin5} 
    &\sum_{i=2}^K \alpha_i\left(\dfrac{\mu_i(\phi+\mu_1)}{\mu_1(\phi+\mu_i)}\right)^{y}\left(\dfrac{\phi}{\phi+\mu_i}\right)^\phi  \nonumber \\
   & +\sum_{i=2}^K \beta_i\left(\dfrac{\mu_i(\phi+\mu_1)}{\mu_1(\phi+\mu_i)}\right)^{y}\left(\dfrac{\mu_i}{\phi+\mu_i}\right)^{-1} \left(\dfrac{\phi}{\phi+\mu_i}\right)^{\phi+1}\dfrac{y-\mu_i}{\phi+\mu_i} \nonumber \\
   & + \sum_{i=2}^K\gamma_i\left(\dfrac{\mu_i(\phi+\mu_1)}{\mu_1(\phi+\mu_i)}\right)^{y} \left(\dfrac{\mu_i}{\phi+\mu_i}\right)^{-2}\left(\dfrac{\phi}{\phi+\mu_i}\right)^{\phi+1} \dfrac{A_i(y)}{(\phi+\mu_i)^3} \nonumber \\
   & +\alpha_1\left(\dfrac{\phi}{\phi+\mu_1}\right)^\phi=0, \forall y \in \mathbb{N}.  
\end{align}
It is clear to see that $\alpha_1=0$ when $y$ approaches $\infty$ in Eq. \eqref{eq:second order NegBin5}.
We have shown that $\alpha_1,\beta_1,\gamma_1=0$. Inductively, we obtain that $\alpha_i,\beta_i,\gamma_i=0$ for $i=2,...,K$.
\end{proof}

\begin{proof}[Proof of Proposition~\ref{prop:family complete_ident}]

Since $\Pbb_X$ is absolutely continuous with respect to the Lebesgue measure on $\Rbb^p$, it is sufficient to prove the result in this proposition with respect to the Lebesgue measure.

    (a) We will prove this part by applying an inductive argument with respect to $p$ (dimension of covariate $x$).
    Suppose $p=1$. For $\theta \neq \theta'$, the equation $h(x,\theta)=h(x,\theta')$ is a non-trivial polynomial equation, it only has a finite number of solutions. Thus the set of solution has Lebesgue measure zero, so we have $h(x,\theta) \neq h(x,\theta')$ a.s.
    
    Assume now that the proposition is valid up to the parameter space dimension $p-1$.  Now we prove that it is correct for $p$. Using a similar argument as above, it suffices to show that the set of solutions for any non-trivial polynomial has zero measure. Indeed, consider any such polynomial of degree $d$ of variable $x \in \Rbb^p$, we can write $x=(X_{p-1},x_p)$, where $X_{p-1} \in \Rbb^{p-1}$ and $x_p \in \Rbb$. The polynomial then can be written as :
    \begin{align}
        \sum_{j=0}^{d} p_j(X_{p-1})x_p^i =0\label{eq:proposition 3},
    \end{align}
    where $(p_j(X_{p-1}))_{j=0}^d$ are polynomial of $X_{p-1}\in \Rbb^{p-1}$, at least one of which is non-trivial.
    
    Now, partition the set $Z$ of the solutions for this polynomial into two measurable sets $Z=A \cup B$, where 
    \begin{align*}
        A &= \{(X_{p-1}, x_p): p_j(X_{p-1})=0 \mbox{ } \forall j = 1, ..., d\}\\
        B &= \{(X_{p-1}, x_p):\mbox{at least one  } p_j(X_{p-1})\neq 0, \mbox{ and } x_p \mbox{ satisfies \eqref{eq:proposition 3}.} \}
    \end{align*}
    The Lebesgue measure of set $A$ is $0$ using the induction hypothesis. While for any $x= (X_{p-1}, x_p)$ in set $B$, for each such $X_{p-1}$, there exist only a finite number of $x_p \in \Rbb$ to satisfy \eqref{eq:proposition 3}, which has zero Lebesgue measure in $\Rbb$. Therefore, we can use Fubini's theorem to deduce that the measure of $B$ is also zero. Thus, $Z$ has measure zero. 
    We have established that $h(x,\theta)$ is completely identifiable for any polynomial $h(x,\theta)$.
    
    Turning to the verification of Assumption (A5.), since $\dfrac{\partial}{\partial \theta} h(x,\theta)$ is again a non-trivial polynomial of $x$, (A5.) is also satisfied using the same argument above.
    
    (b) Similar to part (a), we only need to prove that a non-trivial (not all coefficients are 0) trigonometric polynomial of $x$:
    \begin{equation}
        a_0 + \sum_{n=1}^{d} b_n \cos(n x) + \sum_{n=1}^{d} \sin(n x) = 0
    \end{equation}
    has a countable number of solutions.  
    Write $\cos(nx) = \dfrac{1}{2}(e^{inx} + e^{-inx}), \sin(nx) = \dfrac{1}{2i}(e^{inx} - e^{-inx})$, where $i$ is the imaginary unit, we can rewrite a non-trivial trigonometric polynomial above as
    \begin{equation}
        a_0 + \sum_{n=1}^{d} \tilde{b}_n e^{inx} + \sum_{n=1}^{d} \tilde{c}_n e^{-inx} = 0,
    \end{equation}
    where $\tilde{b}_n$ and $\tilde{c}_n$ are computed from $b_n, c_n$, and the tuple $(a_0, \tilde{b}_n, \tilde{c}_n)$ is non-trivial. Set $y = e^{ix}$, this becomes a polynomial in $y \in \mathbb{C}$, which has a finite number of solutions, by the fundamental theorem of algebra. Combining this with the fact that $e^{ix} =y$ only has a countable solution in $x$, we arrive at the conclusion. The condition spelled out in Assumption (A5.) also holds because the derivative of a trigonometric polynomial is still of the same form.

    (c) Similar to above, we express a non-trivial mixture of polynomials and trigonometric polynomials in the form:
    \begin{equation}
        \sum_{n=0}^{d} a_n x^{n} + \sum_{n=1}^{d} \tilde{b}_n e^{inx} + \sum_{n=1}^{d} \tilde{c}_n e^{-inx} = 0,
    \end{equation}
    which is a holomorphic function in $\mathbb{C}$. This function is known to have an isolated set of solutions, which has zero measure \cite{stein2010complex}. Thus, these functions are completely identifiable. Conditions in Assumption (A5.) also follow because the derivative of a function of this type is still of the same form.
    
    (d) Given $h(x,\theta) = g(p(x,\theta))$, where $g$ is diffeomorphic.
    
    To verify the complete identifiability condition, it can be seen that for $\theta \neq \theta'$:
    $h(x,\theta) = h(x,\theta') \Leftrightarrow p(x,\theta) = p(x,\theta')$,
    so that the complete identifiability of $h$ can be deduced from what of $p$.
    
    To verify Assumption (A5.), note that
    $\beta^{\top}\dfrac{\partial h(x,\theta)}{\partial \theta} = \beta^{\top} g'(p(x,\theta)) \dfrac{\partial p(x,\theta)}{\partial \theta}$.
    Since $g'(p(x,\theta))\neq 0$ (as $g$ is a diffeomorphism), the two equations below are equivalent.
    $$\beta^{\top}\dfrac{\partial h(x,\theta)}{\partial \theta} = 0 \Longleftrightarrow \beta^{\top}\dfrac{\partial p(x,\theta)}{\partial \theta} = 0. $$
    Hence, $h$ satisfies assumption (A5.) if $p$ does.

\end{proof}

The following result illustrates the discussion in Section~\ref{sec:experiments} by showing that a mixture of binomial regression model may be strongly identifiable even though the (unconditional) mixture of binomial distributions is not identifiable in even in the classical sense.

\begin{proposition}\label{prop:mixture-logistic-regression}
    Suppose that the link function $h(x, \theta) = \sigma(\theta x):= 1/(1+e^{\theta x})$ for $\theta, x\in \Rbb$, and the density kernel $f(y) = \textrm{Bin}(y|1, q)$, where $q=h(x, \theta)$. Moreover, the support of $\Pbb_X$ contains an open set in $\Rbb$.  Then, the mixture of two binomial regression components associated with the mixing measure $G = p_1 \delta_{\theta_1} + p_2 \delta_{\theta_2}$,
    where $\theta_1 + \theta_2 \neq 0$, is strongly identifiable in the first order.
\end{proposition}
\begin{proof}
Consider $\theta_1 \neq \pm \theta_2$. Suppose that for some $a_1, a_2, b_1, b_2 \in \Rbb$ we have
\begin{align*}
    & a_1 \textrm{Bin}(y|1, h(x, \theta_1)) + a_2 \textrm{Bin}(y|1, h(x, \theta_2)) \\ 
    & + b_1 \dfrac{\partial}{\partial \theta} \textrm{Bin}(y|1, h(x, \theta_1)) + b_2 \dfrac{\partial}{\partial \theta} \textrm{Bin}(y|1, h(x, \theta_2)) = 0,
\end{align*}
for all $y = 0, 1,\text{ and } x\in \textrm{supp}(\Pbb_X)$. Then we will show that $a_1 = a_2 = b_1 = b_2 = 0$. 
Denote $\sigma_i(x) = h(x, \theta_i)$, we have $\textrm{Bin}(1|1, h(x, \theta_i)) = \sigma_i = 1- \textrm{Bin}(0|1, h(x, \theta_i))$. Besides, $\dfrac{\partial}{\partial \theta} \textrm{Bin}(1|1, h(x, \theta_i)) = x \sigma_i(x)(1-\sigma_i(x)) = -\dfrac{\partial}{\partial \theta} \textrm{Bin}(0|1, h(x, \theta_i))$, so that
\begin{equation}
    a_1 + a_2 = 0,
\end{equation}
and 
\begin{equation}\label{eq:ident-Binom}
    a_1 \sigma_1(x) + a_2 \sigma_2(x) + b_1 x \sigma_1(x) (1-\sigma_1(x)) + b_2 x \sigma_2(x) (1-\sigma_2(x)) = 0, \, \forall \, x\in \textrm{supp}(\Pbb_X). 
\end{equation}
Because Eq.~\eqref{eq:ident-Binom} satisfies for all $x$ in an open set, and it is an analytic function of $x$, it satisfies for all $x\in \Rbb$ (identity theorem) \cite{stein2010complex}. Without the loss of generality, we assume $\theta_1 < \theta_2$. If $0\leq \theta_1 < \theta_2$, then   
by dividing both sides of Eq.~\eqref{eq:ident-Binom} by $\sigma_1(x) x$, one obtains
\begin{equation*}
    \dfrac{a_1}{x} + a_2 \dfrac{1+\exp(\theta_1 x)}{1+\exp(\theta_2 x)} \dfrac{1}{x} + b_1  (1-\sigma_1(x)) + b_2  (1-\sigma_2(x)) \dfrac{1+\exp(\theta_1 x)}{1+\exp(\theta_2 x)} = 0, \quad \forall \, x\in \Rbb.
\end{equation*}
Let $x\to \infty$, we have $b_1 (1-\sigma_1(x)) \to b_1/2$ or $b_1$ (depending on whether $\theta_1 = 0$ or $\theta_1 > 0$) and all other terms go to 0. Hence $b_1 = 0$. Next, dividing both sides of  Eq.~\eqref{eq:ident-Binom} by $\sigma_1(x)$, one obtains
\begin{equation*}
    a_1 + a_2 \dfrac{1+\exp(\theta_1 x)}{1+\exp(\theta_2 x)} + b_2  (1-\sigma_2(x)) \dfrac{(1+\exp(\theta_1 x))x}{1+\exp(\theta_2 x)} = 0, \quad \forall \, x\in \Rbb.
\end{equation*}
Let $x\to \infty$, we have $a_1 = 0$. Therefore,
\begin{equation*}
    a_2  + b_2 \dfrac{\exp(\theta_2 x)}{1+\exp(\theta_2 x)} x = 0 \quad \forall \, x\in \Rbb,
\end{equation*}
which implies $a_2 = b_2 = 0$.  In the other case where $\theta_1 < 0 < \theta_2$, we let $x \to \infty$ in Eq.~\eqref{eq:ident-Binom} and notice that $\theta_1(x) \to 1, \theta_2(x)\to 0$ then $a_1 = 0$. Similarly let $x\to -\infty$, we have $a_2 = 0$. Then Eq.~\eqref{eq:ident-Binom} becomes
\begin{equation*}
    b_1 h(x, \theta_1) (1-h(x, \theta_1)) + b_2 h(x, \theta_2) (1-h(x, \theta_2)) = 0.
\end{equation*}
But notice that $h(x, \theta_1) (1-h(x, \theta_1)) = h(x, -\theta_1) (1-h(x, -\theta_1))$, so by letting $\theta'_1 = - \theta_1$, we are back to the case $\theta_1', \theta_2 > 0$. Similar to the case $\theta_1, \theta_2 < 0$, we can transform $\theta_1\mapsto -\theta_1, \theta_2 \mapsto \theta_2'$ to go back to the first case $\theta_1, \theta_2 > 0$ (because Eq.~\eqref{eq:ident-Binom} satisfies for all $x\in \Rbb$). Hence, in all cases we have $a_1 = a_2 = b_1 = b_2 = 0$. Hence, strong identifiability in the first order is established.
\end{proof}
\begin{remark}
    \begin{enumerate}
        \item The fact that mixture of Binomial distributions is not identifiable in general can be seen from a simple example: $0.5 \textrm{Bin}(y|1, q_1) + 0.5 \textrm{Bin}(y|1, q_2) = 0.5 \textrm{Bin}(y|1, q_1 + \epsilon) + 0.5 \textrm{Bin}(y|1, q_2-\epsilon)$ for all valid $\epsilon>0$. That is also the reason why one cannot include the intercept parameter in the definition of $h$ in the proposition above.
        \item The proof technique of this proposition is to perform analytic continuation so that the identifiability equation satisfies for all $x\in \Rbb$ then we can examine the limits $x\to \pm \infty$. Extending this proof technique mixture of more components (more than 2) is generally more challenging because several components can have the same limit as $x\to \pm \infty$. We once more highlight the usefulness of Theorem~\ref{thm:identifiable-equivalent} and Theorem~\ref{thm:inverse-bounds} for providing the guarantee for a large class of identifiable mixture densities.
    \end{enumerate}
\end{remark}

\subsection{Minimax bound for mean-dispersion negative binomial regression mixtures}\label{subsec:Minimax-proof}

\begin{proof}[Proof of Theorem~\ref{thm:minimax-NB}]

\textbf{Step 1.} We will prove that for any $k\geq 2$ there exist $G_0\in \Ecal_{k}(\Theta)$ and a sequence $G_n\in \Ecal_{k}(\Theta)$ such that:
\begin{equation}\label{eq:minimax-NB-bound}
    W_r(G_n, G_0)\to 0, \quad \sup_{x} d_H(f_{G_n}(\cdot|x), f_{G_0}(\cdot|x)) = O(W_r^{2r}(G_n, G_0)).
\end{equation}
Intuitively, we want to choose $G_0$ to be in the pathological case described in equation~\eqref{eq:weak-ident-NB}. In particular, choose $G_0 = \sum_{j=1}^{k} p_j^0 \delta_{\beta_{j}^0, \phi_j^0}$ where $\beta_{j}^{0} = (\beta_{jt}^{0})_{t=0}^{p} \in \Rbb^{p+1}$ such that $\phi_{2}^0 = \phi_{1}^0 + 1, \beta_{20}^{0} = \beta_{10}^{0} + \log\left(\dfrac{\phi_2^0}{\phi_1^0}\right), \beta_{2i}^{0} = \beta_{1i}^0$ for all $i = 1,\dots, p$. Let $\mu_{1}^0 \equiv \mu_1^0(x) = \exp((\beta_{1}^0)^{\top} x), \mu_{2}^0 \equiv \mu_2^0(x) = \exp((\beta_{2}^0)^{\top} x)$, we have $\dfrac{\mu_1^0}{\phi_1^0} = \dfrac{\mu_2^0}{\phi_2^0}$. A combination of chain rule with equation~\eqref{eq:weak-ident-NB} yields:
\begin{align}\label{eq:weakly-ident-NB-proof}
\begin{split}
    \dfrac{\partial }{\partial \beta_{10}^0} \NB(y|\exp((\beta_1^0)^{\top} x), \phi_1^0) & = \dfrac{d \mu_{1}^0 }{d \beta_{10}^0} \dfrac{\partial }{\partial \mu_{1}^0} \NB(y|\mu_1^0, \phi_1^0)\bigg|_{\mu_1^0 = \exp((\beta_1^0)^{\top} x)} \\ 
    & \hspace{-3cm} = \mu_1^0\left( \dfrac{\phi_1^0}{\mu_1^0} \NB(y|\exp((\beta_2^0)^{\top} x), \phi_2^0) - \dfrac{\phi_1^0}{\mu_1^0} \NB(y|\exp((\beta_1^0)^{\top} x), \phi_1^0) \right)\\
    & \hspace{-3cm} = \phi_1^0 \NB(y|\exp((\beta_2^0)^{\top} x), \phi_2^0) - \phi_1^0 \NB(y|\exp((\beta_1^0)^{\top} x), \phi_1^0).
\end{split}
\end{align}
for all $x=[1, \bar{x}] \in \Rbb^{p+1}, y\in \Rbb$. Now, choose a sequence $G_n = \sum_{j=1}^{k} p_j^{n} \delta_{(\beta_{1j}^{n}, \phi_j^0)}$ such that $p_1^{n} = p_1^0 + \dfrac{p_1^0\phi_1^0}{n}, p_2^{n} = p_2^0 - \dfrac{p_1^0\phi_1^0}{n}, p_j^n = p_j^0$ for all $j\geq 3$; $\beta_{10}^{n} = \beta_{10}^{0} + \dfrac{1}{n}$, $\beta_{ji}^{n} = \beta_{ji}^0$ for all $(j, i)\neq (1, 0)$; and $\phi_{j}^{n} = \phi_{j}^{0}$ for all $j$. It can be checked that 
$$W_r^r(G, G_0) \asymp \dfrac{1}{n} + (p_1^0 - \phi_1^0/n)(\beta_{10}^{n} - \beta_{10}^{0})^{r}\asymp \dfrac{1}{n} =: \epsilon_n.$$ 
Meanwhile, using Taylor's expansion up to second order with integral remainder, we have
\begin{align*}
    f_{G_n}(y|x) - f_{G_0}(y|x) & = (p_1^{n} - p_1^0) \NB(y|\exp((\beta_1^0)^{\top} x), \phi_1^0) \\
    & \hspace{-1cm} + (p_2^{n} - p_2^0) \NB(y|\exp((\beta_2^0)^{\top} x), \phi_2^0)  \\
    & \hspace{-1cm} + p_1^0 (\NB(y|\exp((\beta_1^n)^{\top} x, \phi_1^0)) - \NB(y|\exp((\beta_1^0)^{\top} x), \phi_1^0)))\\
    & \hspace{-1cm} = \dfrac{1}{n}p_1^0\phi_1^0 \NB(y|\exp((\beta_1^0)^{\top} x), \phi_1^0) -  \dfrac{1}{n} p_1^0\phi_1^0 \NB(y|\exp((\beta_2^0)^{\top} x), \phi_2^0)\\
    & \hspace{-1cm} +  p_1^0  (\beta_{10}^{n} - \beta_{10}^{0}) \dfrac{\partial}{\partial (\beta_{10}^0)}\NB(y|\exp((\beta_1^0)^{\top} x), \phi_1^0) \\
    & \hspace{-1cm} +  p_1^0  \dfrac{(\beta_{10}^{n} - \beta_{10}^{0})^{2}}{2} \int_0^{1} dt (1-t) \dfrac{\partial^2}{(\partial \beta_{10}^0)^{2}}\NB(y|\mu_1^0 \exp(t\epsilon_n), \phi_1^0)\\
    & \hspace{-1cm} = p_1^0  \dfrac{\epsilon_n^{2}}{2} \int_0^{1} dt (1-t) \dfrac{\partial^2}{(\partial \beta_{10}^0)^{2}}\NB(y|\mu_1^0 \exp(t\epsilon_n), \phi_1^0),
\end{align*}
where the zero and first-order terms are canceled out due to the equation~\eqref{eq:weakly-ident-NB-proof}. Therefore, 
\begin{align*}
d_H^2(f_{G_n}(\cdot|x), f_{G_0}(\cdot|x)) & = \sum_{y=0}^{\infty} (f_{G_n}^{1/2}(y|x) - f_{G_0}^{1/2}(y|x))^2\\
& = \sum_{y=0}^{\infty} \dfrac{(f_{G_n}(y|x) - f_{G_0}(y|x))^2}{(f^{1/2}_{G_n}(y|x) + f^{1/2}_{G_0}(y|x))^2}\\
& \leq \sum_{y=0}^{\infty} \dfrac{(f_{G_n}(y|x) - f_{G_0}(y|x))^2}{(f^{1/2}_{G_0}(y|x))^2}\\
& \leq \sum_{y=0}^{\infty} \dfrac{(f_{G_n}(y|x) - f_{G_0}(y|x))^2}{p_1^0 \NB(y | \mu_1^0, \phi_1^0)}\\
& \leq p_1^0 \dfrac{\epsilon_n^4}{2} \sum_{y=0}^{\infty} \int_0^{1} dt \dfrac{\left((1-t) \frac{\partial^2}{(\partial \beta^0_{10})^2} \NB(y | \mu_1^0\exp(t\epsilon_n), \phi_1^0)\right)^2}{\NB(y | \mu_1^0, \phi_1^0)}\\
& = p_1^0 \dfrac{\epsilon_n^4}{2} \int_0^{1} dt (1-t)^2 \sum_{y=0}^{\infty}  \dfrac{\left(\frac{\partial^2}{(\partial \beta^0_{10})^2} \NB(y | \mu_1^0\exp(t\epsilon_n), \phi_1^0)\right)^2}{\NB(y | \mu_1^0, \phi_1^0)}\\
& \preccurlyeq \epsilon_n^4,
\end{align*}
uniformly in $x$, where the first two inequalities are due to the fact that $\NB$ is non-negative, the third inequality is because of Holder's inequality, the equality after that is because of Fubini theorem, and the last comparison is an application of Lemma~\ref{lem:aux-smoothness-NB} with $\epsilon$ chosen to be $t\epsilon_n$. 
Hence, 
$$\sup_{x} d_H(f_{G_n}(\cdot|x), f_{G_0}(\cdot|x)) = O(W_r^{2r}(G_n, G_0)), $$
as $G_n \xrightarrow{W_r} G_0$.

\textbf{Step 2} After having the limit above, the rest of this proof follows a standard proof technique for minimax lower bound (e.g., see Theorem 4.4. in \cite{Ho-Nguyen-EJS-16}). Indeed, for any sufficient small $\epsilon > 0$, there exist $G_0, G_0'\in \Ecal_{k_0}$ such that $W_r(G_0, G_0') = 2\epsilon$ and $\sup_{x}d_H(f_{G_0}, f_{G_0'}) \leq C\epsilon^{2r}$. Applying Lemma 1 in \cite{yu1997assouad}, we have the following inequality for any sequence of estimator $\hat{G}_n$ in $\Ecal_{k_0}$:
$$\sup_{G\in \{G_0, G_0'\}} \Ebb_{\Pbb_{G}} W_r(\hat{G}_n, G) \geq \epsilon (1 - \Ebb_X^n d_{TV}(f_{G_0}^n, f_{G_0'}^n)).$$
where $\Ebb_{X}^n := \Ebb_{\Pbb_X^n}$ and $f_G^n := \prod_{i=1}^{n}f_{G}(y_i|X_i)$. Besides, we have

\begin{align*}
    d_{TV}(f_{G_0}^{n}, f_{G_0'}^{n}) & \leq d_H(f_{G_0}^{n}, f_{G_0'}^{n}) \\
    & = \sqrt{1- (1-d_H^2(f_{G_0}, f_{G_0'}))^n}\\
    & \leq \sqrt{1 - (1-C^2\epsilon^{4r})^n}
\end{align*}
Selecting $\epsilon = (1/(C^2n))^{4r}$, we have $(1-C^2\epsilon^{4r})^n\to e^{-1}$ so that
$$\sup_{G\in \{G_0, G_0'\}} \Ebb_{\Pbb_{G}} W_r(\hat{G}_n, G) \succcurlyeq \epsilon \asymp 1/n^{4r} .$$
Hence,
$$\inf_{\hat{G}_n\in \Ecal_{k_0}}\sup_{G_0\in \Ecal_{k_0}} \Ebb_{\Pbb_{G}} W_r(\hat{G}_n, G) \succcurlyeq 1/n^{4r}.$$
\end{proof}

\begin{lemma}\label{lem:aux-smoothness-NB}
    Let $\mu = \exp(\beta^{\top} x)$ for $\beta = (\beta_{j})_{j=0}^{p}$ and $x = (1, \overline{x})\in \Rbb^{p+1}$, both range in compact subspaces of $\Rbb^{p+1}$ and $\phi > 0$. There exists $\epsilon_0 > 0$ such that
    $$\sup_x \sup_{\epsilon\in [0, \epsilon_0]} \sum_{y=0}^{\infty} \dfrac{\left(\frac{\partial^2}{(\partial \beta_{0})^2} \NB(y | e^{\epsilon} \mu, \phi)\right)^2}{\NB(y | \mu, \phi)} < \infty.$$
\end{lemma}
\begin{proof}
    Because both $x$ and $\beta$ range in compact sets, we have $\beta^{\top}x$ is bounded away from $\pm \infty$. Therefore, for sufficiently small $\epsilon_0$, we have $q_{\epsilon} := \dfrac{\exp(\beta^{\top}x + \epsilon)}{\exp(\beta^{\top}x + \epsilon) + \phi}$ is bounded away from 0 and 1 for all $\epsilon\in [0, \epsilon_0]$ and $x, \beta 
    $. Denote $c = \inf_{x, \beta, \epsilon} q_{\epsilon} > 0$ and $C = \sup_{x, \beta, \epsilon} q_{\epsilon} < 1$. Direct calculation gives:
    $$\dfrac{\partial}{\partial \beta_0} \NB(y | e^{\epsilon} \mu, \phi) = [q_{\epsilon} (\phi + y) - y] \NB(y | e^{\epsilon} \mu, \phi),$$
    \begin{align*}\dfrac{\partial^2}{(\partial \beta_0)^{2}} \NB(y | e^{\epsilon} \mu, \phi) & = \underbrace{[q_{\epsilon} (1 - q_{\epsilon}) (\phi + y) + (q_{\epsilon} (\phi + y) - y)^2]}_{P_{\epsilon}(y)} \NB(y | e^{\epsilon} \mu, \phi) 
    \end{align*}
    where $P_{\epsilon}(y)$ is a polynomial of the fourth order of $y$, and
    $$\dfrac{\NB(y | e^{\epsilon} \mu, \phi)}{\NB(y | \mu, \phi)} = \left(\dfrac{q_{\epsilon}}{q_0}\right)^{y} \left(\dfrac{1-q_{\epsilon}}{1-q_0}\right)^{\phi}.$$
    Hence,
    \begin{align*}
    \sum_{y=0}^{\infty} \dfrac{\left(\frac{\partial^2}{(\partial \beta_{0})^2} \NB(y | e^{\epsilon} \mu, \phi)\right)^2}{\NB(y | \mu, \phi)} & = \sum_{y} (P_{\epsilon}(y))^2 \left(\dfrac{q_{\epsilon}}{q_0}\right)^{2y} \left(\dfrac{1-q_{\epsilon}}{1-q_0}\right)^{2\phi} \NB(y | \mu, \phi)\\
    & = \left(\dfrac{1-q_{\epsilon}}{1-q_0}\right)^{2\phi} \Ebb_{Y\sim \NB(\mu, \phi)} (P_{\epsilon}(Y))^2 \left(\dfrac{q_{\epsilon}}{q_{0}}\right)^{2Y}.
    \end{align*}
    The first term is easily bounded by the comment on the range of $q_{\epsilon}$ in the beginning. To uniformly bound the expectation in the expression above, we will bound the expectation of $(P_{\epsilon}(Y))^{4}$ and $(q_{\epsilon} / q_0)^{4Y}$ separately, and then an application of Cauchy-Schwarz inequality yields the result. Because $P_{\epsilon}(Y)$ is a polynomial of $Y$ with bounded coefficient, we have $\Ebb (P_{\epsilon}(Y))^{4} < \infty$ uniformly in $x$. For the second term, recall that the moment-generating function of $\Ebb e^{Yt}$ exists and equals $(\frac{q_0}{1-(1 - q_0)e^{t}})^{\phi}$ for all $t < \log(1 / (1-q_0))$. Given an arbitrary $\delta > 0$, we can choose $\epsilon_0$ sufficient small so that $4\log(q_{\epsilon} / q_0) < 1+\delta < \log(1/(1-q_0))$ uniformly in $x$. So that $\Ebb (q_{\epsilon}/q_0)^{4Y} \leq \Ebb e^{(1+\delta)Y} = (\frac{q_0}{1-(1 - q_0)e^{1+\delta}})^{\phi}$, which is also uniformly bounded in $x$. These claims together conclude the lemma.
\end{proof}

\subsection{Strong identifiability for negative binomial regression mixtures}\label{subsec:mean-disperson-NB}
Theorem~\ref{thm:minimax-NB} and its proof indicate that the family of negative binomial distributions does not enjoy first order identifiability in general. However, we shall show that the set of parameter values where first order identitifiability fails to hold has Lebesgue measure zero. In particular, the following holds.
\begin{proposition}\label{prop:strong-ident-NB-1}
    Given $k$ distinct pairs $(\mu_1, \phi_1), \dots, (\mu_k, \phi_k) \in \mathbb{R}_{+}\times \mathbb{R}_{+}$ such that there does not exist two indices $i\neq j$ satisfying 
\begin{equation*}
    \begin{cases}
        \dfrac{\mu_i}{\phi_i} = \dfrac{\mu_j}{\phi_j} \\
        \phi_i = \phi_j + 1,
    \end{cases}
\end{equation*}
then the mixture of negative binomials $(\NB(\mu_i, \phi_i))_{i=1}^{k}$ is strongly identifiable in the first order.
\end{proposition} 
\begin{proof}
We need to prove that if there exist $(a_i, b_i, c_i)_{i=1}^{k}$ such that 
\begin{equation}\label{mean-dispersion-NB}
    \sum_{i=1}^{k} a_i \NB(y|\mu_i, \phi_i) + b_i\dfrac{\partial }{\partial \mu} \NB(y|\mu_i, \phi_i) + c_i\dfrac{\partial }{\partial \phi}  \NB(y|\mu_i, \phi_i) = 0,
\end{equation}
for all $y\in \Nbb$, then $a_i = b_i = c_i = 0 \; \forall \, y = 1,\dots, k.$ To simplify the presentation, we will write the negative binomial in terms of probability-dispersion parameters, i.e., set $q = \mu / (\mu+\phi)$ (and $q_i = \mu_i / (\mu_i + \phi_i)$ for all $i$), then the negative binomial mass function becomes
\begin{equation*}
    f(y|q, \phi) = \dfrac{\Gamma(\phi + y)}{\Gamma(\phi) y!} q^{y} (1-q)^{\phi}, \quad \forall y\in \Nbb. 
\end{equation*}
Under this presentation, we have
\begin{equation*}
    \dfrac{\partial}{\partial \mu} \NB(y|\mu, \phi) = \dfrac{\partial q}{\partial \mu} \times \dfrac{\partial f(y|q, \phi)}{\partial q} = \dfrac{\phi}{(\mu+\phi)^2} \dfrac{\partial f(y|q, \phi)}{\partial q},
\end{equation*}
and
\begin{equation*}
    \dfrac{\partial}{\partial \phi} \NB(y|\mu, \phi) = \dfrac{\partial q}{\partial \phi} \times \dfrac{\partial f(y|q, \phi)}{\partial q} + \dfrac{\partial f(y|q, \phi)}{\partial \phi} = -\dfrac{\mu}{(\mu+\phi)^2} \dfrac{\partial f(y|q, \phi)}{\partial q} + \dfrac{\partial f(y|q, \phi)}{\partial \phi},
\end{equation*}
therefore, we can write Eq.~\eqref{mean-dispersion-NB} as 
\begin{equation}\label{prop-dispersion-NB}
    \sum_{i=1}^{k} \alpha_i \NB(y|q_i, \phi_i) + \beta_i\dfrac{\partial }{\partial q} \NB(y|q_i, \phi_i) + \gamma_i\dfrac{\partial }{\partial \phi}  \NB(y|q_i, \phi_i) = 0,
\end{equation}
where $\alpha_i = a_i, \beta_i = \dfrac{\phi_i}{(\mu_i + \phi_i)^2}b_i -\dfrac{\mu_i}{(\mu_i + \phi_i)^2} c_i$, and $\gamma_i = c_i$. If we can prove that $\alpha_i = \beta_i = \gamma_i = 0$, it immediately follows that $a_i=b_i=c_i=0$ for all $i=1,\dots, k$, and we get the identifiability result. We recall that pairs $(q_1, \phi_1), \dots, (q_k, \phi_k)$ are distinct (implied from the assumption $(\mu_1, \phi_1), \dots, (\mu_k, \phi_k)$ are distinct) and there does not exist indices $i\neq j$ such that
\begin{equation*}
    \begin{cases}
        q_i = q_j \\
        \phi_i = \phi_j + 1.
    \end{cases}
\end{equation*}
We can simplify Eq.~\eqref{prop-dispersion-NB} as
\begin{equation}\label{prop-dispersion-NB-simplified}
     \sum_{i=1}^{k} \left[\alpha_i + \beta_i \left(\dfrac{y}{q_i} - \dfrac{\phi_i}{1-q_i} \right) + \gamma_i\left(f_{y}(\phi_i) + \log(1-q_i) \right) \right] P_{y}(\theta_i) q_i^{y} (1-q_i)^{\phi_i} = 0,
\end{equation}
where
\begin{equation*}
    P_y(\phi) = \dfrac{\Gamma(\phi+y)}{\Gamma(\phi)} = \phi(\phi+1)\dots (\phi + y -1), \quad f_y(\phi) = \sum_{i=0}^{y-1}\dfrac{1}{\phi + i}.
\end{equation*}
The function $P_{y}(\phi)$ has the following properties:
\begin{enumerate}
    \item By Stirling's formula, $P_y(\phi) \asymp \dfrac{1}{\Gamma(\phi)}\sqrt{2\pi(\phi+y-1)} \left(\dfrac{\phi+y-1}{e}\right)^{\phi+y-1}$ as $y\to \infty$; 
    \item $P_{y}(\phi)$ is an increasing polynomial of $\phi$, and $\dfrac{P_y(\phi)}{P_y(\phi') \log(\phi')} \to \infty$ as $y\to \infty$ if $\phi > \phi'$;
    \item For all $\phi, \phi' \in \Rbb_{+}$, $0\leq q < q' \leq 1$, and polynomial $p(y)$ we have $p(y)\dfrac{P_y(\phi)}{P_y(\phi')} \left(\dfrac{q}{q'}\right)^{y}\to 0$ as $y\to \infty$;
    \item $f_y(\phi) \asymp \log(y)$ as $y\to \infty$.
\end{enumerate}
The second and third properties are consequences of the first one. Now, consider the subset of $(q_{1i})_{i=1}^{k_1}$ of $(q_i)_{i=1}^{k}$ which consists of all maximal elements, i.e. $q_{11} = q_{12} = \dots q_{1k_1} = \max_{1\leq i \leq k} q_{i}$. Dividing both sides of Eq.~\eqref{prop-dispersion-NB-simplified} by $q_{11}^{y}$ and letting $y\to \infty$, from the third property above, we obtain:

\begin{align}\label{prop-dispersion-NB-max-q}
     \sum_{i=1}^{k_1} \left[\left(\alpha_{1i} - \beta_{1i} \dfrac{\phi_{1i}}{1-q_{11}} + \gamma_{1i} \log(1-q_{11}) \right) +  \gamma_{1i}f_{y}(\phi_{1i}) + \beta_{1i} \dfrac{y}{q_{11}}  \right] \nonumber \\
     \times P_{y}(\phi_{1i}) (1-q_{11})^{\phi_{1i}} \to 0
\end{align}
as $y\to \infty$. Without loss of generality, assume $\phi_{11} > \phi_{12} > \dots > \phi_{1k_1}> 0$. Because $\phi_{11} \neq \phi_{12} + 1$, consider two cases: $\phi_{11} > \phi_{12}+1$ and $\phi_{11} < \phi_{12}+1$. For the first case, we notice that
\begin{equation*}
    \dfrac{P_y(\phi_{11})}{y P_y(\phi_{1i})} = \dfrac{P_y(\phi_{11})}{ P_y(\phi_{1i} + 1)}\dfrac{\phi_{1i}+y}{y} \to \infty, 
\end{equation*}
as $y\to \infty$. Hence, 
by dividing both sides of Eq.~\eqref{prop-dispersion-NB-max-q} by $P_{y}(\phi_{11})$, we have 
\begin{equation}\label{eq:NB-limit-11}
    \left(\alpha_{11} - \beta_{11} \dfrac{\phi_{1i}}{1-q_{11}} + \gamma_{11} \log(1-q_{11}) \right) + \gamma_{11}f_{y}(\phi_{11}) + \beta_{11} \dfrac{y}{q_{11}} \to 0 \quad (y\to \infty),
\end{equation}
which implies that $\beta_{11} = 0$, followed by $\gamma_{11} = 0$ and $\alpha_{11} = 0$. For the second case where $\phi_{11} < \phi_{12} + 1$. We have $P_y(\phi_{12}) / P_{y}(\phi_{11})\to 0$ and 
$P_y(\phi_{11}) / (y P_{y}(\phi_{12}))\to 0$. 
By dividing both sides of Eq.~\eqref{prop-dispersion-NB-max-q} by $y P_{y}(\phi_{11})$ and let $y\to \infty$, we have $\beta_{11} = 0$. Dividing both sides of Eq.~\eqref{prop-dispersion-NB-max-q} by $yP_{y}(\phi_{11})$, we have $\beta_{12} = 0$. Finally, dividing both sides of Eq.~\eqref{prop-dispersion-NB-max-q} by $P_{y}(\phi_{11})$ and letting $y\to \infty$, we also obtain that the limit~\eqref{eq:NB-limit-11} holds. Hence, in all cases we obtain $\alpha_{11} = \beta_{11} = \gamma_{11} = 0$. Continuing this argument, we have $\alpha_{1i} = \beta_{1i} = \gamma_{1i} = 0$ for all $i=1,\dots, k_1$, then $\alpha_{i} = \beta_{i} = \gamma_{i} = 0$ for all $i=1,\dots, k$.
\end{proof}

\begin{proposition}\label{prop:strong-ident-NB-2}
    Given $k$ distinct pairs $(\mu_1, \phi_1), \dots, (\mu_k, \phi_k) \in \mathbb{R}\times \mathbb{R}_{+}$ such that there does not exist two indices $i\neq j$ satisfying 
\begin{equation*}
    \begin{cases}
        \dfrac{\mu_i}{\phi_i} = \dfrac{\mu_j}{\phi_j} \\
        |\phi_i - \phi_j| \in \{1, 2\},
    \end{cases}
\end{equation*}
then the mixture of negative binomials $(\NB(\mu_i, \phi_i))_{i=1}^{k}$ is strongly identifiable in the second order.
\end{proposition}

\begin{proof}
Using the same transformation $q_i = \mu_i/(\mu_i+\phi_i)$ as in the proof of Proposition~\ref{prop:strong-ident-NB-1}, we only need to prove that if there exist $(a_i, b_i, c_i, d_i, e_i, f_i)_{i=1}^{k}$ such that for all $y\in \Nbb$:
\begin{align}\label{prop-dispersion-NB2}
    \sum_{i=1}^{k} a_i \NB(y|q_i, \phi_i) + b_i\dfrac{\partial }{\partial q} \NB(y|q_i, \phi_i) + c_i\dfrac{\partial }{\partial \phi}  \NB(y|q_i, \phi_i) \nonumber\\
    d_i\dfrac{\partial^2 }{\partial q^2} \NB(y|q_i, \phi_i) + e_i\dfrac{\partial^2 }{\partial q \partial \phi} \NB(y|q_i, \phi_i) + f_i\dfrac{\partial^2 }{\partial \phi^2} \NB(y|q_i, \phi_i) = 0,
\end{align}
then $a_i = b_i = c_i = d_i = e_i = f_i = 0$ for all $i = 1,\dots, k$. We recall that pairs $(q_1, \phi_1), \dots, (q_k, \phi_k)$ are distinct (implied from the assumption $(\mu_1, \phi_1), \dots, (\mu_k, \phi_k)$ are distinct) and there does not exist indices $i\neq j$ such that
\begin{equation*}
    \begin{cases}
        q_i = q_j \\
        |\phi_i- \phi_j| \in \{1, 2\}.
    \end{cases}
\end{equation*}
Taking all the derivatives and rewrite Eq.~\eqref{prop-dispersion-NB2} as
\begin{align}\label{prop-dispersion-NB2-simplified}
     \sum_{i=1}^{k} \bigg[\dfrac{d_i}{q_i^2} y^2 + \left(2d_i \dfrac{\phi_i}{q_i(1-q_i)} - \dfrac{d_i}{q_i^2} + \dfrac{b_i}{q_i}\right)y + \dfrac{e_i}{q_i} f_y(\phi_i) y + f_i f_y^2(\phi_i) \nonumber \\ 
     + \left(2f_i \log(1-q_i) + c_i + e_i \dfrac{\phi_i}{1-q_i} \right) f_y(\theta_i) + C_i(y) \bigg] P_{y}(\phi_i) q_i^{y} (1-q_i)^{\phi_i} = 0,
\end{align}
where
\begin{equation*}
    P_y(\phi) = \dfrac{\Gamma(\phi+y)}{\Gamma(\phi)} = \phi(\phi+1)\dots (\phi + y -1), \quad f_y(\phi) = \sum_{i=0}^{y-1}\dfrac{1}{\phi + i}, 
\end{equation*}
and 
\begin{align*}
    C_i(y) & = a_i + b_i\dfrac{\phi_i}{1-q_i} + c_i \log(1-q_i) + d_i\dfrac{\phi_i^2+\phi_i}{(1-q_i)^2} \\ 
    & + e_i\left( \dfrac{\phi_i}{1-q_i} \log(q_i) + \dfrac{1}{1-q_i}\right) + f_i(\log(1-q_i)^2 + f_y'(\phi_i)).
\end{align*}
Recall some facts as follows:
\begin{enumerate}
    \item By Stirling's formula, $P_y(\phi) \asymp \dfrac{1}{\Gamma(\phi)}\sqrt{2\pi(\phi+y-1)} \left(\dfrac{\phi+y-1}{e}\right)^{\phi+y-1}$ as $y\to \infty$; 
    \item $P_{y}(\phi)$ is an increasing polynomial of $\phi$, and $\dfrac{P_y(\phi)}{P_y(\phi')\log(\phi')} \to \infty$ as $y\to \infty$ if $\phi > \phi'$;
    \item For all $\phi, \phi' \in \Rbb_{+}$, $0\leq q < q' \leq 1$, and polynomial $p(y)$ we have $p(y)\dfrac{P_y(\phi)}{P_y(\phi')} \left(\dfrac{q}{q'}\right)^{y}\to 0$ as $y\to \infty$;
    \item $f_y(\phi) \asymp \log(y)$ as $y\to \infty$;
    \item $f_y'(\phi_i) = -\sum_{j=0}^{y} \dfrac{1}{(\phi_i + j)^2} \in [-\pi^2/6 - 1/\phi_i^2, 0]$ for all $\phi_i>0, y\in \Nbb$.
\end{enumerate}
Now, consider the subset of $(q_{1i})_{i=1}^{k_1}$ of $(q_i)_{i=1}^{k}$ which consists of all maximal elements, i.e. $q_{11} = q_{12} = \dots q_{1k_1} = \max_{1\leq i \leq k} q_{i}$, then by dividing both sides of Eq.~\eqref{prop-dispersion-NB2-simplified} by $q_{11}^{y}$ and let $y\to \infty$, from the third property above, we obtain:
\begin{align}\label{prop-dispersion-NB2-max-q}
     \sum_{i=1}^{k_1} \bigg[\dfrac{d_{1i}}{q_{1}^2} y^2 + \left(2d_{1i} \dfrac{\phi_{1i}}{q_1(1-q_1)} - \dfrac{d_{1i}}{q_1^2} + \dfrac{b_{1i}}{q_1}\right)y + \dfrac{e_{1i}}{q_1} f_y(\phi_{1i}) y + f_{1i} f_y^2(\phi_{1i}) \nonumber \\ 
     + \left(2f_{1i} \log(1-q_1) + c_{1i} + e_{1i} \dfrac{\phi_{1i}}{1-q_1} \right) f_y(\phi_{1i}) + C_i(y) \bigg] P_{y}(\phi_{1i}) (1-q_1)^{\phi_{1i}} = 0,
\end{align}
as $y\to \infty$. Without loss of generality, assume $\phi_{11} > \phi_{12} > \dots > \phi_{1k_1}> 0$. Because $|\phi_{11}- \phi_{12}| \neq 1, 2$, there are three cases: 
$$\begin{cases}
    \phi_{11} > \phi_{12} + 2,\\
    \phi_{12} + 2 > \phi_{11} > \phi_{12} + 1,\\
    \phi_{12} + 1 > \phi_{11} > \phi_{12}.
\end{cases}$$
For the first case, note that
\begin{equation*}
    \dfrac{P_y(\phi_{11})}{y^2 P_y(\phi_{1i})} = \dfrac{P_y(\phi_{11})}{P_y(\phi_{1i} + 2)}\dfrac{(\phi_{1i}+y+1)(\phi_{1i}+y)}{y^2} \to \infty, 
\end{equation*}
as $y\to \infty$. Hence, 
by dividing both sides of Eq.~\eqref{prop-dispersion-NB2-max-q} by $P_{y}(\phi_{11})$, we have 
\begin{align}\label{eq:NB2-limit-11}
    \dfrac{d_{11}}{q_{1}^2} y^2 + \left(2d_{11} \dfrac{\phi_{1i}}{q_1(1-q_1)} - \dfrac{d_{11}}{q_1^2} + \dfrac{b_{11}}{q_1}\right)y + \dfrac{e_{11}}{q_1} f_y(\phi_{11}) y + f_{11} f_y^2(\phi_{11})\nonumber\\
    +\left(2f_{11} \log(1-q_1) + c_{11} + e_{11} \dfrac{\phi_{11}}{1-q_1} \right) f_y(\phi_{11}) + C_{11}(y) \to 0 \quad (y\to \infty),
\end{align}
which, by considering the order of $y$, implies that $d_{11} = b_{11} = e_{11} = f_{11} = c_{11} = a_{11} = 0$, respectively. For the second case, where $\phi_{12} + 1\phi_{11} < \phi_{12} + 2$, we have 
$P_y(\phi_{11}) / (y^2 P_{y}(\phi_{12}))\to 0$ and $P_y(\phi_{12}) / (y P_{y}(\phi_{11}))\to 0$. 
By dividing both sides of Eq.~\eqref{prop-dispersion-NB2-max-q} by $ y^2P_{y}(\phi_{11})$ and let $y\to \infty$, we have $b_{11} = 0$. Then dividing two sides by $y f_{y}(\phi_{11}) P_{y}(\phi_{11})$ and $y P_{y}(\phi_{11})$, we have $d_{11} = e_{11} = 0$. Continuing in the same way with $y^2 P_{y}(\phi_{11})$, we have $b_{12} =0$. Now dividing both sides of Eq.~\eqref{prop-dispersion-NB2-max-q} by $P_{y}(\phi_{11})$ and letting $y\to \infty$, we obtain that the limit~\eqref{eq:NB-limit-11} holds, which once again entails that $ e_{11} = f_{11} = c_{11} = a_{11} = 0$. In the final case, dividing both sides of Eq.~\eqref{prop-dispersion-NB2-max-q} by $y^2P_{y}(\phi_{11}), y^2 P_{y}(\phi_{12}), y f_{y}(\phi_{11}) P_{y}(\phi_{11}), y P_{y}(\phi_{11}), y f_{y}(\phi_{12}) P_{y}(\phi_{12})$, and $y P_{y}(\phi_{12})$, respectively, we arrive at the same conclusion.
Hence, in all cases, we have $d_{11} = b_{11} = e_{11} = f_{11} = c_{11} = a_{11} = 0$. Applying repreatedly this argument, we have $a_{1i} = b_{1i} = c_{1i} = d_{1i} = e_{1i} = f_{1i} = 0$ for all $i=1,\dots, k_1$, then $a_{i} = b_{i} = c_{i} = d_{i} = e_{i} = f_{1i} = 0$ for all $i=1,\dots, k$.
\end{proof}

\paragraph{Implication in negative binomial regression mixtures.} From the argument above, we can see that the family of 
binomial regression mixture model is strongly identifiable in the first order if we adjust the assumption (A4.) as follows:
\begin{itemize}
\item[(A4'.)] For every set of $k+1$ distinct elements $(\theta_{11}, \theta_{21}),..., (\theta_{1(k+1)}, \theta_{2(k+1)})\in \Theta_1 \times \Theta_2$, there exists a subset $A \subset \mathcal{X}$, $\mathbb{P}_X(A)>0$ such that $$\left(h_1(x,\theta_{11}), h_2(x,\theta_{21})\right),..., \left(h_1(x,\theta_{1(k+1)}), h_2(x,\theta_{2(k+1)})\right)$$ 
are distinct and $|h_2(x,\theta_{2i}) - h_2(x,\theta_{2j})|\neq 1 \,(\forall\, i, j)$ for every $x \in A$.
\end{itemize}
Similarly, the second-order strong identifiability condition is satisfied if we adjust assumption (A4.) as:
\begin{itemize}
\item[(A4''.)] For every set of $k+1$ distinct elements $(\theta_{11}, \theta_{21}),..., (\theta_{1(k+1)}, \theta_{2(k+1)})\in \Theta_1 \times \Theta_2$, there exists a subset $A \subset \mathcal{X}$, $\mathbb{P}_X(A)>0$ such that $$\left(h_1(x,\theta_{11}), h_2(x,\theta_{21})\right),..., \left(h_1(x,\theta_{1(k+1)}), h_2(x,\theta_{2(k+1)})\right)$$  
are distinct and $|h_2(x,\theta_{2i}) - h_2(x,\theta_{2j})|\not \in \{1, 2\} \,(\forall\, i, j)$ for every $x \in A$.
\end{itemize}



\section{Convergence rates for conditional densities via MLE}\label{sec:conv-rate-m-estimator-theory}
We present in this section a proof of Theorem \ref{thm:density_estimation_rate}, which provides general convergence rates of conditional densities estimation.
The proof technique follows a general framework of M-estimation theory \cite{Vandegeer,gine2021mathematical}, with a suitable adaptation for handling conditional density functions. Assume that we have $n$ i.i.d. observations $(x_1, y_1), \dots, (x_n, y_n)$, where $x_i \overset{i.i.d.}{\sim} \Pbb_X$ and $y_i | x_i \sim f_{0}(y|x), i = 1,\dots, n$, for $f_0 \in \Fcal$ being some family of conditional densities of $y$ given $x$ (commonly dominated by $\nu$). Assume that there exists
\begin{equation*}
    \widehat{f}_n \in \argmax_{f \in \Fcal} \sum_{i=1}^{n}\log f(y_i|x_i),
\end{equation*}
Set
\begin{equation*}
    \overline{\Fcal} = \{((f + f_0)/2) : f \in \Fcal \}, \quad 
    \overline{\Fcal}^{1/2} = \{ \overline{f}^{1/2} : \overline{f} \in \overline{\Fcal} \},
\end{equation*}
and denote the (expected) Hellinger ball centered around $f_{0}$ by
\begin{equation*}
    \overline{\Fcal}^{1/2}(\delta) = \{\overline{f}^{1/2}\in \overline{\Fcal}^{1/2}(\Theta) : \overline{d}_{H}(\overline{f}, f_{0})\leq \delta\}.
\end{equation*}
The size of this set is characterized by the bracket entropy integral
\begin{equation}
    \mathcal{J}(\delta) := \int_{\delta^2/2^{13}}^{\delta} H_B^{1/2}(u, \overline{\Fcal}^{1/2}(\delta), L_2(\Pbb_X\times \nu)) du \vee \delta, 
\end{equation}
where $H_B(u, \Fcal, L_2(\Pbb_X\times \nu)) = \log N_B(u, \Fcal, L_2(\Pbb_X\times \nu))$, and $N_B(u, \Fcal, L_2(\Pbb_X\times \nu))$ is the minimal number of pairs $(f_{j}^{L}, f_{j}^{U})_{j}$ such that for every $f\in \Fcal$, there exists $j$ to have $f_{j}^{L}\leq f\leq f_{j}^{U}$ and $\norm{ f_{j}^{U}-f_{j}^{L}}_{L^2}\leq \epsilon$. 
Define $d\Pbb_{0}(x, y) = d\Pbb_X(x) \times f_{0}(y|x) d\nu(y)$ be the true joint distribution of $(x, y)$. Denote by $\Pbb_n$ the empirical distribution of $(x, y)$, i.e.,  $\Pbb_n = \dfrac{1}{n} \sum_{i=1}^{n} \delta_{(x_i, y_i)}$ and $g_{f} = \dfrac{1}{2}\log \dfrac{f + f_{0}}{2 f_{0}} 1(f_{0} > 0)$. For a probability distribution $\Pbb$ and a function $g$, sometimes we write $\Pbb g$ for $\int g d\Pbb$. We start with a basic inequality that links the quality of the conditional density estimate to the associated empirical process:

\begin{lemma}\label{lem:basic-ineq} With the notations defined as above, we have
    \begin{equation}
        \dfrac{1}{2}\overline{d}_{H}^2 \left(\dfrac{\widehat{f}_n + f_{0}}{2}, f_{0}\right) \leq  (\Pbb_n - \Pbb_{0}) g_{\widehat{f}_n} .
    \end{equation}
\end{lemma}
\begin{proof}
    Due to the concavity of logarithm,
    \begin{equation*}
         \dfrac{1}{2}\log \dfrac{f + f_{0}}{2 f_{0}} 1_{(f_{0} > 0)} \geq  \dfrac{1}{2}\log \dfrac{f}{f_{0}} 1_{(f_{0} > 0)}.
    \end{equation*}
  Combining the above with the fact that $\widehat{f}_n$ is the maximum conditional likelihood estimate to obtain
    \begin{align*}
        0 & \leq \int_{f_{0} > 0} \dfrac{1}{4} \log \dfrac{\widehat{f}_n}{f_{0}} \Pbb_n \leq \int_{f_{0}>0} \dfrac{1}{2} \log \dfrac{\widehat{f}_n + f_{0}}{2f_{0}} d\Pbb_n \\
        & = \int_{f_{0} > 0} g_{\widehat{f}_n} d(\Pbb_n - \Pbb_{0}) + \int_{f_{0} > 0} \dfrac{1}{2} \log \dfrac{\widehat{f}_n + f_{0}}{2 f_{0}} d\Pbb_{0}. 
    \end{align*}
  Equivalently,
    \begin{equation*}
        \int_{f_{0} > 0} \dfrac{1}{2} \log \dfrac{2 f_{0}}{\widehat{f}_n + f_{0}} d\Pbb_{0}\leq \int_{f_{0} > 0} g_{\widehat{f}_n} d(\Pbb_n - \Pbb_{0}).
    \end{equation*}
By the inequality $d_H^2(\frac{1}{2} (\widehat{f}_n(\cdot|x) + f_{0}(\cdot|x))), f_{0}(\cdot|x)) \leq K( f_{0}(\cdot|x) \| \frac{1}{2}(\widehat{f}_n(\cdot|x) + f_{0}(\cdot|x))$ for almost all $x$, we can take the expectation with respect to $\Pbb_X$ to arrive at
    \begin{equation*}
        \dfrac{1}{2}\overline{d}^2_{H} \left(\dfrac{\widehat{f}_n + f_{0}}{2}, f_{0}\right) \leq \int_{f_{0} > 0} \dfrac{1}{2} \log \dfrac{2 f_{0}}{\widehat{f}_n + f_{0}} d\Pbb_{0} \leq \int g_{\widehat{f}_n} d(\Pbb_n - \Pbb_{0}).
    \end{equation*}
\end{proof}
For each $f\in \Fcal$, define the squared "Bernstein norm":
\begin{equation}
    \rho^2(f) := 2 \Pbb_0(e^{|f(X)|} - |f(X)| - 1).
\end{equation}
Let $H_{B1}(\epsilon, \Fcal, \Pbb_X\times \nu)$ be the bracketing number with respect to Bernstein norm of $\Fcal$ (cf. Definition 3.5.20. in \cite{gine2021mathematical}).
We shall make use of the concentration behavior of empirical processes associated with the class $\Fcal$ by the following theorem, which is essentially Theorem 3.5.21. in \cite{gine2021mathematical} adapted to our setting.

\begin{theorem}\label{thm:bound-tail-uniform}
    Let $\Fcal$ be a class of measurable functions such that $\rho(f) \leq R$ for all $f\in \Fcal$. Given $C_1 < \infty$, for all $C$ sufficiently large and $C_0$ satisfying
    \begin{equation}
        C_0^2 \geq C^2(C+1),
    \end{equation}
    and for $n\in \mathbb{N}$ and $t>0$ satisfying
    \begin{equation}
        C_0\left(R \vee \int_{t/(2^6\sqrt{n})}^{R} \sqrt{H_{B1}(\epsilon, \Fcal, \Pbb_X\times \nu)} \right) \leq t \leq \sqrt{n}((8R) \wedge (C_1 R^2 / K)),
    \end{equation}
    we have
    \begin{equation}
        \Pbb_0 (\sqrt{n} \sup_{f\in \Fcal} |(\Pbb_n - \Pbb_0)f|\geq t) \leq C \exp\left(-\dfrac{t^2}{C^2(C_1+1)R^2}\right).
    \end{equation}
\end{theorem}

Now we are ready to prove Theorem~\ref{thm:density_estimation_rate}.
\begin{proof}[Proof of Theorem~\ref{thm:density_estimation_rate}] 
We have 
\begin{align*}
    & \Pbb_0(\overline{d}_H(\widehat{f}_n, f_{0})\geq \delta)\\
    & \leq \Pbb\bigg(\sqrt{n} (\Pbb_n - \Pbb_{0})(g_{\widehat{f}_n}) - \sqrt{n} \overline{d}_H^2\left(\dfrac{\widehat{f}_n + f_0}{2}, f_0 \right)\geq 0, \overline{d}_H^2\left(\dfrac{\widehat{f}_n + f_0}{2}, f_0 \right) \geq \delta^2 / C \bigg)\\
    & \leq \Pbb_0\left(\sup_{f:\overline{d}_H^2(\overline{f}, f_{0})\geq \delta^2/C} [\sqrt{n} (\Pbb_{n} - \Pbb_{0})(g_{f}) - \sqrt{n} \overline{d}_H^2 (\overline{f}, f_{0})] \geq 0 \right)\\
    & \leq \sum_{s=0}^{S} \Pbb_0\left(\sup_{f: 2^{s}\delta^2/C\leq \overline{d}_H^2(\overline{f}, f_{0})\leq 2^{s+1}\delta^2/C} |\sqrt{n} (\Pbb_{n} - \Pbb_{0})(g_{f})| \geq \sqrt{n}2^{s} \delta^2/C \right)\\
    &\leq \sum_{s=0}^{S} \Pbb_0\left(\sup_{f: \overline{d}_H^2(\overline{f}, f_{0})\leq 2^{s+1}\delta^2/C} |\sqrt{n} (\Pbb_{n} - \Pbb_{0})(g_{f})| \geq \sqrt{n}2^{s} \delta^2/C \right),
\end{align*}
where $S$ is a smallest number such that $2^S \delta^2 /C > 1$, as $\overline{d}_H(\overline{f}, f_0) \leq 1$. Now we will bound the tail probability of the empirical process 
\begin{equation*}
    \Pbb_0\left(\sup_{f: \overline{d}_H^2(\overline{f}, f_{0})\leq 2^{s+1}\delta^2/C} |\sqrt{n} (\Pbb_{n} - \Pbb_{0})(g_{f})| \geq \sqrt{n}2^{s} \delta^2/C \right)
\end{equation*}
by using Theorem~\ref{thm:bound-tail-uniform}. Indeed, since $p(x) = \dfrac{(e^{|x|} - |x| - 1)}{(e^{x}-1)^2}$ is a decreasing function, and $g_{f}\geq -(\log 2)/ 2$ for all $f$, 
\begin{equation*}
    \exp(|g_f|) - |g_f| - 1 \leq p(-(\log 2)/ 2) (\exp(g_f) - 1)^2 \leq  \left(\sqrt{\dfrac{f_0+f}{2f_0}} - 1 \right)^2.
\end{equation*}
Taking expectation with respect to $\Pbb_0$ both sides to obtain
\begin{equation*}
    \rho^2(g_f) \leq  2\overline{d}_H^2(\overline{f}, f_0)\leq  2^{s+2} \delta/C.
\end{equation*}
Applying Theorem~\ref{thm:bound-tail-uniform} with $R = 2^{s/2+1}\delta/C^{1/2}, t = \sqrt{n}2^{s}\delta^2/C$, we obtain
\begin{equation*}
    \Pbb_0\left(\sup_{f: \overline{d}_H^2(\overline{f}, f_{0})\leq 2^{s+1}\delta^2/C} |\sqrt{n} (\Pbb_{n} - \Pbb_{0})(g_{f})| \geq \sqrt{n}2^{s} \delta^2/C \right) \leq C' \exp\left(-\dfrac{2^{2s}n \delta^2}{C'} \right).
\end{equation*}
Hence,
\begin{equation*}
    \Pbb_0(\overline{d}_H(\widehat{f}_n, f_{0})\geq \delta) \leq \sum_{s=0}^{S} C' \exp\left(-\dfrac{2^{2s}n \delta^2}{C'} \right) \leq c \exp(-n\delta^2 / c).
\end{equation*}
\end{proof}

\section{Convergence rates of conditional densities via Bayesian estimation}\label{sec:posterior-contraction-rate-theory}

We present in this section a general theorem for the Bayesian posterior contraction behavior of conditional density functions that arise in the regression problem. 
The proof technique follows a general approach of Bayesian estimation theory \cite{ghosal2017fundamentals}, with a suitable adaptation for handling conditional density functions. Let us recall the setup.
Given i.i.d. pairs $(x_1, y_1), (x_2, y_2), \dots, (x_n, y_n)$ in $\Xcal \times \Ycal$ from the true generating model
\begin{align*}
    y_i | x_i & \sim f_0(y|x), \\
    x_i & \sim \Pbb_X.
\end{align*}
Here, $\Pbb_X$ is some unknown distribution of covariate $X$, while $f_0$ is assumed to belong to a family of conditional probability functions 
$\{f(y|x): f\in \Fcal\}$, which are absolutely continuous with respect to a common dominating $\sigma$-finite measure $\nu$. To make inference of $f_0$ from the data using the Bayesian approach, we assume
\begin{align*}
    y_i | x_i, f & \sim f(y|x), \\
    f & \sim \Pi, 
\end{align*}
for some prior distribution $\Pi$ on the space of conditional probability functions 
$\Fcal$. The posterior distribution of $f$ is given by, for any measurable subset $B \subset \Fcal$,
\begin{equation*}
    \Pi(f\in B|(x_i, y_i)_{i=1}^{n}) = \dfrac{\int_B \prod_{i=1}^{n} f(y_i|x_i) d\Pi(f)}{\int_{\Fcal} \prod_{i=1}^{n} f(y_i|x_i) d\Pi(f)}.
\end{equation*}

As in the MLE analysis, the posterior contraction behavior of $f$ will be assessed by the expected (squared) Hellinger distance $\overline{d}_{H}(f, f_{0}) = (\Ebb_X d_H^2(f(y|X), f_{0}(y|X)))^{1/2}$. That is, we will find a sequence 
$(\epsilon_n) \to 0$ such that
\begin{equation}
    \Pi(\overline{d}_{H}(f(y|X), f_{0}(y|X)) \geq M_n \epsilon_n | x_1, \dots, x_n, y_1, \dots, y_n) \to 0, 
\end{equation}
in $\otimes_{i=1}^{n} \Pbb_{0}$-probability, as $n\to \infty$. Here, $M_n$ is an arbitrary diverging sequence.

Recall the following basic fact (cf. \cite{ghosal2017fundamentals} Lemma D.2., or \cite{gine2021mathematical} Chapter 7).

\begin{lemma}\label{lem:basic-construct-test}
    Given arbitrary probability density $p$ and $q$, there exist probability densities $\overline{p}$ and $\overline{q}$ such that for any probability density $r$ (all are commonly dominated by $\nu$)
    \begin{equation*}
        \Ebb_{y\sim r} \sqrt{\dfrac{\overline{q}(y)}{\overline{p}(y)}}\leq 1 - \dfrac{1}{6} d_H^2(p, q) + d_H^2(p, r), \quad\quad \Ebb_{y\sim r} \sqrt{\dfrac{\overline{p}(y)}{\overline{q}(y)}}\leq 1 - \dfrac{1}{6} d_H^2(p, q) + d_H^2(q, r)
    \end{equation*}
\end{lemma}
From now, for every conditional density $f(y|x)$, denote by $\Pbb_{f}$ the joint distribution of $x, y$, i.e., $d\Pbb_{f}(x, y) = f(y|x)d\nu(y) \times d\Pbb_X(x)$.
Using Lemma~\ref{lem:basic-construct-test} one arrives at the following result on the existence of tests for conditional density functions:

\begin{lemma}\label{lem:basic-testing}
    For any two conditional density functions $f_0, f_1$ such that $\overline{d}^2_{H}(f_0, f_1) = \epsilon^2$, there exists a test $\Psi_n$ based on $x_1, \dots, x_n, y_1, \dots, y_n$ such that
    \begin{equation}\label{eq:basic-testing}
        \Pbb_{f_0}^n \Psi_n \leq e^{-n\epsilon^2/6}, \quad\quad \sup_{f\in B(f_1, \epsilon/4)} \Pbb_{f}^n (1-\Psi_n) \leq e^{-n\epsilon^2/12},
    \end{equation}
    where $B(f, \epsilon) := \{g\in \Fcal: \overline{d}_{H}(f, g) \leq \epsilon\}$ for all $\epsilon \geq 0$ and $f\in \Fcal$.
\end{lemma}
\begin{proof}
For any $x\in \Xcal$, consider probability density functions $f_0(\cdot | x)$ and $f_1(\cdot | x)$. By Lemma~\ref{lem:basic-construct-test}, there exist density functions $\overline{f}_0(\cdot; x)$ and $\overline{f}_1(\cdot; x)$ such that for all probability density functions $f(\cdot | x)$
\begin{equation*}
    \Ebb_{y\sim f_0(\cdot|x)} \sqrt{\dfrac{\overline{f}_1(y|x)}{\overline{f}_0(y|x)}}\leq 1 - \dfrac{1}{6} d_H^2(f_0(\cdot|x), f_1(\cdot|x)), 
\end{equation*}
and 
\begin{equation*}
     \Ebb_{y\sim f(\cdot | x)} \sqrt{\dfrac{\overline{f}_0(y|x)}{\overline{f}_1(y|x)}}\leq 1 - \dfrac{1}{6} d_H^2(f_0(\cdot|x), f_1(\cdot|x))+ d_H^2(f(\cdot|x), f_1(\cdot|x)).
\end{equation*}
Define test function $\Psi_n(x_1, \dots, x_n, y_1, \dots, y_n) = 1\left(\prod_{i=1}^{n} \dfrac{\overline{f}_1(y_i; x_i)}{\overline{f}_0(y_i; x_i)}\geq 1\right)$. 
To verify~\eqref{eq:basic-testing}, by Markov's inequality,
\begin{align*}
    \Pbb_{f_0}^n \Psi_n & \leq \Pbb_{f_0}^n \prod_{i=1}^{n} \sqrt{\dfrac{\overline{f}_1(y_i|x_i)}{\overline{f}_0(y_i|x_i)}}\\
    & = \prod_{i=1}^{n} \Ebb_X \Ebb_{Y\sim f_0(\cdot|X)} \sqrt{\dfrac{\overline{f}_1(Y|X)}{\overline{f}_0(Y|X)}}\\
    & \leq \left(1 - \dfrac{1}{6} \Ebb_X d_H^2(f_0(\cdot|X), f_1(\cdot|X))\right)^n \\
    & = (1-\epsilon^2/6)^n \\
    & \leq  e^{-n\epsilon^2/6},
\end{align*}
and for every $f\in B(f_1, \epsilon/4)$,
\begin{align*}
    \Pbb_{f}^n \Psi_n & \leq \Pbb_{f}^n \prod_{i=1}^{n} \sqrt{\dfrac{\overline{f}_0(y_i|x_i)}{\overline{f}_1(y_i|x_i)}}\\
    & = \prod_{i=1}^{n} \Ebb_X \Ebb_{Y\sim f_0(\cdot|X)} \sqrt{\dfrac{\overline{f}_0(Y|X)}{\overline{f}_1(Y|X)}}\\
    & \leq \left(1 - \dfrac{1}{6} \Ebb_X d_H^2(f_0(\cdot|X), f_1(\cdot|X))+ \Ebb_X d_H^2(f(\cdot|X), f_1(\cdot|X))\right)^n\\ & \leq \left(1-\dfrac{1}{12}\epsilon^2\right)^n \\
    & \leq e^{-n\epsilon^2/12}.
\end{align*}
\end{proof}

The above lemma shows the existence of tests to distinguish between $f_0$ and a small ball around any $f_1\neq f_0$. 
Next, we establish the existence of tests for $f_0$ against all $f\in \Fcal$ being a bounded distance away from $f_0$. Recall that $N(\Fcal, d, \epsilon)$ denotes the covering number of $\Fcal$ by $d$-balls with radius $\epsilon$.
\begin{lemma}\label{lem:Bayes-testing-complement}
    For every natural number $M$ large enough, there exists a test $\Psi_n$ such that
    \begin{equation}\label{eq:Bayes-testing-complement}
        \Pbb_{f_0}^n \Psi_n \leq N(\Fcal, d_H, \epsilon) \dfrac{e^{-nM^2\epsilon^2/12}}{1-e^{-nM^2\epsilon^2/12}}, \quad\quad \sup_{f\in \Fcal: \overline{d}_{H}(f, f_0)>M\epsilon} \Pbb_{f}^n (1-\Psi_n) \leq e^{-nM^2\epsilon^2/12}.
    \end{equation}
\end{lemma}
\begin{proof}
    For every $j\in \mathbb{N}$ such that $j\geq M$, consider a minimal covering of the set $\Fcal_{j} := \{f\in \Fcal : j\epsilon < \overline{d}_{H}(f, f_0) < 2j\epsilon \}$ by balls $(F_{j,l})_{l}$ of radius $j\epsilon/4$. Because $\Fcal_j \subset \Fcal$ and $j\epsilon/4 \geq \epsilon$, the number of such balls is no more than $N(\Fcal, d_H, \epsilon)$. Moreover, by Lemma~\ref{lem:basic-testing} for each $(F_{j,l})_{l}$ there exists a test $\phi_{j, l}$ satisfying
    \begin{equation}\label{eq:basic-testing-global}
        \Pbb_{f_0}^n \phi_{j, l} \leq e^{-nj^2\epsilon^2/6}, \quad\quad \sup_{f\in F_{j, l}} \Pbb_{f}^n (1-\phi_{j, l}) \leq e^{-nj^2\epsilon^2/12}.
    \end{equation}
    Let $\Psi_n := \max_{j \geq M;l} \phi_{j,l}$. Then
    \begin{equation*}
        \Pbb_{f_0}^{n}\Psi_n \leq N(\Fcal, d_H, \epsilon) \sum_{j\geq M}e^{-nj^2\epsilon^2/6}\leq N(\Fcal, d_H, \epsilon)\dfrac{e^{-nM^2\epsilon^2/12}}{1-e^{-nM^2\epsilon^2/12}},
    \end{equation*}
    and 
    \begin{equation*}
        \sup_{f\in \Fcal: \overline{d}_{H}(f, f_0)>M\epsilon} \Pbb_{f}^n (1-\Psi_n)
        \leq \sup_{j, l}\sup_{f\in F_{j, l}} \Pbb_{f}^n (1-\phi_{j, l})
        \leq e^{-nM^2\epsilon^2/12}.
    \end{equation*}
\end{proof}

Now for every $\epsilon > 0$, define a ball with radius $\epsilon$ around $f_0$ as 
\begin{equation}
B_2(f_0, \epsilon) := \{f\in \Fcal : \Pbb_{f_0} \log(f_0(Y|X) / f(Y|X))\leq \epsilon^2, \Pbb_{f_0} (\log(f_0(Y|X) / f(Y|X)))^2\leq \epsilon^2\}.
\end{equation}
The following theorem establishes the posterior contraction convergence rate for conditional density functions under the expected squared Hellinger distance.
\begin{theorem}\label{thm:basic-contraction-rate}
    Assume that there exist sequences $\overline{\epsilon}_n, \epsilon_n$, such that $\overline{\epsilon}_n\leq \epsilon_n$, and $\sqrt{n} \overline{\epsilon}_n\to \infty$, a sequence of measurable set $\Fcal_n\subset \Fcal$, and a constant $C$ such that
    \begin{enumerate}[label=(\roman*)]
        \item $\Pi(B_2(f_0, \overline{\epsilon}_n))\geq e^{-C n \overline{\epsilon}_n^2}$;
        \item $\log N(\epsilon_n, \Fcal_n, \overline{d}_{H})\leq n\epsilon_n^2$;
        \item $\Pi(\Fcal_n^c)\leq  e^{-(C+4)n\overline{\epsilon}_n^2}$.
    \end{enumerate}
    Then, for every sequence $M_n \rightarrow \infty$, there holds
    \begin{equation}
        \Pi(f: \overline{d}_{H}(f, f_0) > M_n\epsilon_n | x_1, \dots, x_n, y_1, \dots, y_n) \to 0 
    \end{equation}
    in $\Pbb_{f_0}^{n}$-probability, as $n\to \infty$.
\end{theorem}

\begin{proof}
Write $x^{[n]}, y^{[n]} = \{x_1, \dots, x_n, y_1, \dots, y_n\}$ for short. By Lemma~\ref{lem:Bayes-testing-complement}, there exists a test $\Psi_n$ such that
\begin{equation*}
    \Pbb_0^n \Psi_n \leq e^{n\epsilon_n^2} \dfrac{e^{-nM_n^2\epsilon_n^2/12}}{1-e^{-nM_n^2\epsilon_n^2/12}}, \quad \quad \sup_{f\in \Fcal_n : \overline{d}_{H}(f, f_0)\geq M_n \epsilon_n} \Pbb_f^n (1-\Psi_n) \leq  e^{-nM_n^2\epsilon_n^2/12}.
\end{equation*}
As $M_n \to \infty$ and $n\epsilon_n^2\to \infty$, both probabilities above go to 0. By Bayes' rule, 
\begin{equation}
    \Pi(f: \overline{d}_{H}(f, f_0) > M_n\epsilon_n | x^{[n]}, y^{[n]}) = \dfrac{\int_{\overline{d}_{H}(f, f_0) > M_n\epsilon_n} \prod_{i=1}^{n} (f/f_0)(y_i|x_i) d\Pi(f)}{\int_{\Fcal} \prod_{i=1}^{n} (f/f_0)(y_i|x_i) d\Pi(f)}.
\end{equation}
Let $B_n := B_2(f_0, \epsilon_n)$ and $A_n = \{(x^{[n]}, y^{[n]}) : \int_{B_n} \prod_{i=1}^{n} (f/f_0)(y_i|x_i) d\Pi(f) \geq e^{-(C+2)n\overline{\epsilon}_n^2} \}$. Because a probability is always less than or equal to 1, we have
\begin{align*}
    & \Pi(f: \overline{d}_{H}(f, f_0) > M_n\epsilon_n | x^{[n]}, y^{[n]}) \\
    & \leq 1_{(A_n)^c} + 1_{A_n} \dfrac{\int_{\overline{d}_{H}(f, f_0) > M_n\epsilon_n} \prod_{i=1}^{n} (f/f_0) d\Pi(f)}{\int_{\Fcal} \prod_{i=1}^{n} (f/f_0)(y_i|x_i) d\Pi(f)}\\
    & \leq \Psi_n + 1_{(A_n)^c} + 1_{A_n} \dfrac{\int_{\overline{d}_{H}(f, f_0) > M_n\epsilon_n} \prod_{i=1}^{n} (f/f_0) d\Pi(f) (1-\Psi_n)}{\int_{\Fcal} \prod_{i=1}^{n} (f/f_0)(y_i|x_i) d\Pi(f)}\\
    & \leq \Psi_n + 1_{(A_n)^c} + e^{(C+2)n\overline{\epsilon}_n^2} \int_{\overline{d}_{H}(f, f_0) > M_n\epsilon_n} \prod_{i=1}^{n} (f/f_0) d\Pi(f) (1-\Psi_n).
\end{align*}
By the construction of the test $\Psi_n$, we have $\Pbb_0^n \Psi_n \to 0$. Besides, assumption (i) and Lemma~\ref{lem:evidence-lower-bound} imply that
\begin{align*}
    \Pbb_0^{n}(A_n^{c}) & =  \Pbb_0^{n} \left(\int_{B_n} \prod_{i=1}^{n} \frac{f}{f_0}(y_i|x_i) d\Pi(f)\leq e^{-(C+2)n\overline{\epsilon}_n^2}  \right) \\
    & \leq \Pbb_0^{n} \left(\int_{B_n} \prod_{i=1}^{n} \frac{f}{f_0}(y_i|x_i) d\Pi(f)\leq e^{-2n\overline{\epsilon}_n^2} \Pi(B_n) \right)\\
    &\leq \dfrac{1}{n\overline{\epsilon}_n^2}\to 0.
\end{align*}
For the last term, by Fubini's theorem,
{\fontsize{10}{10}\selectfont 
\begin{align*}
    & \Pbb_0^n  e^{(C+2)n\overline{\epsilon}_n^2} \int_{\overline{d}_{H}(f, f_0) > M_n\epsilon_n} \prod_{i=1}^{n} (f/f_0) d\Pi(f) (1-\Psi_n) \\
    & = e^{(C+2)n\overline{\epsilon}_n^2} \int_{\overline{d}_{H}(f, f_0) > M_n\epsilon_n} \Pbb_0^n \prod_{i=1}^{n} (f/f_0)  (1-\Psi_n) d\Pi(f)\\
    & = e^{(C+2)n\overline{\epsilon}_n^2} \int_{\overline{d}_{H}(f, f_0) > M_n\epsilon_n} \Pbb_f^n   (1-\Psi_n) d\Pi(f) \\
    & \leq e^{(C+2)n\overline{\epsilon}_n^2} \bigg( \int_{f\in \Fcal_n: \overline{d}_{H}(f, f_0) > M_n\epsilon_n} \Pbb_f^n   (1-\Psi_n) d\Pi(f)\\
    & \quad + \int_{f\in \Fcal_n^c} \Pbb_f^n (1-\Psi_n) d\Pi(f) \bigg)\\
    & \leq e^{(C+2)n\overline{\epsilon}_n^2} (e^{-nM_n^2\epsilon_n^2/12} + \Pi(\Fcal_n^c)),
\end{align*}}
which tends to 0, thanks to the construction of the test and assumption (iii). 
\end{proof}

The above proof made use of the following lemma, which is taken from \cite{ghosal2017fundamentals} (and adapted for conditional densities). We include its proof for completeness.
\begin{lemma}\label{lem:evidence-lower-bound}
For every $\epsilon > 0$, let $B = B_2(f_0, \epsilon)$. For all $c>0$, we have
\begin{equation}\label{eq:lem-evidence-lower-bound}
    \Pbb_0^{n}\left(\int_{B} \prod_{i=1}^{n}\dfrac{f}{f_0}(y_i|x_i) d\Pi(f) \leq \exp(-(c+1)n\epsilon^2) \Pi(B) \right) \leq \dfrac{1}{c^2n\epsilon^2}.
\end{equation}
\end{lemma}
\begin{proof}
By dividing two sides of the inequality inside $\Pbb_0^n$ by $\Pi(B)$, we can (without loss of generality) assume that $\Pi(B) = 1$. By Jensen's inequality
\begin{equation*}
    \log \int \prod_{i=1}^{n} (f/f_0)(y_i|x_i) d\Pi(f) \geq \sum_{i=1}^{n} \int \log  (f/f_0)(y_i|x_i) d\Pi(f).
\end{equation*}
Hence, for $\Pbb_n$ being the empirical distribution, we have 
\begin{align*}\label{eq:lem-evidence-lower-bound-empirical}
    &\Pbb_0^{n}\left(\int \prod_{i=1}^{n}\dfrac{f}{f_0}(y_i|x_i) d\Pi(f) \leq \exp(-(c+1)n\epsilon^2) \right) \\
    &\leq \Pbb_0^{n}\left(\sum_{i=1}^{n} \int \log  (f/f_0)(y_i|x_i) d\Pi(f) \leq -(c+1)n\epsilon^2 \right) \\ 
    &\leq \Pbb_0^n\left(\sqrt{n}\int\int \log (f/f_0) d\Pi(f) d(\Pbb_n-\Pbb_0) \right. \\
    &\hspace{5cm} \left. \leq -\sqrt{n}(1+c) \epsilon^2 - \sqrt{n} \int\int \log (f/f_0) d\Pi(f) d\Pbb_0 \right)
\end{align*}
By Fubini's theorem and the definition of $B = B_2(f_0, \epsilon)$,
\begin{equation}
    - \sqrt{n} \int\int \log (f/f_0) d\Pi(f) d\Pbb_0 = \sqrt{n} \int \Pbb_0 \log(f_0/f) d\Pi(f) \leq \sqrt{n} \epsilon^2.
\end{equation}
Therefore,
\begin{align*}
     & \Pbb_0^{n}\left(\int \prod_{i=1}^{n}\dfrac{f}{f_0}(y_i|x_i) d\Pi(f) \leq \exp(-(c+1)n\epsilon^2) \right)\\
     & \leq \Pbb_0^n\left(\sqrt{n}\int\int \log (f/f_0) d\Pi(f) d(\Pbb_n-\Pbb_0) \leq \sqrt{n} c \epsilon^2 \right) \\
     & \overset{(*)}{\leq} \dfrac{\text{Var}_{\Pbb_0} (\int \log (f/f_0) d\Pi(f))}{c^2n\epsilon^4}\\
     &  \leq \dfrac{\Pbb_0 (\int \log (f/f_0) d\Pi(f))^2}{c^2n\epsilon^4}\\
     & \overset{(**)}{\leq} \dfrac{\Pbb_0 \int (\log (f/f_0))^2 d\Pi(f)}{c^2n\epsilon^4}\\
     & \leq \dfrac{1}{c^2n\epsilon^2},
\end{align*}
where we apply Chebyshev's inequality in $(*)$ and Jensen's inequality in $(**)$. Hence, inequality \eqref{eq:lem-evidence-lower-bound} is proved.
\end{proof}

\section{Computational details}\label{sec:EM}

\subsection{EM algorithms}
The finite mixture of regression models are based on the idea that the observed data come from a population which can be split into $K$ subpopulations or components. The models are then represented as the form
\begin{align}
    f_G(y_i|x_i,\boldsymbol{\psi})=\sum_{j=1}^{K}p_jf_j(y_i|h_1(x_i, \theta_{1j}),h_2(x_i,\theta_{2j})),
\end{align}
where $y_i$ is the value of the response variable in the $i$th observation; $x_i^{\top} \mbox{ }(i=1,...,n)$ denotes the transpose of $d$-dimensional vector of independent variables for the $i$th observation, $\theta_{1j} \mbox{ and } \theta_{2j} (j=1,...,K)$ denote the $d$-dimensional vectors of coefficients corresponding to the link functions $h_1(x,\theta_{1j})$ and $h_2(x,\theta_{2j})$ of the $j$th component, $p_j$ are the mixing probabilities ($0<p_j<1$, for all $j=1,...,K$ and $\sum_{j=1}^{K}p_j=1)$, $\boldsymbol{\psi}=(p_1,...,p_K,\theta_1,...,\theta_K,\eta_1,...,\eta_K)$ is the complete parameter set of the mixture model.

The parameters of the models can be efficiently estimated through the EM algorithm which is well-known as the standard tool for finding the maximum likelihood solution. The log-likelihood of the model is written as below:
\begin{align}\label{loglikelihood}
    l_0(\boldsymbol{\psi})=L(\boldsymbol{\psi}|x_1,...,x_n,y_1,...,y_n)=\sum_{i=1}^{n}\log \left(\sum_{j=1}^{K}p_jf_j(y_i|h_1(x_i, \theta_{1j}),h_2(x_i,\theta_{2j}))\right).
\end{align}
Let $Z_{ij}$ be an indicator variable which receives value of $1$ if the $i$th observation belongs to the $j$th component, and $0$ otherwise. We can write the log-likelihood function for complete data:
\begin{align}
    l(\boldsymbol{\psi})=\sum_{i=1}^{n}\sum_{j=1}^{K}Z_{ij}\log
    \left(p_jf_j(y_i|h_1(x_i, \theta_{1j}),h_2(x_i,\theta_{2j}))\right).
\end{align}

A description of a generic algorithm can be found in Algorithm \ref{EM}.

\begin{algorithm}
\caption{EM algorithm for mixture of regression models}\label{EM}
 \hspace*{\algorithmicindent} \textbf{Input:} Random initial values of $p_j$, $\theta_{1j}$, for $j=1,...,K$, and the value of $\epsilon>0$; \\
 \hspace*{\algorithmicindent} \textbf{Output:} The maximum likelihood estimation of $p_j$, $\theta_{1j}$ and $\theta_{2j}$, for $j=1,...,K$.\\
\begin{algorithmic}[1]
    \State $t=0$. Choose an initial values for  $\boldsymbol{\psi}_0=\left(p_1^{(0)},...,p_K^{(0)},\theta_1^{(0)},..., \theta_K^{(0)},\eta_1^{(0)},..., \eta_K^{(0)}\right)$. 
    \While{$l_0\left(\boldsymbol{\psi}_{t}\right)- l_0\left(\boldsymbol{\psi}_{t-1}\right)>\epsilon$} 
        \State \textbf{E-step} Calculate the posterior membership probabilities of observation $i$th belonging to the component $j$th as the following 
 $$w_{ij}^{(t-1)}=\dfrac{p_j^{(t-1)}f_j\left(y_i|h_1(x_i,\theta_{1j}^{(t-1)}),h_2(x_i,\theta_{2j}^{(t-1)})\right)}{\sum_{j=1}^{K}p_j^{(t-1)}f_j\left(y_i|h_1(x_i,\theta_{1j}^{(t-1)}),h_2(x_i,\theta_{2j}^{(t-1)})\right)}, \mbox{ } (i=1,...n,j=1,...,K),$$ 
       \State \textbf{M-step} Update $p_j^{(t)}$ and $\theta_{1j}^{(t)}$ by maximizing $Q(\boldsymbol{\psi}|\boldsymbol{\psi}^{(t-1)})=\mathbb{E}_{\mathbf {Z} \mid \mathbf {X} ,{\boldsymbol {\psi}}^{(t-1)}}[l(\boldsymbol{\psi})]$, we have that
       $p_j^{(t+1)}=\frac{1}{n}\sum_{i=1}^{n}w_{ij}, \mbox{ for } j=1,...,K$ ;
       $\theta_{1j}^{(t)}$ satisfies the equations $\dfrac{\partial Q(\boldsymbol{\psi}|\boldsymbol{\psi}^{(t-1)})}{\partial \theta_{1j}}=\mathbf{0}^{\top}$; and $\theta_{2j}^{(t)}$ satisfies the equations $\dfrac{\partial Q(\boldsymbol{\psi}|\boldsymbol{\psi}^{(t-1)})}{\partial \theta_{2j}}=\mathbf{0}^{\top}.$ See Table \ref{table for M step} for more detail in case of some popular GLM models.
       \EndWhile
\end{algorithmic}
\end{algorithm}

For the cases where the equation in M-step does not have a close-form solution (i.e., Poisson, negative binomial, Binomial cases), the EM algorithm in general (GEM) is considered. The idea is that we want to find a way to update the parameters such that the likelihood function \eqref{loglikelihood} will not decrease during the whole process. With this consideration in mind, the EM1 algorithm and gradient ascent algorithm are used. To take account of solving the equation $\dfrac{\partial Q(\boldsymbol{\psi}|\boldsymbol{\psi^{(t-1)}})}{\partial \theta_{mj}}=\mathbf{0}^{\top}$ for $m=1, 2$, we applied Newton-Raphson method (McLachlan and Krishnan, 1997, p. 29), where $Q(\boldsymbol{\psi}|\boldsymbol{\psi}^{(t-1)})=\mathbb{E}_{\mathbf {Z} \mid \mathbf {X} ,{\boldsymbol {\psi}}^{(t-1)}}[l(\boldsymbol{\psi})]$. In 1993, Rai and Matthews \cite{Rai1993} proposed an EM1 algorithm in which Newton-Raphson process is reduced to one iteration, as below
\begin{align}\label{EM1}
    \theta_{mj}^{(t)}= \theta_{mj}^{(t-1)}-\left(\dfrac{\partial^2 Q(\boldsymbol{\psi}^{(t-1)}|\boldsymbol{\psi}^{(t-1)})}{\partial \theta_{mj}\partial \theta_{mj}^{\top}}\right)^{-1}\left(\dfrac{\partial Q(\boldsymbol{\psi}^{(t-1)}|\boldsymbol{\psi}^{(t-1)})}{\partial \theta_{mj}}\right)^{\top} \mbox{ for } m=1,2.
\end{align}
By this way, EM1 algorithm speeds up GEM process in practice because it saves computational cost in the M-step.
Note that to use the EM1 algorithm, both matrices $\left(\dfrac{\partial^2 Q(\boldsymbol{\psi}^{(t-1)}|\boldsymbol{\psi}^{(t-1)})}{\partial \theta_{mj}\partial \theta_{mj}^{\top}}\right)$ and $\left(\dfrac{\partial^2 Q(\boldsymbol{\psi}^{(t)}|\boldsymbol{\psi}^{(t-1)})}{\partial \theta_{mj}\partial \theta_{mj}^{\top}}\right)$ for 
$m=1, 2$ are required to be negative definite. 

In terms of the gradient ascent approach, an update rule is given as follows
\begin{align}\label{gradient ascent}
    \theta_{mj}^{(t)}= \theta_{mj}^{(t-1)}+\nu\left(\dfrac{\partial Q(\boldsymbol{\psi}^{(t-1)}|\boldsymbol{\psi}^{(t-1)})}{\partial \theta_{mj}}\right), \mbox{ for } m = 1,2.
\end{align}
where $\nu$ is a step size (user-specific parameter). We also used one iteration of it to reduce the computational cost.
\begin{table}
    \centering
    \begin{tabular}{|p{1.3cm}|c |p{7.2cm}|}
    \hline
          & $\boldsymbol{f_j\left(y_i|h_1(x_i, \theta_{1j}),h_2(x_i, \theta_{2j})\right)}$& \textbf{Updating} $\boldsymbol{\theta_{1j}^{(t)}}$   \\
        \hline
    Normal &$\Ncal_j(y_i\mid x_i^{\top}\theta_{1j},\sigma^2)$&$\theta_{1j}^{(t)}=(X^{\top} W_j X)^{-1}(X^{\top} W_j Y)$
     \newline
     where $X$ is a matrix of size $n\times d$, $Y$ is a vector of size $n$, $
     W_j=diag(\{w_{ij}\}_{i=1}^n)$ is a matrix of size $n\times n$.\\
         \hline
    Poisson &$\Poi_j\left(y_i|\exp(x_i^{\top}\theta_{1j})\right)$& Using \eqref{EM1} to update $\theta_{1j}^{(t)}$, where
    \newline $\dfrac{\partial Q(\boldsymbol{\psi}|\boldsymbol{\psi}^{(t-1)})}{\partial \theta_{1j}}=\sum_{i=1}^n \left(-\exp\left(x_i^{\top}\theta_{1j}\right)x_i+y_ix_i\right)w_{ij}^{(t-1)}$
    \newline
    $\dfrac{\partial^2 Q(\boldsymbol{\psi}|\boldsymbol{\psi}^{(t-1)})}{\partial \theta_{1j} \partial\theta^{\top}_j}=\sum_{i=1}^n\left(-\exp\left(x_i^{\top}\theta_{1j}\right)\right)x_ix_i^{\top}w_{ij}^{(t-1)}$.\\
    \hline
    Binomial &$\Bin_j\left(y_i|N,\dfrac{1}{1+\exp\left(-x_i^{\top}\theta_{1j}\right)}\right)$& Using $\eqref{EM1}$ to update $\theta_{1j}^{(t)}$, where
    \newline $\dfrac{\partial Q(\boldsymbol{\psi}|\boldsymbol{\psi}^{(t-1)})}{\partial \theta_{1j}}$
    \newline
    $=\sum_{i=1}^{n}x_i w_{ij}^{(t-1)}\left[\dfrac{-N}{1+\exp(-x_i^{\top}\theta_{1j})}+y_i\right]$,
    \newline
    $\dfrac{\partial^2 Q(\boldsymbol{\psi}|\boldsymbol{\psi}^{(t-1)})}{\partial \theta_{1j} \partial\theta^{\top}_j}=\sum_{i=1}^n\dfrac{-N\exp\left(-x_i^{\top}\theta_{1j}\right)}{(1+\exp(-x_i^{\top}\theta_{1j}))
    ^2}x_ix_i^{\top}w_{ij}^{(t-1)}$
    \\
    \hline
    Negative\newline binomial&$\NB_j\left(y_i|\exp\left(x_i^{\top}\theta_{1j}\right),\phi\right)$& Using $\eqref{gradient ascent}$ to update $\theta_{1j}^{(t)}$, where
    \newline $\dfrac{\partial Q(\boldsymbol{\psi}|\boldsymbol{\psi}^{(t-1)})}{\partial \theta_{1j}}=\sum_{i=1}^{n}x_i w_{ij}^{(t-1)}\dfrac{y_i-\exp(x_i^{\top}\theta_{1j})}{1+\frac{\exp(x_i^{\top}\theta_{1j})}{\phi}}$. \\
    \hline
    \end{tabular}
    \caption{Updating of $\theta_{1j}^{(t)}$ of GLM models in the M-step.}\label{table for M step}
  
\end{table}

\subsection{Bayesian approach}\label{appenx: Bayesian approach}
Here we describe in details the derivation of Markov Chain Monte Carlo (MCMC) algorithm that we use in Section~\ref{sec:experiments}. In particular, given a mixture of $k$-negative binomial regression model:
$$f_G(y|x) = \sum_{j=1}^k p_j\NB(y|h(x,\theta_j), \phi_j).$$

As mentioned in Section~\ref{sec:experiments}, given the data, $\{x_i, y_i\}_{i=1}^n$, we performed the prior distributions of $p= (p_1, p_2,...,p_k)$, $\theta_j$ and $\eta_j = \phi_j^{-1} $, for $j=1,..., k$ as the following
\begin{align*}
    p &\sim \Dir(1,1,...,1)\\
    \theta_j &\sim \mathrm{MVN}(0, \Sigma) \text{ (multivariate normal distribution), for } j=1,...,k,\\
    \eta_j &\sim \mathrm{Gamma}(0.01,0.01) \text{ (a non-informative gamma distribution), for } j=1,...,k.
\end{align*}  
The full conditional posterior distributions of the model parameters are given below.
  \begin{align}
    &P(Z_i=j|y,x)=\dfrac{p_j \mathrm{NB}\left(y_i|h(x_i,\theta_j),\eta_j\right)}{\sum_{m=1}^{2} p_m \mathrm{NB}\left(y_i|h(x_i,\theta_m), \eta_m\right)},\\
    &p|(y,x,\mathbf{Z},...) \sim  \Dir(1+n_1,..,,1+n_k), \mbox{ where } n_j=\#\{i:Z_i=j\}; \mbox{ for } j=1,...,k,\\
    &\label{theta}f(\theta_j |y,x,\mathbf{Z},...) \propto \left[ \prod_{i: Z_i = j} \mathrm{NB}\left(y_i|h(x_i,\theta_j),\eta_j\right)\right] \exp\left(-\frac{1}{2}\theta_j'\Sigma^{-1}\theta_j\right),\\
    &\label{eta}g(\eta_j|y,x,\mathbf{Z},...)
    \propto \left[ \prod_{i: Z_i = j} \mathrm{NB}\left(y_i|h(x_i,\theta_j),\eta_j\right)\right] \eta_j^{0.01-1}\exp\left(-0.01 \eta_j\right),
  \end{align}  
where $$\mathrm{NB}\left(y_i|h(x_i,\theta_j),\eta_j\right) = \dfrac{\Gamma(y_i+1/\eta_j)}{\Gamma(y_i+1)\Gamma(1/\eta_j)}\left(\dfrac{\exp(x_i'\theta_j)}{\exp(x_i'\theta_j+1/\eta_j)}\right)^{y_i}
\left(\dfrac{1/\eta_j}{x_i'\theta_j+1/\eta_j}\right)^{1/\eta_j}.$$

The full posterior distribution is sampled by using Gibbs sampling algorithm (Algorithm \ref{Gibbs}). Since the posterior distributions of $\theta_j$ and $\eta_j$ ($j=1,...,k$) are known up to a normalizing constant, the Metropolis-Hasting (MH) algorithm has been used to sample the distribution. When it comes to $\theta_j$, a multivariate normal distribution is used as a proposal density.  In particular, for each $j=1,...,k$, a candidate $\theta_j^* \sim \mathrm{MVN}(\theta_j^{(t-1)},\Sigma')$ is accepted with probability
$$\min\left\{1,\dfrac{f(\theta_j^*|...)}{f(\theta_j^{(t-1)}|...)}\right\}.$$

In terms of $\eta_j$, the proposal density is from a Gamma distribution. Specifically, for each $j= 1,...,k$, a candidate $\eta_j^* \sim \mathrm{Gamma}(2,2/\eta_j^{(t-1)})$ is accepted with probability
$$\min\left\{1,\dfrac{g(\eta_j^*|...)p\left(\eta_j^{(t-1)}|\eta_j^*\right)}{g(\eta_j^{(t-1)}|...)p\left(\eta_j^*|\eta_j^{(t-1)}\right)}\right\},$$
where $f(\theta_j|y,x,\mathbf{Z}^{(t)},p{(t)},\eta^{(t)})$ and $g(\eta_j|y,x,\mathbf{Z}^{(t)},p{(t)},\theta^{(t)})$ are as in Eq. \eqref{theta}, \eqref{eta}, respectively, and $p\left(\eta_j^{(t-1)}|\eta_j^*\right)$ is the gamma density $\mathrm{Gamma} (2,2/\eta_j^{*})$.

For each different sample size $n$, we run the experiment 16 times. For each time of running, we produced 2500 MCMC samples and discarded the first 500 as a “burn-in” set. From among the remaining 2000, we computed the mean of a vector containing 2000 Wasserstein distances ($W_1$) between the MCMC results and the true mixing distribution.

\begin{algorithm}
\caption{Gibbs sampling algorithm}\label{Gibbs}
 \hspace*{\algorithmicindent} \textbf{Input:} The prior distributions of $p$, $\theta_j$ and $\eta_j$, for $j=1,...,k$; \\
 \hspace*{\algorithmicindent} The number of iterations ($T_{max}$) and burn-in steps\\ 
\hspace*{\algorithmicindent} \textbf{Output:} A Markov Chains $\{\boldsymbol{\Phi}_t\}_{t\geq 0}$ attaining posterior distribution of       ($p|(y,x,\mathbf{Z}$,...), $\eta_1 | (y,x,\mathbf{Z},...) $,..., $\eta_k | (y,x,\mathbf{Z},...) $, $\theta_j|(y,x,\mathbf{Z},...)$ (for $j = 1,...,k)$ as the stationary distribution.
\begin{algorithmic}[1]
    \State $t=0$. Draw $\boldsymbol{\Phi}_0=(p^{(0)}$, $\eta_1^{(0)},...,\eta_k^{(0)}, \theta_1^{(0)}, ..., \theta_k^{(0)}$) randomly. 
    \For {$t = 1, 2, \dots, T_{max}$}
        \State Generate $\mathbf{Z}^{(t)} \sim \mathbf{Z}|y,x, p^{(t-1)},\eta_1^{(t-1)},..., \eta_k^{(t-1)}, \theta_1^{(t-1)},..., \theta_k^{(t-1)}$ 
        \State Generate $p^{(t)}\sim p|y,x,\mathbf{Z}^{(t)},\theta_1^{(t-1)},..., \theta_k^{(t-1)}, \eta_1^{(t-1)},..., \eta_k^{(t-1)}$
        \State Generate $\eta_j^{(t)}\sim \eta_j|y, x,\mathbf{Z}^{(t)},p^{(t)}, \theta_1^{(t-1)},..., \theta_k^{(t-1)}$
        \Comment{Using Metropolis-Hasting algorithm}
        \State\label{MH} Generate $\theta_j^{(t)} \sim \theta_j|y,x,\mathbf{Z}^{(t)},p^{(t)},\eta_1^{(t)},...,\eta_k^{(t)}$ \Comment{Using Metropolis-Hasting algorithm}
        \State Set $\boldsymbol{\Phi}_t=(p^{(t)}, \eta_1^{(t)},...,\eta_k^{(t)}, \theta_1^{(t)},..., \theta_k^{(t)})$
    \EndFor
    \end{algorithmic}
\end{algorithm}

\end{document}